\documentclass[11pt,letterpaper]{amsart}

\usepackage{amsmath,amssymb,amsthm,mathtools,microtype,bm}
\usepackage{bbm}
\usepackage[dvipsnames]{xcolor}
\usepackage[a4paper,top=3cm,bottom=2.5cm,left=3cm,right=3cm,marginparwidth=1.75cm]{geometry}

\usepackage{enumitem,needspace}
\setlist[enumerate,1]{label=\textup{(\alph*)},leftmargin=*,itemsep=2pt,beginpenalty=10000}
\usepackage[colorlinks=true,linkcolor=blue,citecolor=blue,urlcolor=blue]{hyperref}
\usepackage{comment}

\newtheorem{theorem}{Theorem}[section]
\newtheorem{conjecture}{Conjecture}[section]
\newtheorem{proposition}[theorem]{Proposition}
\newtheorem{lemma}[theorem]{Lemma}
\newtheorem{corollary}[theorem]{Corollary}
\theoremstyle{definition}
\newtheorem{definition}[theorem]{Definition}

\newtheorem{remark}[theorem]{Remark}
\numberwithin{equation}{section}

\newcommand{\xra}{\xrightarrow}

\newcommand{\sth}{\mid} 
\newcommand{\set}[1]{\left\{ \, #1 \, \right\}}

\newcommand{\D}{\mathrm{D}}
\newcommand{\V}{\mathrm{V}}

\newcommand{\tf}{\mathrm{tf}}

\newcommand{\Db}{\mathrm{D}^{\mathrm{b}}}

\newcommand{\ZZ}{\mathbb{Z}}

\newcommand{\QQ}{\mathbb{Q}}

\newcommand{\CC}{\mathbb{C}}
\newcommand{\RR}{\mathbb{R}}
\newcommand{\PP}{\mathbb{P}}

\newcommand{\DD}{\mathbb{D}}
\newcommand{\bv}{\mathbf{v}}
\newcommand{\bw}{\mathbf{w}}

\newcommand{\DT}{\mathsf{DT}}

\newcommand{\CH}{\mathrm{CH}}

\newcommand{\NS}{\mathrm{NS}}

\newcommand{\ch}{\mathrm{ch}}

\newcommand{\td}{\mathrm{td}}

\newcommand{\KK}{\mathrm{K}}
\newcommand{\Knum}{\mathrm{K}_{\mathrm{num}}}

\newcommand{\T}{\mathrm{T}}

\newcommand{\Del}{\overline{\Delta}}

\renewcommand{\Re}{\operatorname{Re}}
\renewcommand{\Im}{\operatorname{Im}}

\DeclareMathOperator{\rk}{rk}

\DeclareMathOperator{\Coh}{\mathrm{Coh}}

\DeclareMathOperator{\Sym}{Sym}

\DeclareMathOperator{\Hom}{Hom}

\DeclareMathOperator{\Spec}{Spec}
\DeclareMathOperator{\Pic}{Pic}

\DeclareMathOperator{\Stab}{Stab}

\newcommand{\cX}{\mathcal{X}}

\newcommand{\cA}{\mathcal{A}}
\newcommand{\cE}{\mathcal{E}}

\newcommand{\cH}{\mathcal{H}}

\newcommand{\cI}{\mathcal{I}}

\newcommand{\cT}{\mathcal{T}}

\newcommand{\cP}{\mathcal{P}}
\newcommand{\cD}{\mathcal{D}}
\newcommand{\cL}{\mathcal{L}}

\newcommand{\rH}{\mathrm{H}}
\newcommand{\cO}{\mathcal{O}}
\newcommand{\bP}{\mathbb{P}}
\newcommand{\bZ}{\mathbb{Z}}

\newcommand{\kk}{\mathbbm{k}}

\usepackage{xpatch}
\makeatletter   
\xpatchcmd{\@tocline}
{\hfil\hbox to\@pnumwidth{\@tocpagenum{#7}}\par}
{\ifnum#1<0\hfill\else\dotfill\fi\hbox to\@pnumwidth{\@tocpagenum{#7}}\par}
{}{}
\makeatother 

\makeatletter
\renewcommand\part{%
   \if@noskipsec \leavevmode \fi
   \par
   \addvspace{4ex}%
   \@afterindentfalse
   \secdef\@part\@spart}

\def\@part[#1]#2{%
    \ifnum \c@secnumdepth >\m@ne
      \refstepcounter{part}%
      \addcontentsline{toc}{part}{Part \thepart.\hspace{1em}#1}%
    \else
      \addcontentsline{toc}{part}{#1}%
    \fi
    {\parindent \z@ \raggedright
     \interlinepenalty \@M
     \normalfont
     \ifnum \c@secnumdepth >\m@ne
     \centering 
     \Large\bfseries \partname\nobreakspace\thepart     
       \nobreak. 
     \fi
     \Large \bfseries { #2}%
     \par}%
    \nobreak
    \vskip 3ex
    \@afterheading}
\def\@spart#1{%
    {\parindent \z@ \raggedright
     \interlinepenalty \@M
     \normalfont
     \huge \bfseries #1\par}%
     \nobreak
     \vskip 3ex
     \@afterheading}
\makeatother

\renewcommand{\thepart}{\Roman{part}}

\makeatletter 
\def\l@subsection{\@tocline{2}{0pt}{3pc}{6pc}{}} 
\makeatother

\hypersetup{bookmarksdepth=2,
 pdftitle={BMT via Frobenius},
 pdfsubject={Tilt stability, specialization, and Frobenius estimates}}
\title[Generalized BG inequalities via Frobenius] %and Frobenius morphism]
 {Generalized Bogomolov--Gieseker inequalities \\ via Frobenius} 
\date{}
\subjclass[2020]{14F08 (Primary); 14F06, 14J30 (Secondary)}
\keywords{Tilt-stability, Generalized Bogomolov--Gieseker inequalities,
Frobenius morphism, Derived categories}

\AtBeginDocument{\begingroup\fontsize{6}{7}\selectfont\endgroup}

\begin{document}

\author[Chunyi Li]{Chunyi Li}
\address{\parbox{0.9\textwidth}{Mathematics Institute (WMI), University of Warwick\\[1pt]
Coventry, CV4 7AL, United Kingdom
\vspace{1mm}}}
\email{C.Li.25@warwick.ac.uk}
\urladdr{\url{https://sites.google.com/site/chunyili0401/}}

\author[Zhiyu Liu]{Zhiyu Liu}
\address{\parbox{0.9\textwidth}{Department of Mathematics, Princeton University\\[1pt] Princeton, NJ 08544, USA
\vspace{1mm}}}
\email{zl3301@princeton.edu}
\urladdr{\url{https://sites.google.com/view/zhiyuliu}}

\author[Alexander Perry]{Alexander Perry}
\address{\parbox{0.9\textwidth}{Department of Mathematics, University of Michigan\\[1pt]
530 Church Street, Ann Arbor, MI 48109, USA
\vspace{1mm}}}
\email{arper@umich.edu}
\urladdr{\url{http://www-personal.umich.edu/~arper/}}

\author[Xiaolei Zhao]{Xiaolei Zhao}
\address{\parbox{0.9\textwidth}{Department of Mathematics, University of California Santa Barbara\\[1pt]
552 University Rd, Santa Barbara, CA 93117, USA
\vspace{1mm}}}
\email{xlzhao@ucsb.edu}
\urladdr{\url{https://sites.google.com/site/xiaoleizhaoswebsite/}}

\begin{abstract}
We establish a Bogomolov--Gieseker-type inequality for tilt-semistable objects on smooth projective varieties of dimension $\geq 3$, in the spirit of the Bayer--Macr\`i--Toda conjecture. A key ingredient is an analysis of the action of the Frobenius morphism on tilt-semistable objects.

As an application, we prove the Castelnuovo Bound Conjecture, up to an explicit constant term. Combining our results with work of Feyzbakhsh and Thomas, we further show that on a Calabi--Yau threefold,  arbitrary rank Donaldson--Thomas invariants of semistable sheaves are determined by rank one invariants, and hence are symplectic invariants.
\end{abstract}
\maketitle

\setcounter{tocdepth}{1}
\hypersetup{linktoc=page}
\tableofcontents

\section{Introduction}

Motivated by Douglas's $\Pi$-stability \cite{Douglas,DFR}, Bridgeland introduced the notion of stability conditions on   triangulated categories \cite{bridgeland:stability}. Their existence on the derived categories of complex smooth projective varieties has recently been established in \cite{Li}. 
This construction was subsequently extended in \cite{LLL} to the relative setting of flat projective families, where the existence of proper moduli spaces of semistable objects was also shown. These results are based on the construction of well-behaved stability conditions on projective space, which are then restricted to closed subvarieties. 
On the other hand, the traditional approach to the construction of stability conditions on threefolds went through the notion of tilt-stability, a weak stability condition defined in \cite{bridgeland:K3,BMT}. 
For smooth projective threefolds, Bayer--Macr\`i--Toda proposed a conjectural Bogomolov--Gieseker-type inequality \cite{Gieseker,Langer} involving $\ch_3$ of tilt-semistable objects \cite{BMT,BMS}. 
The original statement of this conjecture turned out to be incorrect \cite{Sch17,martinez-schmidt:blowups}, but later a more flexible formulation incorporating a rational 1-cycle $\Gamma$ was proposed in \cite{BMSZ}, 
which we call the \emph{$\Gamma$-BMT Conjecture}. 
When true, these inequalities yield an open subset of the spaces of stability conditions on threefolds \cite{BMS} and 
have important applications in enumerative geometry \cite{Toda:limit,toda:bogomolov-counting,Toda:CYBG,Toda:GepnerBG,feyz:curve-counting,FT2,FT1,fey:modularity,liu-ruan,liu} and birational geometry \cite{BM:projectivity,bbmt14}. 
The original BMT inequality and its $\Gamma$-variant have been proved for various classes of threefolds; see, for example, \cite{macri:projective-space,schmidt:quadric,maciocia-piyaratne:abelian-I,maciocia-piyaratne:abelian-II,BMS,li:fano-pic-rk-1,BMSZ,li:quintic,koseki:product-threefolds, sun:tilt-vanishing, koseki:double-triple-solids,liu:quadratic,FKLR,xu:quintic} and the survey \cite{bm:icm}.

The goal of this paper is to prove a stronger form of the $\Gamma$-BMT Conjecture for any smooth projective variety of dimension $\geq 3$ and to discuss some of its applications.

\subsection{Main theorem}
Let $(X,H)$ be a polarized smooth projective variety of dimension $n\geq 3$ over a field of characteristic $0$. To state our main result, we recall some notation from tilt-stability and refer to Section~\ref{sec:parameters} for more details. For any $B\in\NS(X)_\RR$ and $E\in \Db(X)$, we define $\ch^B(E)\coloneqq e^{-B}.\ch(E)$ and
\[\Del^H_B(E)\coloneqq \bigl(H^{n-1}.\ch_1^B(E)\bigr)^2 -2H^n.\ch_0(E)\,H^{n-2}.\ch_2^B(E).\]
Let $\Coh^{H,B}(X)$ be the first tilted heart. For any $E\in \Coh^{H,B}(X)$ and $\alpha>0$, we set
\[\nu^H_{\alpha,B}(E)\coloneqq \frac{H^{n-2}.\ch_2^B(E)-\frac{\alpha^2}{2}H^n.\ch_0(E)} {H^{n-1}.\ch_1^B(E)}\]
with value $+\infty$ when the denominator is $0$. Our main theorem is as follows.

\begin{theorem}\label{thm:main1.1}
There is a constant $\mathsf{C}=\mathsf{C}(X,H)>0$ depending only on the pair $(X,H)$ such that for every
$(\alpha,B)\in\RR_{>0}\times \NS(X)_\RR$ and every $\nu^H_{\alpha,B}$-semistable object $E\in\Coh^{H,B}(X)$, we have
\begin{equation}\label{eq:quad}
\begin{aligned}
&\alpha^2\Del^H_B(E)
 +4\bigl(H^{n-2}.\ch_2^B(E)\bigr)^2\\
&\qquad
 -6\bigl(H^{n-1}.\ch_1^B(E)\bigr)
 \left(H^{n-3}.\ch_3^B(E)
 -\mathsf{C}\sqrt{\Del^H_B(E)}\right)\ge0.
\end{aligned}
\end{equation}
\end{theorem}

We can make the constant $\mathsf{C}(X,H)$ very explicit, see Remark~\ref{rmk:constant}.

If we set $\Gamma\coloneqq \mathsf{C}H^{n-1}$, then the quadratic form in
\cite[Proposition 2.8]{BMSZ}, with this $\Gamma$, equals the left-hand side of \eqref{eq:quad} plus a nonnegative term. Hence, \eqref{eq:quad} implies the $\Gamma$-BMT Conjecture.

\begin{corollary}\label{cor:gamma-bmt}
Set $\Gamma=\mathsf{C}H^{n-1}$. Then for every $(\alpha,B)\in\RR_{>0}\times \NS(X)_\RR$ and every $\nu^H_{\alpha,B}$-semistable object $E\in\Coh^{H,B}(X)$, we have
\[\begin{aligned}
0 \leq Q^{\Gamma}_{\alpha,B}(E)
&\coloneqq \alpha^2\!\left(\Del^H_B(E)+3\,\frac{\Gamma. H}{H^n}\big(H^n.\ch_0^{B}(E)\big)^2\right) \\
&\quad +\,2\big(H^{n-2}.\ch_2^{B}(E)\big)\!\left(2\,H^{n-2}.\ch_2^{B}(E)-3\,(\Gamma. H)\,\ch_0^{B}(E)\right) \\
&\quad -\,6\big(H^{n-1}.\ch_1^{B}(E)\big)\!\left(H^{n-3}.\ch_3^{B}(E)-\Gamma. \ch_1^{B}(E)\right).
\end{aligned}\]
\end{corollary}

\subsection{Applications}

When $n=3$, our Theorem~\ref{thm:main1.1} and Corollary~\ref{cor:gamma-bmt} have several important applications. 

Let $(X,H)$ be a polarized Calabi--Yau threefold over $\CC$. For any numerical class $\bv\in \Knum(X)$, we denote by $\DT_H(\bv)\in \QQ$ the generalized Donaldson--Thomas (DT) invariant counting Gieseker-semistable sheaves (with respect to $H$) of class $\bv$, as  constructed in \cite{JS}. If $\rk(\bv)=1$, up to a multiple, this is the usual DT invariant counting ideal sheaves of subschemes in $X$ (see~\cite{Th}).

In \cite{FT1,FT2}, it was shown that arbitrary rank DT invariants can be expressed in terms of those of rank $1$, conditional on a version of the BMT Conjecture.  Using Theorem~\ref{thm:main1.1}, we can make their result unconditional.

\begin{theorem}[{Theorem~\ref{thm:rk-r}}]
Let $(X,H)$ be a polarized Calabi--Yau threefold. Then for a fixed $\bv\in \Knum(X)$ with $\rk(\bv)\geq 0$, \[\mathsf{DT}_H(\bv)=F_1(\mathsf{DT}_H(\alpha_1),\mathsf{DT}_H(\alpha_2),\dots)\]
is a universal polynomial in invariants $\mathsf{DT}_H(\alpha_i)$ with $\rk(\alpha_i)=1$.
\end{theorem}

As a corollary, which is already mentioned in \cite{FT1}, we get the symplectic invariance of $\DT_H(\bv)$.

\begin{corollary}[{Corollary~\ref{cor:sym-inv}}]
Let $(X_1,H_1)$ and $(X_2,H_2)$ be polarized Calabi--Yau threefolds.
Choose K\"ahler forms $\omega_i$ with $[\omega_i]=H_i\in \rH^2(X_i,\RR)$ and a
symplectomorphism
\[\phi\colon(X_1,\omega_1)\xra{\simeq}(X_2,\omega_2).\]
Let $\bv_i\in\Knum(X_i)$ satisfy
$\ch(\bv_1)=\phi^*\ch(\bv_2)$. Then
\[\DT_{H_1}(\bv_1)=\DT_{H_2}(\bv_2).\]
\end{corollary}

Another application of Theorem~\ref{thm:main1.1} concerns the Castelnuovo Bound Conjecture. Roughly speaking, for a Calabi--Yau threefold $X$ of Picard number one, physics predicts that its Gopakumar--Vafa invariants are zero when the genus is larger than an explicit bound. We refer to \cite{liu-ruan} for more background and the relation to Gromov--Witten invariants.

By \cite{liu-ruan,fey:modularity}, the Castelnuovo Bound Conjecture holds when $X$ satisfies \eqref{eq:quad} with $\mathsf C=0$. Moreover, the original conjecture is reduced to the following form.

\begin{conjecture}
Let $X$ be a Calabi--Yau threefold of Picard number one and $H$ be the ample generator of $\NS(X)$. Then for any one-dimensional closed subscheme $C\subset X$ of degree $d$ and arithmetic genus $g$, we have
\[g\leq \frac{1}{2H^3}d^2+\frac{1}{2}d+1.\]
\end{conjecture}

Without assuming \eqref{eq:quad}, \cite{liu} proves an unconditional upper bound for $g$  when $d\gg 0$ with the optimal quadratic term. Following the algorithm in \cite{liu}, the linear term of such an asymptotic bound was later improved in \cite{GRS} when $d\gg 0$, while leaving a large constant term.

Using Theorem~\ref{thm:main1.1} and arguments from \cite{liu-ruan,liu}, we can prove the Castelnuovo Bound Conjecture for any $d>0$, up to an explicit constant.

\begin{theorem}[{Corollary~\ref{cor:cast-cy}}]
Let $X$ be a Calabi--Yau threefold of Picard number one and $H$ be the ample generator of $\NS(X)$. Then for any one-dimensional closed subscheme $C\subset X$ of degree $d>0$ and arithmetic genus $g$, we have
\[g\leq \frac{1}{2H^3}d^2+\frac{1}{2}d+1+660m^{\frac{8}{3}}H^3,\]
where $m\geq 1$ is an integer such that $mH$ is very ample.
\end{theorem}

For a more general version concerning any smooth projective threefold of Picard number one, see Theorem~\ref{thm:cast-bound}.

We expect Theorem~\ref{thm:main1.1} to have further applications, particularly in dimensions $n\geq 4$, which we leave for future work.

\subsection{Outline of the proof}
The proof of Theorem~\ref{thm:main1.1} combines the induced stability conditions of \cite{Li,LLL} and the canonical filtration of \cite{Sun} with estimates for Frobenius pullbacks.

By Lemma~\ref{lem:rationalparametertoreal} and Proposition~\ref{prop:ch3boundimplyquadform}, it suffices to prove
\[H^{n-3}.\ch_3^B(E)-\frac{\alpha^2}{6}H^{n-1}.\ch_1^B(E) \le\mathsf C\sqrt{\Del^H_B(E)}\]
for rational $B$ and $\alpha^2$ and every $\nu^H_{\alpha,B}$-semistable object $E\in\Coh^{H,B}(X)$ with $\nu^H_{\alpha,B}(E)=0$. According to \cite{LLSZ}, we have a weaker bound on $H^{n-3}.\ch_3^B$ for $\nu^{H}_{\alpha,B}$-semistable objects of slope zero when $\alpha$ is sufficiently large (see~Proposition~\ref{prop:weak}). We also give a direct proof of this result when $n=3$ in Section~\ref{sec:geometry}.

In order to strengthen this inequality and extend it to all $\alpha>0$, we spread out $E$ and the related data and consider the specializations $E_s$ in characteristic $p>0$. Using the canonical filtration of \cite{Sun} (see Lemma~\ref{lem:graded}), we bound the tilt-slopes of the Harder--Narasimhan factors of the normalized Frobenius pullback of $E_s$ in Theorem~\ref{thm:twoendpoints}. We then obtain an upper bound for $H^{n-3}_s.\ch_3^{B_s}(E_s)$ by applying the weak inequality obtained from Proposition~\ref{prop:weak} to these factors after moving each factor to slope zero. Comparing the growth of the terms in this bound and letting $p\to\infty$, we obtain the estimate for $H^{n-3}.\ch^B_3(E)$ in Lemma~\ref{lem:jointmoments}. Finally, optimizing the constants in Lemma~\ref{lem:jointmoments} and extending the resulting estimate to all $\alpha>0$ gives Theorem~\ref{thm:main1.1} (see~Section~\ref{sec:proof}).

\begin{remark}
The idea of applying Frobenius to the BMT inequality has circulated in the Bridgeland stability community for more than a decade. 
For example, special cases of the conjectural inequality were previously established in \cite{BMS}, where multiplication maps on abelian threefolds played a role analogous to Frobenius; 
in \cite{BMSZ, Naoki:Stabnefbundle}, which used toric Frobenius; and in \cite{sun2021stabilityconditionsthreefoldsvanishing}, which directly used the Frobenius morphism. These methods have antecedents in Gieseker's proof of the classical Bogomolov inequality via reduction to positive characteristic \cite{Gieseker}, Shepherd-Barron's study of the instability of Frobenius pullbacks \cite{ShepherdBarron:Frobenius}, and Langer's work on Bogomolov inequalities in positive characteristic \cite{Langer,Langer:mixed}.
\end{remark}

\subsection{Notation}
For a Noetherian scheme $X$, we denote by $\Db(X)$ the bounded derived category of coherent sheaves on $X$.  For the heart $\cA$ of a bounded t-structure on $\Db(X)$, we denote by $(\D_{\cA}^{\leq 0},\D^{\geq 0}_{\cA})$ the associated t-structure. For an object $E\in \Db(X)$, we write $\cH^i_{\cA}(E)$ for its $i$-th cohomology object with respect to $\cA$. When $\cA=\Coh(X)$, we simply write $\cH^{i}(E)\coloneqq \cH^i_{\Coh(X)}(E)$. All pushforward and pullback functors in the paper are derived.

\subsection{Acknowledgments}

We want to thank Soheyla Feyzbakhsh, Laura Pertusi, Yukinobu Toda, and Zhixian Zhu for useful discussion on this subject. We also  thank Arend Bayer, Emanuele Macr\`i, Paolo Stellari, Richard Thomas, and Yukinobu Toda for sharing their opinions, comments, and enlightening thoughts on the usage of AI.

During the preparation of this paper, C.L.~was supported by the Royal Society University Research Fellowship URF\textbackslash R1\textbackslash 251025 ``Stability condition and application in algebraic geometry''; Zh.L.~was supported by NSFC Grant 123B2002;  A.P.~was partially supported by the NSF CAREER grant DMS-2143271;  and X.Z.~was partially supported by the NSF FRG grant DMS-2052665.

\subsection{AI disclosure} 
The conjectural BMT inequality has long been a central focus of our research. We began using GPT-5.6 Sol to assist our investigations in the middle of July, 2026. In our initial interactions, GPT produced several new families of Calabi--Yau threefolds satisfying the BMT conjecture, following the approaches developed in \cite{li:quintic,FKLR}. 

The project gained new momentum when we supplied GPT-6 Astra several of our drafts together with a heuristic strategy for approaching the BMT inequality on all Calabi--Yau threefolds. This strategy included establishing an adjusted inequality for sufficiently large $\alpha$ by pulling back stability conditions, following \cite{Liu,Li,LLL,CF}, as well as using Frobenius pullback and mimicking the method of Xiaotao Sun in \cite{Sun,Sun:Frobenius-survey} to control Harder--Narasimhan factors.
With these inputs and additional prompting, GPT-6 Astra eventually produced a BMT-type inequality on all smooth projective threefolds. 
We then substantially simplified this argument, as well as generalized it to varieties of arbitrary dimension, yielding the stated form of Theorem~\ref{thm:main1.1}. 

The argument developed here depends essentially on the recent existence results \cite{Liu,Li,LLL} for stability conditions on varieties. We doubt that GPT-6 Astra would have obtained a proof without this input. We believe that a human mathematician could also have arrived at the proof in the paper; our subjective impression, however, is that doing so without the assistance of LLMs might have required significantly more time, perhaps several years or more. 

\subsection{Related work} 
We first circulated a complete draft of this paper with experts on September 23, 2026. 
While we were revising the paper, 
OpenAI released on October~6, 2026 a collection of mathematical manuscripts generated by an internal model. 
The manuscript \cite{OAI-BMT} contains a proof of a weaker version of Theorem \ref{thm:main1.1} for $n=3$. 
We have found the paper's exposition difficult to follow and have not yet been able to verify many of its technical details or assess the correctness of its arguments.

On a preliminary reading, however, its proof strategy appears substantially different from ours. There is some overlap between Section~\ref{sec:geometry} of our paper and Sections 3 and 4 of the OpenAI manuscript. This is natural, since both arguments take the existence of stability conditions, and more specifically the support property induced from powers of an elliptic curve, as a starting point. Beyond this common input, the approaches appear to diverge considerably. The OpenAI paper uses Frobenius in a more mild and completely different way, combining it with tools in geometric analysis such as heat kernels and Hermitian--Einstein metrics.  We may include a more detailed comparison in a future version, although this will require significant further work.

\section{Tilt-stability and BMT inequalities}\label{sec:parameters}

In this section, we first review the theory of tilt-stability, following \cite{BMS}. Then we discuss the conjectural inequalities proposed in \cite{BMT,BMS,BMSZ} and some reductions of them.

\subsection{Stability conditions}

We first recall definitions and basic properties of (weak) stability conditions on triangulated categories. We follow \cite{bridgeland:stability,Bay:deformation,LMPSZ}. Fix a triangulated category $\cD$.

\begin{definition}
A \emph{slicing} $\cP$ on $\cD$ is a collection $\{\cP(\phi)\}_{\phi\in \mathbb{R}}$ of full additive subcategories of $\cD$ that satisfies

\begin{itemize}
    \item $\cP(\phi+1)=\cP(\phi)[1]$ for all $\phi\in \mathbb{R}$,

    \item $\Hom(\cP(\phi_1), \cP(\phi_2))=0$ for $\phi_1>\phi_2$, and

    \item {(Harder--Narasimhan filtrations)} for any non-zero $E\in \cD$, there exists a finite sequence of morphisms
    \[0=E_0\xra{s_1}E_1\to \cdots \xra{s_m} E_m=E\]
    such that the cone $\mathrm{cone}(s_i)$ is in $\cP(\phi_i)$ for a sequence $\phi_1>\phi_2>\cdots>\phi_m$ of real numbers.
\end{itemize}
\end{definition}

For an interval $I\subset \mathbb{R}$, we define $$\cP(I)\coloneqq\{E \mid \phi_{\cP}^+(E), \phi_{\cP}^-(E)\in I\}\cup \{0\}=\langle \cP(\phi) \rangle_{\phi\in I},$$ where $\langle -\rangle$ denotes the extension closure.

\begin{definition}
Let $\Lambda$ be a finite rank lattice and $v\colon \KK(\cD)\to \Lambda$ be a group homomorphism.

\begin{itemize}
    \item A \emph{weak pre-stability condition} on $\cD$ (with respect to $(\Lambda,v)$) is a pair $\sigma=(Z, \cP)$ where $\cP$ is a slicing of $\cD$ and $Z\colon \Lambda\to \CC$ is a group homomorphism, called the \emph{central charge}, such that for all non-zero objects $E\in \cP(\phi)$, we have $Z(v(E))\in \mathbb{R}_{> 0}\cdot e^{\mathfrak{i}\pi \phi}$ for $\phi\notin \ZZ$ and $Z(v(E))\in \mathbb{R}_{\geq 0}\cdot e^{\mathfrak{i}\pi \phi}$ for $\phi\in \ZZ$.

    If, moreover, for all $\phi\in \RR$ and non-zero objects $E\in \cP(\phi)$, we have $Z(v(E))\neq 0$, then $\sigma$ is called a \emph{pre-stability condition}.

    By abuse of notation, we write $Z(E)$ for $Z(v(E))$. The non-zero objects of $\cP(\phi)$ are called \emph{$\sigma$-semistable of phase $\phi$}. Simple objects of $\cP(\phi)$ are called \emph{$\sigma$-stable}.

    \item A weak pre-stability condition $\sigma=(Z, \cP)$ satisfies the \emph{support property} if there exists a quadratic form $Q$ on $(\Lambda/\Lambda_0)_{\RR}$ such that $Q|_{(\ker(Z)/\Lambda_0)_{\RR}}$ is negative definite and $Q(\overline{v}(E))\geq 0$ for any $\sigma$-semistable object $E\in \cD$. Here, $$\Lambda_0\coloneqq\langle v(E) \mid E\in \cP(0,1], Z(E)=0 \rangle\subset \Lambda,$$
where $\langle S \rangle$ denotes the saturated sublattice generated by a subset $S\subset \Lambda$, and $\overline{v}$ is the induced homomorphism $\overline{v}\colon \KK(\cD)\to \Lambda/\Lambda_0$.

    \item A (weak) pre-stability condition $\sigma=(Z, \cP)$ is called a \emph{(weak) stability condition} if $\sigma$ satisfies the support property.
\end{itemize}
\end{definition}

We denote by $\Stab^{\mathsf{w}}_{(\Lambda,v)}(\cD)$ the set of all weak stability conditions on $\cD$ with respect to $(\Lambda,v)$. In the following, we equip $\Stab^{\mathsf{w}}_{(\Lambda,v)}(\cD)$ with the coarsest topology such that the maps
\[\Stab^{\mathsf{w}}_{(\Lambda,v)}(\cD)\to \Hom_{\ZZ}(\Lambda, \CC), \quad \sigma \mapsto Z\]
and
\[\Stab^{\mathsf{w}}_{(\Lambda,v)}(\cD)\to \RR^2, \quad \sigma \mapsto (\phi_{\sigma}^+(E), \phi_{\sigma}^-(E))\]
are continuous for any $0\neq E\in \cD$. The corresponding subspace consisting of stability conditions is written as $\Stab_{(\Lambda,v)}(\cD)$.

We will also use a generalized distance between slicings defined as follows:
\[\mathrm{dist}(\cP_1,\cP_2)\coloneqq \sup_{0\neq F\in \cD}\left\{\left|\phi^+_{\cP_1}(F)-\phi^+_{\cP_2}(F)\right|,\left|\phi^-_{\cP_1}(F)-\phi^-_{\cP_2}(F)\right|\right\}\in [0,+\infty].\]
If $\sigma_i=(Z_i,\cP_i)\in \Stab_{(\Lambda,v)}(\cD)$ for $i=1,2$, we define
\[\mathrm{dist}(\sigma_1,\sigma_2)\coloneqq \max\{\mathrm{dist}(\cP_1,\cP_2), \|Z_1-Z_2\|\}\]
where $\|-\|$ is a fixed norm on $\Hom_{\ZZ}(\Lambda,\CC)$.

Now, we recall how to describe a weak stability condition in terms of hearts.

\begin{definition}
    Let $\cA$ be an abelian category. 
    Let $\Lambda$ be a finite rank lattice and fix a homomorphism $v \colon \KK_0(\cA) \to \Lambda$. 
    Consider a homomorphism $Z \colon \Lambda \to \CC$. 
    \begin{enumerate}
        \item We say that $Z$ is a \emph{weak stability function} on $\cA$ with respect to $(\Lambda, v)$ if for all objects $0 \neq F \in \cA$, we have 
        \[Z(v(F)) \in \{{ z \in \CC \mid \Im z > 0, \text{ or } \Im z = 0 \text{ and } \Re z \leq 0}\}.\]
        In this case, we define the \emph{phase} of a nonzero object $0 \neq F \in \cA$ by 
        \[\phi(F) \coloneqq \frac{1}{\pi} \arg Z(F) \in (0,1],\]
        and we say $F$ is \emph{$Z$-semistable} if for all subobjects $0 \neq A \hookrightarrow F$ we have $\phi(A) \leq \phi(F)$. 

\item If $Z$ is a weak stability function and $Z(v(F))\neq 0$ for any $0\neq F\in \cA$, then we say $Z$ is a \emph{stability function}.
        
        \item If $Z$ is a weak stability function, we say that it satisfies the \emph{Harder--Narasimhan (HN) property} if for every nonzero object $F \in \cA$ there exists a sequence
        \[0 = F_0 \hookrightarrow F_1 \hookrightarrow \dots \hookrightarrow F_m = F\]
        of subobjects in $\cA$ such that each factor $F_i/F_{i-1}$ is $Z$-semistable and 
        \[\phi(F_1/F_0) > \phi(F_2/F_1) > \cdots > \phi(F_m/F_{m-1}).\] 
    \end{enumerate}
\end{definition}

\begin{lemma}[{\cite[Proposition~5.3]{bridgeland:stability}}]
    Let $\Lambda$ be a finite rank lattice and fix a homomorphism
    $v \colon \KK_0(\cD) \to \Lambda$. 
    Then the following data are equivalent: 
    \begin{enumerate}
        \item \label{correspondence-P} A weak stability condition on $\cD$ with respect to $(\Lambda, v)$. 
        \item \label{correspondence-A} A pair $(Z, \cA)$ where $\cA \subset \cD$ is the heart of a bounded t-structure and $Z$ is a weak stability function on $\cA$ with respect to $(\Lambda, v)$ which satisfies the Harder--Narasimhan property and support property.
    \end{enumerate}
    Explicitly, given a weak stability condition $\sigma = (Z, \cP)$ as in~\ref{correspondence-P}, we get a pair $(Z, \cA)$ as in~\ref{correspondence-A} by taking $\cA = \cP((0,1])$. 
    Conversely, given a pair $(Z, \cA)$ as in~\ref{correspondence-A}, we get a weak stability condition $(Z, \cP)$ as in~\ref{correspondence-P} by taking $\cP(\phi)$ to consist of $0$ and $Z$-semistable objects of phase $\phi$ in $\cA$ for $\phi \in (0,1]$ and extending to all $\phi \in \RR$ via the rule $\cP(\phi+1) = \cP(\phi)[1]$. 
\end{lemma}

When we want to emphasize the heart $\cA$ instead of the slicing $\cP$, we use the notation $\sigma = (Z, \cA)$ for a (weak) stability condition on $\cD$.

\begin{definition}[\cite{Li:real}]
Let $\cP_1$ and $\cP_2$ be slicings on $\cD$. We define

\begin{itemize}
    \item $\cP_1\prec\cP_2$ if $\cP_1(\phi)\subset\cP_2(<\phi)$ for all $\phi\in\RR$, and

    \item $\cP_1\preceq\cP_2$ if $\cP_1(\phi)\subset\cP_2(\leq \phi)$ for all $\phi\in\RR$.
\end{itemize}
For stability conditions $\sigma_1=(Z_1,\cP_1)$ and $\sigma_2=(Z_2,\cP_2)$ on $\cD$, we write
\[\sigma_1\prec \sigma_2\]
if $\cP_1\prec \cP_2$ and similarly for $\preceq$.
\end{definition}

\subsection{Tilt-stability}

Let $(X,H)$ be a polarized smooth projective variety of dimension $n\geq 3$ over a field $\kk$. 
We recap the tilt-heart as that from \cite{bridgeland:K3,AB}
Recall that for any $E\in \Coh(X)$, we define
\[\mu_H(E)\coloneqq \frac{H^{n-1}.\ch_1(E)}{H^n\rk(E)}\]
if $\rk(E)\neq 0$, and to be $+\infty$ otherwise. For $B\in\NS(X)_\RR$, we define a torsion pair $(\mathcal F_{H,B}, \mathcal T_{H,B})$ of $\Coh(X)$ as 
\[\begin{aligned}
 \mathcal F_{H,B}&=\left\{F\in\Coh(X)\;\middle|\;\mu_H^+(F)\le \frac{B.H^{n-1}}{H^n}\right\},\\
 \mathcal T_{H,B}&=\left\{F\in\Coh(X)\;\middle|\;\mu_H^-(F)>\frac{B.H^{n-1}}{H^n}\right\}.
 \end{aligned}\]
By the general theory on tilting heart \cite{RS:tilting}, the extension $\Coh^{H,B}(X)\coloneqq \langle\mathcal F_{H,B}[1],\mathcal T_{H,B}\rangle$ is the heart of a bounded t-structure on $\Db(X)$.

For simplicity, we set
\[v_i^{H,B}(E)\coloneqq H^{n-i}.\ch^B_i(E)\]
for any $E\in \Db(X)$. If $B=0$, then we omit $B$ from the notation.

For any object $E\in \Coh^{H,B}(X)$ and any $\alpha>0$, we set
\[\nu^H_{\alpha,B}(E)\coloneqq \frac{v_2^{H,B}(E)-\frac{\alpha^2}{2}v_0^{H,B}(E)} {v_1^{H,B}(E)}\]
with value $+\infty$ when the denominator is zero. Then we say $E$ is $\nu^H_{\alpha,B}$-semistable if \[\nu^H_{\alpha,B}(F)\leq \nu^H_{\alpha,B}(E/F)\] for any subobject $0\ne F\subsetneq E$ in $\Coh^{H,B}(X)$. We say $E$ is $\nu^H_{\alpha,B}$-stable if the above inequality is strict.

 For any $\epsilon\in\RR_{\ge0}$, we define
\[\Del^H_{B,\epsilon}(E)\coloneqq (v_1^{H,B}(E))^2-2v_0^{H,B}(E)v_2^{H,B}(E) +\epsilon\bigl(v_0^H(E)\bigr)^2\]
and $\Del^H_B(E)\coloneqq \Del^H_{B,0}(E)$. Note that $\Del^H_{B,\epsilon}(E)=\Del^H_{B+t H,\epsilon}(E)$ for any $t\in \RR$.

We need the following version of Bogomolov--Gieseker inequality.

\begin{theorem}\label{thm:uniformBG}
For any $B\in \NS(X)_\RR$ and $\mu_H$-semistable torsion-free sheaf $E$ on $X$, we have
\begin{equation}\label{eq:scalarBG}
 \frac{1}{H^n}\Del^H_{B,\epsilon}(E)\geq H^{n-2}.\left(\ch_1(E)^2-2\rk(E)\ch_2(E)\right)+\epsilon\rk(E)^2H^n\geq 0,
\end{equation}
where
\begin{itemize}
    \item when $\mathrm{char}(\kk)=0$, we have $\epsilon=0$, and

    \item when $\mathrm{char}(\kk)>0$, we have
    \[\epsilon=\frac{\max\{0,\,2+5m^{n-2}(K_X+(n-2)mH)^2.H^{n-2}-\chi(\cO_S)\}}{m^{n-2}H^n},\]
where $m>0$ is an integer such that $mH$ is very ample and $K_X+mH$ is ample, and $S$ is a general complete intersection by divisors in $|mH_{\overline{\kk}}|$.
\end{itemize}
\end{theorem}

\begin{proof}
When $\mathrm{char}(\kk)=0$, this is the classical result of Bogomolov--Gieseker. When $\mathrm{char}(\kk)>0$, this follows from \cite[Theorems 1.1 and 1.3]{Kos}. The accurate form of $\epsilon$ will not be relevant to the argument later on.
\end{proof}

For any $B\in\NS(X)_\RR$ and $\alpha>0$, we define a homomorphism $Z^{\alpha,B}_H\colon \mathrm{K}_0(X)\to \CC$ by 
\[Z^{\alpha,B}_H\coloneqq -v_2^{H,B}+\frac{\alpha^2}{2}v_0^H+\mathfrak{i} v_1^{H,B}.\]

Now, we fix a class $B_0\in \NS(X)_{\RR}$, and assume that $B_0\in \NS(X)_{\QQ}$ when $\mathrm{char}(\kk)>0$. The following result describes basic properties of tilt-stability.

\begin{proposition}\label{prop:tiltBG}
There is a continuous family of weak stability conditions on $\Db(X)$, parameterized by $\RR_{>\sqrt{\epsilon}}\times \RR$, given by
\[(\alpha,\beta)\in \RR_{>\sqrt{\epsilon}}\times \RR\mapsto (Z^{\alpha,B_0+\beta H}_H,\Coh^{H, B_0+\beta H}(X)).\]

Moreover, for any $(\alpha,\beta)\in \RR_{>\sqrt{\epsilon}}\times \RR$ and any $\nu^H_{\alpha,B_0+\beta H}$-semistable object $E$, we have 

\begin{enumerate}
    \item $\Del^H_{B_0+\beta H,\epsilon}(E)\ge0$, and
\item there is a collection $\{W_i\}_{i\in I}$ of circles in $\RR^2$ centered on the $\beta$-axis and possibly a vertical line, such that $\{W_i\cap \RR_{>\sqrt{\epsilon}}\times \RR\}_{i\in I}$ (called ``walls") form a locally finite collection and satisfy

\begin{itemize}
    \item walls are disjoint with each other,

    \item $\nu^H_{\alpha,B_0+\beta H}$-(semi)stability of $E$ is unchanged as $(\alpha,\beta)$ varies within
any connected component (called a ``chamber") of $\RR_{>\sqrt{\epsilon}}\times \RR\setminus \cup_{i\in I}W_i$, and

\item for each nonvertical wall $W_i\cap \RR_{>\sqrt{\epsilon}}\times \RR$, there is a subobject $F\subset E$ in $\Coh^{H, B_0+\beta H}(X)$ such that both $E$ and $F$ are $\nu^H_{\alpha,B_0+\beta H}$-semistable when $(\alpha,\beta)\in W_i\cap \RR_{>\sqrt{\epsilon}}\times \RR$, and $F\subset E$ strictly destabilizes $E$ for $(\alpha,\beta)$ in one of the two chambers adjacent to the wall $W_i\cap \RR_{>\sqrt{\epsilon}}\times \RR$.
\end{itemize}
\end{enumerate}
\end{proposition}

\begin{proof}
When $\mathrm{char}(\kk)=0$, we have $\epsilon=0$ and the result follows from \cite[Appendix B]{BMS}. When $\mathrm{char}(\kk)>0$, applying \cite[Theorems 9.8, 9.9]{LM} to $\gamma=e^{-B_0}.\left(1,0,-\frac\epsilon2H^2,0\right)$, $b=\beta$, and $w=\frac{\beta^2+\alpha^2-\epsilon}{2}$ gives the result.
\end{proof}

We shall use the following standard lemmas.

\begin{lemma}[Bertram Nested Wall]\label{lem:nestedwall}
Let $(\alpha,\beta)\in \RR_{>\sqrt{\epsilon}}\times \RR$ and $E$ be a $\nu^H_{\alpha,B_0+\beta H}$-semistable object with $\nu^H_{\alpha,B_0+\beta H}(E)=\mu <+\infty$. Then $E$ is  $\nu^H_{\alpha',B_0+\beta' H}$-semistable for $\alpha'>\sqrt{\epsilon}$ with
\[(\alpha')^2+(\beta'-\beta-\mu)^2=\alpha^2+\mu^2.\]
In particular, $E$ is $\nu^H_{\alpha',B_0+\beta' H}$-semistable with $\nu^H_{\alpha',B_0+\beta' H}(E)=0$ for 
\[(\alpha',\beta')=\left(\sqrt{\alpha^2+\mu^2},\beta+\mu\right).\]
\end{lemma}

\begin{proof}
Using Proposition~\ref{prop:tiltBG}, the argument is the same as \cite[Lemma 4.3]{BMS}.
\end{proof}

\begin{lemma}\label{lem:deltacontrol}
Let $(\alpha,\beta)\in \RR_{>\sqrt{\epsilon}}\times \RR$ and $E$ be a $\nu^H_{\alpha,B_0+\beta H}$-semistable object with $\nu^H_{\alpha,B_0+\beta H}(E)<+\infty$. Then $E$ admits a Jordan--H\"older filtration such that factors $E_1,\cdots,E_k$ are $\nu^H_{\alpha,B_0+\beta H}$-stable and 
\[\sum_{i=1}^k\sqrt{\Del^H_{B_0+\beta H,\epsilon}(E_i)} \le\sqrt{\Del^H_{B_0+\beta H,\epsilon}(E)}.\]
\end{lemma}

\begin{proof}
This follows from \cite[Lemma 4.16 and Lemma 9.10]{LM}, see also \cite[Appendix A]{BMS}.
\end{proof}

\begin{lemma}\label{lem:2.5gap}
Let $(\alpha,\beta)\in \RR_{>\sqrt{\epsilon}}\times \RR$ and $E\in \Coh^{H,B_0+\beta H}(X)$ be a $\nu^H_{\alpha,B_0+\beta H}$-semistable object with $\nu^H_{\alpha,B_0+\beta H}(E)=\mu <+\infty$. Then we have
\begin{equation}\label{eq:2.5}
\begin{aligned}
 \mu_{H}^+(\cH^{-1}(E))
 &\le \frac{(B_0+\beta H). H^{n-1}}{H^n}+\mu-\sqrt{\alpha^2-\epsilon+\mu^2},\\
 \mu_{H}^-(\cH^0(E))
 &\ge \frac{(B_0+\beta H). H^{n-1}}{H^n}+\mu+\sqrt{\alpha^2-\epsilon+\mu^2}.
\end{aligned}
\end{equation}
\end{lemma}

\begin{proof}
The first HN factor (with respect to $\mu_H$) $U$ of $\cH^{-1}(E)$ gives a subobject $U[1]\subset E$ in $\Coh^{H,B_0+\beta H}(X)$, while the last HN factor $V$ of
$\cH^0(E)$ gives a quotient object
$E\twoheadrightarrow V$ in $\Coh^{H,B_0+\beta H}(X)$. Set $d=\frac{(B_0+\beta H). H^{n-1}}{H^n}-\mu_{H}(U)>0$ and
$e=\mu_{H}(V)-\frac{(B_0+\beta H). H^{n-1}}{H^n}>0$. Applying \eqref{eq:scalarBG} to $U$ and $V$, we then get
\[\frac{\alpha^2-\epsilon-d^2}{2d} \le\nu^{H}_{\alpha,B_0+\beta H}(U[1])\le \mu \le\nu^{H}_{\alpha,B_0+\beta H}(V) \le\frac{e^2-(\alpha^2-\epsilon)}{2e}.\]
Solving the quadratic inequalities for $d$ and $e$ gives the bound for $\mu_H(U)$ and $\mu_H(V)$ as that in \eqref{eq:2.5}.
\end{proof}

\subsection{BMT-type inequalities and reductions}
We continue to denote by $(X,H)$ a polarized smooth projective variety of dimension $n\geq 3$ over a field $\kk$. 
We end this section by discussing various BMT-type inequalities for $H^{n-3}.\ch_3$ of tilt-semistable objects in \cite{BMT,BMS,BMSZ}. The key takeaways are Lemma~\ref{lem:smallAimplieslargeA} and Proposition~\ref{prop:ch3boundimplyquadform}, namely, the quadratic form bound for semistable characters as that in \eqref{eq:quad} can be reduced to the bound for zero-slope $\nu^H_{\alpha,B}$-semistable $\ch_3$ with small $\alpha$. Such an inequality will be our main focus in the rest of the paper.

We first record the passage from rational to real divisor twists in characteristic $0$.

\begin{lemma}\label{lem:rationalparametertoreal}
Assume $\operatorname{char}(\kk)=0$ and $\mathsf{C}\ge0$. Suppose
\begin{equation}\label{eq:dense-bound}
 v_3^{H,B}(E)-\frac{\alpha^2}{6}v_1^{H,B}(E)
       \le \mathsf{C}\sqrt{\Del^{H}_{B}(E)}
\end{equation}
for every $B\in \NS(X)_{\QQ}$, every $\alpha>0$ with $\alpha^2\in\QQ$,
and every $\nu^{H}_{\alpha,B}$-stable object $E$ with $\nu^{H}_{\alpha,B}(E)=0$.
Then \eqref{eq:dense-bound} holds for every $B\in \NS(X)_{\RR}$, every
$\alpha>0$, and every $\nu^{H}_{\alpha,B}$-semistable object $E$ with $\nu^{H}_{\alpha,B}(E)=0$.
\end{lemma}
\begin{proof}
By Lemma~\ref{lem:deltacontrol}, we may assume that $E$ is $\nu^{H}_{\alpha,B}$-stable. Therefore, $E$ is $\nu^{H}_{\alpha',B'}$-stable for any $|\alpha'-\alpha|\ll 1$ and $B'$ sufficiently close to $B$ by \cite[Proposition~B.5]{BMS}. Now, choose rational $B_i\in \NS(X)_\QQ$ and $\alpha_i$ such that $B_i\to B$ and $\alpha_i\to \alpha$. Hence $\mu_i\coloneqq \nu^H_{\alpha_i, B_i}(E)\in \QQ$, $\mu_i\to \nu^{H}_{\alpha,B}(E)=0$ and $E$ is $\nu^{H}_{\alpha_i,B_i}$-stable. By Lemma~\ref{lem:nestedwall}, $E$ is $\nu^H_{\sqrt{\alpha_i^2+\mu_i^2},B_i+\mu_iH}$-semistable with $\nu^H_{\sqrt{\alpha_i^2+\mu_i^2},B_i+\mu_iH}(E)=0$. Thus, by our assumption, we have
\[v_3^{H,B_i+\mu_iH}(E)-\frac{\alpha_i^2+\mu_i^2}{6}v_1^{H,B_i+\mu_iH}(E) \le \mathsf{C}\sqrt{\Del^{H}_{B_i}(E)}.\]
Taking the limit gives
\[v_3^{H,B}(E)-\frac{\alpha^2}{6}v_1^{H,B}(E) \le \mathsf{C}\sqrt{\Del^{H}_{B}(E)}.\]
\end{proof}

\begin{proposition}\label{prop:propagate}
Fix $B_0\in\NS(X)_\QQ$. Suppose there exist $A>\sqrt\epsilon$ and $\mathsf{C}\ge0$ that satisfy
\begin{equation}\label{eq:slice}
 v_3^{H,B_0+\beta H}(E)-\frac{A^2+2\epsilon}{6}v_1^{H,B_0+\beta H}(E)
 \le \mathsf{C}\sqrt{\Del^H_{B_0+\beta H,\epsilon}(E)}
\end{equation}
for every $\beta\in\RR$ and every $\nu^H_{A,B_0+\beta H}$-semistable object $E$ with 
$\nu^H_{A,B_0+\beta H}(E)=0$. Then, for every $\alpha>\sqrt\epsilon$,
$\beta\in\RR$, and $\nu^H_{\alpha,B_0+\beta H}$-semistable object $E$ with $\nu^H_{\alpha,B_0+\beta H}(E)=0$, we have
\begin{equation}\label{eq:allzero}
 v_3^{H,B_0+\beta H}(E)-\frac{\alpha^2+2\epsilon}{6}v_1^{H,B_0+\beta H}(E)
 \le\left(\mathsf{C}+\frac{\max\{A^2-\alpha^2,0\}}6\right)
            \sqrt{\Del^H_{B_0+\beta H,\epsilon}(E)}.
\end{equation}
\end{proposition}

\begin{proof}
Let $E$ be an object that has the minimal value of $\Del^H_{B,\epsilon}(E)$ among tilt-semistable objects that violates the statement. Set $t=\alpha^2$ and we consider points on the curve $\nu^H_{\sqrt t,B(t)}(E)=0$, where $B(t)=B_0+\beta(t)H$. Write $\delta=\Del^H_{B_0,\epsilon}(E)$, which is constant along this curve. Then
\[v^{H,B(t)}_1=\sqrt{\delta+(t-\epsilon)(v^H_0)^2}.\]
Moreover, we have
\[\beta(t)=\frac{v_1^{H,B_0}(E)-v^{H,B(t)}_1}{v^H_0}\]
if $v^H_0(E)\neq 0$ and $\beta(t)$ is a constant otherwise.

By the formula
\begin{equation}\label{eq:shift}
 v_i^{H,B+tH}=\sum_{j=0}^i\frac{(-t)^j}{j!}v_{i-j}^{H,B}
\end{equation}
for any $0\le i\le3$, we get $\beta'(t)=-\frac{v^H_0(E)}{2v^{H,B(t)}_1(E)}$, $(v^{H,B(t)}_1(E))'=\frac{(v^H_0)^2(E)}{2v^{H,B(t)}_1(E)}$, and
\[\frac{d}{dt}\left(v^{H,B(t)}_3(E)-\frac{t+2\epsilon}{6}v^{H,B(t)}_1(E)\right) =-\frac{\delta}{6v^{H,B(t)}_1(E)}.\]
Since $v_1^{H,B(t)}(E)\ge\sqrt\delta$, the difference between the two sides of
\eqref{eq:allzero} is nondecreasing for $t\le A^2$ and nonincreasing for $t\ge A^2$. Therefore, there is a wall for $E$ that meets the curve $\nu^H_{\sqrt t,B(t)}(E)=0$ at a point $(\alpha_1,\beta_1)$ as $E$ does not satisfy \eqref{eq:slice}. In particular, $\Del^H_{B_0,\epsilon}(E)>0$. On the other hand, every Jordan--H\"older factor $E_i$ of $E$ with respect to $\nu^H_{\alpha_1,B_0+\beta_1H}$ satisfies $\nu^H_{\alpha_1,B_0+\beta_1H}(E_i)=\nu^H_{\alpha_1,B_0+\beta_1H}(E)=0$ and $\Del^H_{B_0,\epsilon}(E_i)<\Del^H_{B_0,\epsilon}(E)$ by Lemma~\ref{lem:deltacontrol}. Since there exists $E_i$ that violates the statement, we get a contradiction.
\end{proof}

We also need the following reduction of \cite[Question 2.4]{BMSZ}.

\begin{lemma}\label{lem:smallAimplieslargeA}
Fix $B_0\in\NS(X)_\QQ$. Let $\Gamma\in \mathrm{CH}_1(X)_{\QQ}$ with
$H.\Gamma/H^n\ge\epsilon/3$, and let $A_0>\sqrt\epsilon$.
Suppose
\[v_3^{H,B_0+\beta H}(E)\le\frac{A_0^2}{6}v_1^{H,B_0+\beta H}(E)+\Gamma.\ch_1^{B_0+\beta H}(E)\]
for every $\beta\in\RR$ and every $\nu^H_{A_0,B_0+\beta H}$-semistable
object $E\in\Coh^{H,B_0+\beta H}(X)$ with $\nu^H_{A_0,B_0+\beta H}(E)=0$.
Then the same inequality holds with $A_0^2$ replaced by $\alpha^2$ for every $\alpha\ge A_0$, $\beta\in\RR$, and every $\nu^H_{\alpha,B_0+\beta H}$-semistable object $E$ with $\nu^H_{\alpha,B_0+\beta H}(E)=0$.
\end{lemma}

\begin{proof}
Let $E$ be an object that has the minimal value of $\Del^H_{B,\epsilon}(E)$ among tilt-semistable object that violates the statement. As in Proposition~\ref{prop:propagate}, we consider the curve $\nu^H_{\sqrt t,B(t)}(E)=0$, where $B(t)=B_0+\beta(t)H$ and $t=\alpha^2$. Note that
\[\begin{aligned}
 &\frac{d}{dt}\left(v^{H,B(t)}_3(E)-\frac t6v^{H,B(t)}_1(E)-\Gamma.\ch_1^{B(t)}(E)\right)\\
 &\qquad=-\frac{\Del^H_{B_0,\epsilon}(E)
               +(3H.\Gamma/H^n-\epsilon)(v^H_0(E))^2}{6v^{H,B(t)}_1(E)}\le0,
 \end{aligned}\]
so we can get a contradiction as in Proposition~\ref{prop:propagate}.
\end{proof}

Analogously to \cite{BMS,BMSZ,LM}, we generalize \eqref{eq:dense-bound} to a quadratic inequality.

\begin{proposition}\label{prop:ch3boundimplyquadform}
Let $B_0\in\NS(X)_\RR$, with $B_0$ rational if
$\operatorname{char}(\kk)>0$, and let $\mathsf{C}\ge0$.
Suppose
\[v_3^{H,B}(E)-\frac{\alpha^2+2\epsilon}{6}v_1^{H,B}(E) \le \mathsf{C}\sqrt{\Del^H_{B,\epsilon}(E)}\]
for every $\alpha>\sqrt\epsilon$, $B\in B_0+\RR H$, and every
$\nu^H_{\alpha,B}$-semistable object $E$ with $\nu^H_{\alpha,B}(E)=0$.
Then for every $\nu^H_{\alpha,B}$-semistable object $E$ with $\alpha>\sqrt\epsilon$ and $B\in B_0+\RR H$, we have
\begin{equation}\label{eq:rootQ}
 \begin{aligned}
 0\le{}&(\alpha^2-\epsilon)\Del^H_{B,\epsilon}(E)
       +4\left(v_2^{H,B}(E)-\frac\epsilon2v_0^H(E)\right)^2\\
       &-6v_1^{H,B}(E)\left(v_3^{H,B}(E)-\frac\epsilon2v_1^{H,B}(E)\right)
       +6\mathsf{C}v_1^{H,B}(E)\sqrt{\Del^H_{B,\epsilon}(E)}.
 \end{aligned}
\end{equation}
\end{proposition}
\begin{proof}
Write $v_i=v_i^{H,B}(E)$ and
$\Delta=\Del^H_{B,\epsilon}(E)\ge0$.
If $v_1=0$, then \eqref{eq:rootQ} is immediate. Otherwise,
$v_1>0$, and we set $\mu=\nu^H_{\alpha,B}(E)$, $B'=B+\mu H$, and $(\alpha')^2=\alpha^2+\mu^2$. By Lemma~\ref{lem:nestedwall}, $E$ is
$\nu^H_{\alpha',B'}$-semistable of slope zero. Therefore, the hypothesis gives
\[v_3^{H,B'}(E) \le \frac{(\alpha')^2+2\epsilon}{6}v_1^{H,B'}(E) +\mathsf{C}\sqrt{\Delta}.\]
Expanding by \eqref{eq:shift} and using
$v_2=\alpha^2v_0/2+\mu v_1$, we obtain
\[v_3\le \frac{\alpha^2+2\epsilon+4\mu^2}{6}v_1 +\frac{\alpha^2-\epsilon}{3}\mu v_0 +\mathsf{C}\sqrt{\Delta}.\]
Multiplying by $6v_1$ and substituting
$\mu v_1=v_2-\alpha^2v_0/2$ and
$\Delta=v_1^2-2v_0v_2+\epsilon v_0^2$
gives \eqref{eq:rootQ}.
\end{proof}

\section{BMT-type inequalities at large volume} 
\label{sec:geometry}
The goal of this section is to explain Proposition~\ref{prop:weak}, which gives a BMT-type inequality for all sufficiently large parameters. 
This result is a corollary stated in \cite{LLSZ}, but here we combine the induced stability conditions of \cite{Pol,Liu,FLZ,LMPSZ,Li,LLL} with the heart comparison of \cite{CF} to give a direct proof in dimension $3$. 

Throughout this section, the base field $\kk$ is algebraically closed.

\subsection{Powers of an elliptic curve}
Let $E$ be an elliptic curve over $\kk$ and let $n \geq 0$ be an integer. Following \cite[Section 4.1]{LLL},
write $H_i$ for the pullback of a degree one divisor from the $i$th
factor of $E^n$, and set $H\coloneqq\sum_iH_i$. For any subset $I \subset \{1, \dots, n\}$, we write $H_I \coloneqq \prod_{i \in I} H_i$. 
We set $\Lambda_n\coloneqq\ZZ^{2^n}$ and consider the homomorphism 
\[v_n\colon \KK_0(E^n)\to \Lambda_n\]
given by 
\[v_n(F)\coloneqq\bigl(H_I.\ch_{n-|I|}(F)\bigr)_{I\subset\{1,\ldots,n\}}.\]

For $a>0$ and $\mathbf b=(b_1,\dots, b_n)\in\RR^n$, we define
\[Z_n^{a,\mathbf b}(F)\coloneqq -\int_{E^n}e^{-\mathfrak{i}aH-\sum_i b_iH_i}\ch(F).\]
By \cite[Theorems 4.1(2) and 4.2]{LLL}, there is a unique stability
condition 
\[\sigma_n^{1,0}=(Z_n^{1,0},\cP_n^{1,0}) \in \Stab_{(\Lambda_n, v_n)}(\Db(E^n))\]
with
$\phi_{\sigma_n^{1,0}}(\cO_x)=1$ for any closed point $x\in E^n$.

\begin{lemma}\label{lem:support}
There exists a neighborhood $U\subset\RR^n$ of $0$, depending only on $n$, such that for any $\mathbf b\in U$, there is a geometric stability  condition
\[\sigma_n^{1,\mathbf b}=(Z_n^{1,\mathbf b},\cP_n^{1,\mathbf b}) \in \Stab_{(\Lambda_n, v_n)}(\Db(E^n))\]
with $\phi_{\sigma_n^{1,\mathbf b}}(\cO_x)=1$ for all closed points $x \in E^n$ and 
\begin{equation}
\label{dist-Pn1b-Pn10}
 \mathrm{dist}(\cP_n^{1,\mathbf b},\cP_n^{1,0})< 1/2. 
\end{equation}
\end{lemma}

\begin{proof}
Note that for any object $F$, we have
\[Z_n^{1,\mathbf b}(F) =-\sum_I\prod_{j\in I}(-b_j-\mathfrak{i})\,v_n(F)_I.\]
Hence, using the quadratic form $Q^E_{n,1}$ constructed in \cite[proof of Lemma 6.6]{LLL}, \cite[Theorem A.10 and Proposition A.5]{BMS} gives the desired neighborhood $U$ that depends only on $n$. By \cite[Theorem 4.1(1)]{LLL}, each $\sigma_n^{1,\mathbf b}$ is geometric. 
Moreover, as $Z_n^{1,\mathbf b}(\cO_x) = -1$, the bound~\eqref{dist-Pn1b-Pn10} gives $\phi_{\sigma_n^{1,\mathbf b}}(\cO_x)=1$.
\end{proof}

For $\mathbf d = (d_1, \dots, d_n) \in \ZZ^n$, we let $L_{\mathbf d}\in \Pic(E^n)$ denote any line bundle with 
\[c_1(L_{\mathbf d})=\sum_i d_iH_i.\]

\begin{lemma}\label{lem:enlarge}
Let $K\subset\RR^n$ be a compact subset and $S\subset\ZZ^n$ be a finite subset.
Then there exists an integer $m_0$,  depending only on $K, S$, and $n$, such that for every integer $m\geq m_0$ and $\mathbf b \in K$, there is a geometric
stability condition
\[\sigma_n^{m^2,\mathbf b}=(Z_n^{m^2,\mathbf b},\cP_n^{m^2,\mathbf b}) \in \Stab_{(\Lambda_n, v_n)}(\Db(E^n))\] 
with $\phi_{\sigma_n^{m^2,\mathbf b}}(\cO_x)=1$ for all closed points $x \in E^n$ and 
\begin{equation}\label{eq:smalltwist}
 \mathrm{dist}(\cP_n^{m^2,\mathbf b},
 \cP_n^{m^2,\mathbf b}\otimes L_{\mathbf d})< 1
\end{equation}
for all $\mathbf d \in S$. 
In particular, $\cP_n^{m^2,\mathbf b}\prec
\cP_n^{m^2,\mathbf b}\otimes L_{\mathbf d}[1]$.
\end{lemma}

\begin{proof}
Let $U$ be as in Lemma~\ref{lem:support} and choose $m_0$ such that 
$m^{-2}(K\cup(K+S))\subset U$ for all integers $m \geq m_0$. 
Then for any such $m$ 
and $\mathbf b \in K$, we have $\mathbf{b}/m^2 \in U$, so that Lemma~\ref{lem:support} gives a stability condition $\sigma^{1,\mathbf b/m^2}_n$. 
If 
$\pi_m \colon E^n \to E^n$ denotes the multiplication by $m$ map, then  by \cite[Lemma 4.3]{LLL} the pullback  
\[\sigma_n^{m^2,\mathbf b}\coloneqq \pi_m^\sharp\sigma_n^{1,\mathbf b/m^2}\]
is a stability condition on $\Db(E^n)$. 
As $\sigma^{1,\mathbf{b}/m^2}_{n}$ is geometric with skyscraper sheaves of points of phase $1$, the same is true of $\sigma_n^{m^2, \mathbf{b}}$. Note that by Grothendieck--Riemann--Roch and the projection formula, for any $\mathbf t\in \RR^n$ we have $v_n(\pi_{m*}F)_I=m^{2|I|}v_n(F)_I$ and $Z_n^{1,\mathbf t}(\pi_{m*}F)=Z_n^{m^2,m^2\mathbf t}(F)$. In particular the central charge of $\sigma_n^{m^2, \mathbf{b}}$ is indeed $Z^{m^2, \mathbf b}_n$. 

For any $\mathbf d \in S$ we have $(\mathbf{b} + \mathbf{d})/m^2 \in U$, so as above 
\[\sigma_n^{m^2,\mathbf b+\mathbf d} \coloneqq   \pi_m^\sharp\sigma_n^{1,(\mathbf b+\mathbf d)/m^2}\]
is a geometric stability condition on $\Db(E^n)$ with skyscraper sheaves of points of phase $1$ and central charge $Z^{m^2, \mathbf{b} + \mathbf{d}}_n$. 
Since the stability condition 
\[\sigma_{n}^{m^2, \mathbf b} \otimes L_{\mathbf d} = (Z_n^{m^2,\mathbf b} \circ (- \otimes L_{\mathbf d}^{-1}), \cP_{n}^{m^2, \mathbf b}\otimes L_{\mathbf d})\]
has the same properties, by \cite[Theorem 4.1(2)]{LLL} we conclude that 
\[\sigma_n^{m^2,\mathbf b}\otimes L_{\mathbf d} =\sigma_n^{m^2,\mathbf b+\mathbf d}.\] 
Thus, if $\cP_n^{m^2,\mathbf b}$ and $\cP_n^{m^2,\mathbf b+\mathbf d}$ denote the slicings of 
$\sigma_n^{m^2,\mathbf b}$ and $\sigma_n^{m^2, \mathbf b + \mathbf d}$, then 
\[\cP_n^{m^2,\mathbf b} \otimes L_{\mathbf d} = \cP_n^{m^2, \mathbf b + \mathbf d}.\]
On the other hand, by \cite[Lemma 3.5(3)]{LLL} and~\eqref{dist-Pn1b-Pn10}, 
we have 
\[\mathrm{dist}(\cP_n^{m^2,\mathbf b}, \cP_n^{m^2,\mathbf b+\mathbf d}) \le\mathrm{dist}(\cP_n^{1,\mathbf b/m^2}, \cP_n^{1,(\mathbf b+\mathbf d)/m^2})<\frac12 + \frac12 = 1,\]
so~\eqref{eq:smalltwist} holds. 
\end{proof}

\subsection{Products of projective spaces} 
Let $n_1,\dots,n_r$ be positive integers. We set
\[\PP\coloneqq\prod_{j=1}^r\PP^{n_j},\]
$n\coloneqq\sum_jn_j$, and $h_j\coloneqq c_1(\cO_{\PP^{n_j}}(1))$. Choose a partition $I_1\sqcup\cdots\sqcup I_r=\{1,\ldots,n\}$ with
$|I_j|=n_j$. 
The quotient of $E$ by $\pm 1$ is $\bP^1$. 
The product of the quotient maps gives a finite flat map
\[q\colon E^n\longrightarrow(\PP^1)^n\longrightarrow \prod_{j=1}^r\operatorname{Sym}^{n_j}\PP^1=\PP\]
such that $q^*h_j=2\sum_{i\in I_j}H_i$ and $\deg q=2^n\prod_jn_j!$.

\begin{proposition}\label{prop:product}
Fix a finite subset $T \subset \ZZ_{>0}$. Then there exists an integer $m_1$, depending only on $T$ and the partition $I_1,\dots,I_r$,  such that for every integer $m\geq m_1$ and element $\mathbf b=(b_j)_{j=1}^r\in\RR^r$, there is a geometric stability condition
\[\sigma_\PP^{m^2/2,\mathbf b}=(Z_\PP^{m^2/2,\mathbf b},\cP_\PP^{m^2/2,\mathbf b})\] 
on $\Db(\PP)$ with $\phi_{\sigma_\PP^{m^2/2,\mathbf b}}(\cO_x)=1$ for all closed points $x \in \PP$ and central charge 
\[Z_\PP^{m^2/2,\mathbf b}(F)\coloneqq -\int_\PP e^{-\sum_j(b_j+\mathfrak{i}m^2/2)h_j}.\ch(F).\]
Moreover, we have
\begin{equation}\label{eq:bayer}
 \begin{aligned}
 \cP_\PP^{m^2/2,\mathbf b}&\prec
 \cP_\PP^{m^2/2,\mathbf b}\otimes \cO_{\PP}(-d,\dots, -d)[1], \\ 
 \cP_\PP^{m^2/2,\mathbf b}&\preceq
 \cP_\PP^{m^2/2,\mathbf b}\otimes\cO_\PP(\mathbf k)
 \end{aligned} 
\end{equation}
for any $d\in T$ and $\mathbf k\in\ZZ_{\ge0}^r$.
\end{proposition}

\begin{proof}
Let $K \subset \RR^n$ be the image of the cube $[0,1]^r$ under the map $\mathbf b \mapsto \mathbf t$  where $\mathbf t=(t_i)^n_{i=1}$ with  $t_i = 2b_j$ for $i\in I_j$, and let  $S \coloneqq \set{ -2d(1,\dots, 1) \sth d \in T} \subset \ZZ^n$. 
Applying Lemma~\ref{lem:enlarge} to $K \subset \RR^n$ and $S \subset \ZZ^{n}$, we obtain an integer $m_1$ depending only on $T$ and the partition $I_1, \dots, I_r$, such that for every integer $m \geq m_1$ and $\mathbf t \in K$, we have a stability condition $\sigma_n^{m^2,\mathbf t}=(Z_n^{m^2,\mathbf t},\cP_n^{m^2,\mathbf t})$ on $\Db(E^n)$ such that 
\begin{equation}
\label{Pm2bn-O-d-d}
    \cP_{n}^{m^2, \mathbf t} \prec \cP_{n}^{m^2, \mathbf t} \otimes q^* \cO_{\PP}(-d, \dots, -d)[1]
\end{equation}
for any $d \in T$. By \cite[Lemma 4.7]{LLL}, we also have
\begin{equation}
\label{Pm2tn-Lell}
\cP_n^{m^2,\mathbf t}\preceq
 \cP_n^{m^2,\mathbf t}\otimes L_{\boldsymbol \ell}
\end{equation}
for any $\boldsymbol{\ell} \in \ZZ^n_{\geq 0}$. 

The next step is to push forward to $\PP$. 
The group $G \coloneqq (\bZ/2)^n \rtimes \prod_{j} \mathfrak{S}_{I_j}$ acts on $E^n$, where each copy of $\bZ/2$ acts by $\pm 1$ on the corresponding factor and the symmetric group $\mathfrak{S}_{I_j}$ acts by permutations on the $I_j$-block of the product. 
Combining the filtrations of \cite[Example 3.20(2)--(3)]{LLL} by pullback and refinement, we find that $q$ satisfies the filtration property (in the sense of \cite[Definition~3.19]{LLL}) with respect to suitable automorphisms $g \in G$ and line bundles of the form $L_{\boldsymbol{\ell}}$ where $\boldsymbol{\ell} \in \ZZ^{n}_{\geq 0}$. 
Moreover, as for $\mathbf t = (t_i) \in K$ the coordinates $t_i$ are all equal for $i \in I_j$, the action of $G$ preserves the central charge of $\sigma_n^{m^2, \mathbf t}$, and hence by \cite[Theorem 4.1(2)]{LLL}, the stability condition $\sigma_n^{m^2, \mathbf t}$ itself. Consequently, we may apply \cite[Proposition 3.21]{LLL} to obtain a stability condition $\sigma_\PP^{m^2/2,\mathbf b}=(Z_\PP^{m^2/2,\mathbf b},\cP_\PP^{m^2/2,\mathbf b})$ on $\Db(\PP)$ defined by $\cP_\PP^{m^2/2,\mathbf b}\coloneqq q_\sharp\cP_n^{m^2,\mathbf t}$ and
\[Z_\PP^{m^2/2,\mathbf b}(F) \coloneqq\frac1{\deg q}Z_n^{m^2,\mathbf t}(q^*F) =-\int_\PP e^{-\sum_j(b_j+\mathfrak{i}m^2/2)h_j}\ch(F).\]
By \cite[Remark 3.11]{LLL}, $\sigma_\PP^{m^2/2,\mathbf b}$ is geometric, and it follows from the construction that skyscraper sheaves of points have phase $1$. 
In view of \cite[Remark~3.14(3) and Lemma~3.9(1)]{LLL}, it follows from~\eqref{Pm2bn-O-d-d} and~\eqref{Pm2tn-Lell} (for $L_{\boldsymbol{\ell}} = q^*\cO_{\PP}(\mathbf k)$) that 
$\sigma_{\PP}^{m^2/2, \mathbf b}$ satisfies~\eqref{eq:bayer}.

Finally, for each $\mathbf k\in\ZZ^r$, we define
\[\sigma_\PP^{m^2/2,\mathbf b+\mathbf k} \coloneqq \sigma_\PP^{m^2/2,\mathbf b}\otimes\cO_\PP(\mathbf k)\]
which extends the construction to all $\mathbf b\in\RR^r$.
\end{proof}

\begin{lemma}\label{lem:embedding}
Let $\iota\colon X\hookrightarrow\PP$ be a closed embedding of a smooth projective variety. 
Set 
\[J \coloneqq \iota^*(h_1 + \cdots + h_r).\]
Then there exists an integer $m_{\iota}$, depending only on the partition $I_1,\dots,I_r$ and the Hilbert polynomial of $X$ with respect to $J$, such that for every integer $m\geq m_{\iota}$ and element $\mathbf b=(b_j)_{j=1}^r\in\RR^r$, 
there is a  geometric stability
condition 
\[\widetilde\sigma_X^{m^2/2,\mathbf b} \coloneqq \iota^\sharp\sigma_\PP^{m^2/2,\mathbf b}\]
on $\Db(X)$ 
with 
$\phi_{\widetilde\sigma_X^{m^2/2,\mathbf b}}(\cO_x) = 1$ for all closed points $x \in X$ 
and central charge
\[\widetilde Z_X^{m^2/2,\mathbf b}(F) \coloneqq -\int_Xe^{-\sum_j(b_j+\mathfrak{i}m^2/2)\iota^*h_j} \td(N_{X/\PP})^{-1}\ch(F).\]
Moreover, we have
\begin{equation}\label{eq:phaseJ}
 \begin{aligned}
  \widetilde\sigma_X^{m^2/2,\mathbf b} & \prec \widetilde\sigma_X^{m^2/2,\mathbf b}\otimes\cO_X(-J)[1], \\ 
 \widetilde\sigma_X^{m^2/2,\mathbf b} & \preceq\widetilde\sigma_X^{m^2/2,\mathbf b}\otimes\cO_X(J). 
  \end{aligned}
\end{equation}
\end{lemma}
\begin{proof}
Choose a resolution of the ideal sheaf $\cI_{X/\PP}$ of length at least $n$ whose terms, except for the leftmost one, are of the form $\cO_{\PP}(-d_i, \dots, -d_i)^{\oplus a_i}$ for some integers $d_i, a_i > 0$. 
Note that the integers $d_i$ appearing in the resolution can be taken to only depend the ambient product $\PP$ and the Hilbert polynomial of $X$ with respect to $J$. 
We apply Proposition~\ref{prop:product} where $T = \{d_i\} \cup \{1\}$ is the collection of degrees which appear, union with $1$. 
This gives an integer $m_{\iota}$ such that for every integer $m \geq m_{\iota}$ and $\mathbf b \in \RR^r$, we have a stability condition $\sigma_{\PP}^{m^2/2, \mathbf b}$ satisfying the properties listed in Proposition~\ref{prop:product}. 

Now we pull back $\sigma_{\PP}^{m^2/2, \mathbf b}$ along $\iota \colon X \hookrightarrow \PP$ by applying \cite[Proposition 3.24]{LLL}. 
Note that by \cite[Example 3.23(3)]{LLL}, $\iota \colon X \hookrightarrow \PP$ satisfies the cofiltration property (in the sense of \cite[Definition~3.22]{LLL}). 
Moreover, hypothesis (i) in \cite[Proposition 3.24]{LLL} follows from~\eqref{eq:bayer} and hypothesis (ii) follows from Lemma~\ref{lem:point}\ref{lem:point:heart} below. We conclude that $\sigma_\PP^{m^2/2,\mathbf b}$ pulls back to a stability condition $\widetilde\sigma_X^{m^2/2,\mathbf b}$ on $X$. 
By \cite[Remark 3.11]{LLL}, $\widetilde\sigma_X^{m^2/2,\mathbf b}$ is geometric with skyscraper sheaves of points of phase $1$. 
By construction and Grothendieck--Riemann-Roch, the central charge of $\widetilde\sigma_X^{m^2/2,\mathbf b}$ is given by the claimed formula. 
Finally, in view of \cite[Remark 3.14(3) and Lemma 3.5(1)]{LLL}, it follows from~\eqref{eq:bayer} that $\widetilde\sigma_X^{m^2/2,\mathbf b}$ satisfies~\eqref{eq:phaseJ}. 
\end{proof}

The following standard observation was used above. 
\begin{lemma}\label{lem:point}
Let $Y$ be a smooth projective variety of dimension $d\ge1$, and let
$\sigma=(Z,\cP)$ be a geometric stability condition on $\Db(Y)$ with $\phi_{\sigma}(\cO_y)=1$ for each closed point $y\in Y$.  
\begin{enumerate}
    \item \label{lem:point:heart} 
    $\cP(0,1]\subset\D_{\Coh(Y)}^{[-d+1,0]}$.

    \item If $F\in \cP(0,1]$ and $d\geq 2$, then $\cH^{-d+1}(F)$ is torsion-free. 
    \item If $F\in \cP(1)$, then $\cH^0(F)\in \Coh_{\le0}(Y)$.
\end{enumerate}
\end{lemma}

\begin{proof}
Let $F\in\cP(\phi)$ be a $\sigma$-stable object with $0<\phi\le1$ that is not the skyscraper sheaf of a point. For each $y\in Y$, the stability of $\cO_y$ and Serre duality give $\Hom^j(F,\cO_y)=0$ when $j<0$ or $j\geq d$ and $\Hom(F,\cO_y)=0$ when $\phi=1$. Therefore, all parts follow from \cite[Proposition 5.4]{Bridgeland–Maciocia}.
\end{proof}

\subsection{A BMT-type inequality}
Let $(X,H)$ be a polarized smooth projective variety of dimension $n\geq3$ over $\kk$. We choose line bundles $D_1,\ldots,D_\rho$ such that $D_1=K_X$, and an integer $N>0$ such that $NH\pm D_i$ are very ample. The complete linear systems $|NH \pm D_i|$ give an embedding $\iota\colon X\hookrightarrow\PP$ into a product 
\[\PP = \prod_{j=1}^{r} \PP^{n_j}\]
of $r = 2\rho$ projective spaces. 
As before, we write $h_j$ for the hyperplane class of the $j$th factor of $\PP$. 
Further, we set $H_j\coloneqq \iota^*h_j$ and $J \coloneqq \sum_jH_j = 2 \rho NH$. Note that 
\[\operatorname{span}_{\RR}\{H_1, \dots, H_r\} =\operatorname{span}_{\RR}\{H,D_1,\ldots,D_\rho\};\]
in particular, $D_1 = K_X$ is contained in the above span.

\begin{lemma}\label{lem:charge}
Assume that $n=3$. Let $m_{\iota}$ be the integer associated to the embedding $\iota \colon X \hookrightarrow \PP$ by Lemma~\ref{lem:embedding}, and set  
\begin{equation}\label{eq:theta}
 \Theta\coloneqq\frac1{12}\ch_2(N_{X/\PP}). 
\end{equation}
For every $B\in\operatorname{span}_{\RR}\{H_j\}$ and every integer $m\geq m_{\iota}$, 
if we set $a = m^2/2$ then the construction of Lemma~\ref{lem:embedding} gives a stability condition
$\sigma_X^{a,B}=(Z^{a,B}_X, \cP^{a,B}_X)$ with central charge
\begin{equation}\label{eq:expanded}
 \begin{aligned}
 Z_X^{a,B}(E)
 &=-v_3^{J,B}(E)-\Theta.\ch_1^B(E)+\frac{a^2}{2}v_1^{J,B}(E)\\
 &  \quad+\mathfrak{i}a\left(v_2^{J,B}(E)- \left( \frac{a^2}{6} - \frac{J.\Theta}{J^3} \right) v_0^J(E)\right). 
 \end{aligned}
\end{equation}
\end{lemma}
\begin{proof}
The normal and Euler sequences give
\[
 \begin{aligned}
 c_1(N_{X/\PP})&=K_X+\sum_i(n_i+1)H_i,\\
 2\ch_2(N_{X/\PP})&=\sum_i(n_i+1)H_i^2-K_X^2+2c_2(X),\\
 \td(N_{X/\PP})^{-1}&=e^{-c_1(N_{X/\PP})/2}(1+\Theta).
 \end{aligned}
\]
Now, we choose $\mathbf b=(b_j)_{j=1}^r \in \RR^r$ such that
\[B=\frac12K_X+\sum_j\left(b_j+\frac{n_j+1}{2}\right)H_j,\]
and set $\sigma_X^{a,B}\coloneqq\widetilde\sigma_X^{a,\mathbf b}$ to be the stability condition from Lemma~\ref{lem:embedding} with parameters $a = m^2/2$ and $\mathbf b$. 
Then
\[Z_X^{a,B}(E)=-\int_Xe^{-B-\mathfrak{i}aJ}(1+\Theta)\ch(E),\]
which expands to \eqref{eq:expanded}.
\end{proof}

\begin{lemma}\label{lem:window}
Assume that $n=3$. In the notation of Lemma~\ref{lem:charge}, 
we have $ \Coh^{J,B}(X)\subset\cP^{a,B}_X(-1,1]$ and $\Coh^{J,B}(X)\cap\cP^{a,B}_X(1)=\Coh_{\le0}(X)$.
\end{lemma}

\begin{proof}
Using Lemma~\ref{lem:point} and \eqref{eq:phaseJ}, the result follows from the same argument as in \cite[Proposition 3.3 and Lemma 3.6]{CF}.
\end{proof}

Now we can finally state the BMT-type inequality at sufficiently large parameters, which is the main result of this section. 

\begin{proposition}\label{prop:weak}
Fix $B_0\in\operatorname{span}_{\QQ}\{H_j\}$ and
$\epsilon\in\QQ_{\ge0}$ such that
\[\frac{1}{J^n}\Del^J_{B_0,\epsilon}(E)\geq J^{n-2}.\left(\ch_1(E)^2-2\rk(E)\ch_2(E)\right)+\epsilon\rk(E)^2J^n\geq 0\] 
holds for all $\mu_J$-semistable torsion-free sheaves $E$ on $X$. 
Choose $m$ as in Lemma~\ref{lem:charge} such that 
\[\frac{m^4}{12} - \frac{2 J^{n-2}.\Theta}{J^n} > \epsilon,\]
and set 
\[A_0 \coloneqq \sqrt{\frac{m^4}{12} - \frac{2J^{n-2}.\Theta}{J^n}} > \sqrt{\epsilon}.\]
Then for every $B\in B_0+\RR J$ and every
$\nu^J_{A_0,B}$-semistable object $E\in\Coh^{J,B}(X)$ with
$\nu^J_{A_0,B}(E)=0$, we have
\[J^{n-3}.(\ch_3^B(E)+\Theta.\ch_1^B(E)) <\frac{m^4}{8}J^{n-1}.\ch_1^B(E).\]
\end{proposition}

\begin{proof}
The general case is proved in \cite{LLSZ}. Here, we give a proof for $n=3$. By Lemma~\ref{lem:window}, the intersections of
$\Coh^{J,B}(X)$ with $\cP^{a,B}_X(0,1]$ and $\cP^{a,B}_X(-1,0]$ form a
torsion pair. 
The corresponding decomposition of $E$ takes the form 
\[0\longrightarrow E_1\longrightarrow E\longrightarrow E_2 \longrightarrow0\]
where $E_1\in \Coh^{J,B}(X) \cap \cP^{a,B}_X(0,1]$ and $E_2\in \Coh^{J,B}(X) \cap \cP^{a,B}_X(-1,0]$. 
If $E_1\ne0$, the semistability of $E$
gives $v_1^{J,B}(E_1)>0$ and $\nu^J_{A_0,B}(E_1)\leq 0$. Therefore, we get
\[0\le\Im Z_X^{a,B}(E_1)=a v_1^{J,B}(E_1) \nu^J_{A_0,B}(E_1)\leq 0.\]
Hence \[E_1\in\Coh^{J,B}(X)\cap\cP^{a,B}_X(1).\] 
By Lemma~\ref{lem:window}, we get $E_1\in \Coh_{\le0}(X)$, contradicting $v_1^{J,B}(E_1)>0$. Thus $E_1=0$ and $E[1]\in\cP_X^{a,B}(0,1]$. Since $\Im Z_X^{a,B}(E)=0$, \eqref{eq:expanded} gives \[0<\Re Z_X^{a,B}(E) =\frac{m^4}{8}J^2.\ch_1^B(E)-\ch_3^B(E)-\Theta.\ch_1^B(E). \qedhere \]
\end{proof}

\section{Frobenius estimates}\label{sec:frobenius}

Let $(X,H)$ be a polarised smooth projective variety of dimension $n\geq3$ over an algebraically closed field $\kk$ of characteristic $p\geq5$. In this section, we prove bounds on the slope of HN factors of Frobenius pullback of a tilt-semistable object (cf.~Theorem~\ref{thm:twoendpoints}), which is the most technical part of our paper.

\subsection{Normalized Frobenius pullback}\label{sec:normalization}
Fix $B_0\in\NS(X)_{\QQ}$ and a line bundle $L$ with
$c_1(L)=qB_0$, where $q\in\mathbb Z_{>0}$. For any $B\in \NS(X)_{\QQ}$, we write
\[B^\sharp\coloneqq K_{X}-B\in \NS(X)_{\QQ}.\]
Fix $\epsilon\in \QQ_{\geq 0}$ such that
\[\Del^H_{B,\epsilon}(E)\geq 0\]
for any $B\in \NS(X)_{\QQ}$ and any $\mu_H$-semistable sheaf $E$.

Suppose $p\equiv1\pmod{q}$, and let $F\colon X\to X$ be the absolute
Frobenius morphism. We set $M_p\coloneqq L^{-(p-1)/q}$ and $\widehat B_{\beta}\coloneqq B_0+p\beta H$.
For any $E\in \Db(X)$, we define
\[\Phi_p(E)\coloneqq F^*E\otimes M_p,\]
$N_p\coloneqq \omega_{X}^{1-p}\otimes M_p^{-1}$, and $\Psi_{p}(E)\coloneqq F^*E\otimes N_p$. The functor $\Phi_p$ is called the \emph{normalized Frobenius pullback associated with $B_0$.}

Note that $qB_0^\sharp=c_1(\omega_X^q\otimes L^{-1})$ and $N_p=(\omega_X^q\otimes L^{-1})^{-(p-1)/q}$. In particular, $c_1(N_p)=(1-p)B_0^\sharp$.

Recall that we can factor $F$ as 
\[F\colon X\xrightarrow{F_X}X^{(p)}\xrightarrow{h}X,\]
where $X^{(p)}\coloneqq X\times_{\operatorname{Spec}\kk,F_{\kk}}
\operatorname{Spec}\kk$ and $F_X$ is the relative Frobenius morphism of $X$, $h$ is the first projection, and $F_{\kk}\colon \Spec(\kk)\to \Spec(\kk)$ is the Frobenius map of $\kk$. Since $\kk$ is perfect, $h$ is an isomorphism of schemes, although
generally not of $\kk$-schemes. According to \cite[Tags 0CC6 and 0EC0]{Stacks}, both $F$ and $F_X$ are finite flat.

We define an anti-autoequivalence of $\Db(X)$ by
\[\DD(E)\coloneqq \mathbf R\mathcal Hom(E,\omega_X)[1]\]
for any $E\in \Db(X)$.

\begin{proposition}
We have:
\begin{enumerate}
    \item the left adjoint of $\Phi_p$ is
\begin{equation}\label{eq:left}
\mathsf L_p\coloneqq F_*(-\otimes N_p),
\end{equation}

\item there is a natural isomorphism
\begin{equation}\label{eq:dual}
\DD\circ \Phi_p\simeq\Psi_p\circ \DD,
\end{equation}
and

\item for any $E\in\Db(X)$ and $0\leq i\leq n$, we have
\begin{equation}\label{eq:scaling}
\ch_i^{\widehat B_{\beta}}(\Phi_pE)=p^i\ch_i^{B_0+\beta H}(E).
\end{equation}
\end{enumerate}
\end{proposition}

\begin{proof}
We know that $F_X^!E\simeq F_X^*E\otimes F_X^!\cO_{X^{(p)}}$ and $F_X^!\cO_{X^{(p)}}\simeq
\omega_X\otimes\bigl(F_X^*\omega_{X^{(p)}}\bigr)^{-1}$ for any object $E\in\Db(X^{(p)})$. Moreover, since $h$ is an isomorphism, we have $\omega_{X^{(p)}}\simeq h^*\omega_X$ and $h^!=h^*$ and hence
\[\begin{aligned}
F^!E&\simeq F_X^!h^*E\\
&\simeq F^*E\otimes\omega_X\otimes(F^*\omega_X)^{-1}
\simeq F^*E\otimes\omega_X^{1-p}.
\end{aligned}\]
Combining with Serre duality and $N_p\otimes M_p\simeq\omega_X^{1-p}$
give natural isomorphisms
\[\begin{aligned}
\Hom(F_*(G\otimes N_p),E)
&\simeq\Hom(G\otimes N_p,F^*E\otimes\omega_X^{1-p})\\
&\simeq\Hom(G,F^*E\otimes M_p).
\end{aligned}\]
for any $E,G\in \Db(X)$. This proves \eqref{eq:left}.

For any $E\in \Db(X)$, we have
\[\mathbf R\mathcal Hom(F^*E,\cO_X) \simeq F^*\mathbf R\mathcal Hom(E,\cO_X).\]
Since $\DD(E)=\mathbf R\mathcal Hom(E,\omega_X)[1]$ and
$\omega_X^p\otimes N_p\simeq\omega_X\otimes M_p^{-1}$,
both $\DD\Phi_p(E)$ and $\Psi_p\DD(E)$ identify naturally with
\[F^*\mathbf R\mathcal Hom(E,\cO_X) \otimes\omega_X\otimes M_p^{-1}[1]\]
which proves \eqref{eq:dual}.

Finally, $c_1(M_p)-\widehat B_{\beta}=-p(B_0+\beta H)$ and
$\ch_i(F^*E)=p^i\ch_i(E)$. Taking the degree-$i$ component of
\[\ch^{\widehat B_{\beta}}(\Phi_pE) =e^{-p(B_0+\beta H)}.\ch(F^*E)\]
proves \eqref{eq:scaling}.
\end{proof}

Now, we fix some notations and state our main result in this section. Choose integers $u,r_0\in \ZZ_{>0}$ and $v\in \ZZ_{\geq 0}$ such that
\begin{equation}\label{eq:gg}
\Omega_{X}(uH),\quad T_{X}(r_0H),\quad
\omega_{X}^{-1}(vH)
\quad\text{are globally generated}.
\end{equation}
Write $D=nu+v$.
We also choose $k_{\Omega}$ and $k_{\omega^{-1}}$ such that $K_{\Omega}(k_{\Omega}H)$ and $K_{\omega^{-1}}(k_{\omega^{-1}}H)$ are globally generated, where $$K_{\Omega}\coloneqq \ker(\cO_X(-uH)^{\oplus h^0(\Omega_X(uH))}\to \Omega_{X})$$ and $$K_{\omega^{-1}}\coloneqq \ker(\cO_X(-vH)^{\oplus h^0(\omega^{-1}_X(vH))}\to \omega_{X}^{-1})$$
are kernels of evaluation maps. For simplicity, we define
\[k_*=\max\{u,v,k_\Omega,k_{\omega^{-1}}\}.\]

\begin{theorem}\label{thm:upper}\label{thm:twoendpoints}
\setlength{\emergencystretch}{2em}
Let $A\ge2D$ and $\frac{2D(p-1)}{p}<a\le A$. Suppose that $p\equiv1\pmod{q}$ and
\begin{equation}\label{eq:upper-explicitprime}
 p>\max\left\{4,\frac{k_*+\sqrt\epsilon}{A-D},
                   \frac12+\frac{\epsilon}{8D^2}\right\}.
\end{equation}
Let $\beta\in\RR$ and let
$E\in\Coh^{H,B_0+\beta H}(X)$ be a
$\nu^H_{A,B_0+\beta H}$-semistable object with $\nu^H_{A,B_0+\beta H}(E)=0$. Suppose that
\begin{equation}\label{eq:sourcegaps}
 \begin{aligned}
 \mu_H^+(\cH^{-1}(E))
 &\le\frac{(B_0+\beta H).H^{n-1}}{H^n}-a,\\
 \mu_H^-(\cH^0(E))
 &\ge\frac{(B_0+\beta H).H^{n-1}}{H^n}+a.
 \end{aligned}
\end{equation}
Then $\Phi_pE\in\Coh^{H,\widehat B_\beta}(X)$ and
\[-pB_p(A)\le\nu^{H,-}_{pA,\widehat B_\beta}(\Phi_pE) \le\nu^{H,+}_{pA,\widehat B_\beta}(\Phi_pE)\le pB_p(A),\]
where
\begin{equation}\label{eq:upper}
 B_p(A)\coloneqq D(1-p^{-1})+
 \frac{D^2(1-p^{-1})^2}
 {2\bigl(\sqrt{A^2-\epsilon p^{-2}}-D(1-p^{-1})\bigr)}.
\end{equation}
\end{theorem}

We will prove the above upper bound in Section~\ref{sec:upper-proof} and
deduce the lower bound by duality in Section~\ref{sec:duality}.

\subsection{Filtrations induced by Frobenius}

Motivated by \cite{Sun}, in the following, we define a filtration of $\Phi_p\mathsf L_p(E)$ for any $E\in \Db(X)$. This is the starting point of our estimates. For any bundle $E$ on $X$ and $\ell\in \ZZ_{\ge 0}$, we denote by $\T^{\ell}(E)\subset E^{\otimes \ell}$ the \emph{truncated symmetric power} defined in \cite[Definition 3.4]{Sun}, which is also a vector bundle. According to \cite[Section 3]{Sun}, we have a quotient map
\[\Sym^{\ell}(E)\twoheadrightarrow \T^{\ell}(E)\]
which is an isomorphism for $0<\ell<p$. Moreover, $\T^{\ell}(E)=0$ for $\ell>\rk(E)(p-1)$.

\begin{lemma}\label{lem:graded}
For every $E\in\Db(X)$, there is a sequence of morphisms
\[0=\V_{n(p-1)+1}(E)\to \V_{n(p-1)}(E)\to \cdots \to \V_1(E)\to \V_0(E)=\Phi_p\mathsf L_p(E)\]
such that
\[\mathrm{cone}(\V_{j+1}(E)\to \V_{j}(E))\simeq E\otimes\omega_{X}^{1-p}\otimes \T^j(\Omega_{X})\]
for any $0\le j\le n(p-1)$.
\end{lemma}

\begin{proof}
Since $h$ is an isomorphism, we have $F^*F_*\simeq F_X^*F_{X,*}$ and $Z\coloneqq X\times_{F,X,F}X
\simeq X\times_{F_X,X^{(p)},F_X}X$. Write $\pi_1,\pi_2\colon Z\to X$ for the projections,
$\Delta\colon X\to Z$ for the diagonal, and
$\mathcal I=\ker(\cO_Z\to\Delta_*\cO_X)$. 

For any $E\in \Db(X)$, we define $$\V_j(E)\coloneqq \pi_{1*}(\cI^j\otimes \pi_2^*(E\otimes N_p))\otimes M_p.$$
Then the exact sequence $0\to \cI^{j+1}\to \cI^j\to \cI^j/\cI^{j+1}\to 0$ induces a morphism $\V_{j+1}(E)\to \V_j(E)$ with
\[\mathrm{cone}(\V_{j+1}(E)\to \V_j(E))\simeq \pi_{1*}(\cI^j/\cI^{j+1}\otimes \pi_2^*(E\otimes N_p))\otimes M_p.\]
Now the result follows from \cite[Lemma 3.6, Theorem 3.7]{Sun} as $M_p\otimes N_p\simeq \omega_X^{1-p}$.
\end{proof}

Using the above filtration, we have the following bounds.

\begin{lemma}\label{lem:coarse}
Set
\[e_p\coloneqq\frac{r_0\bigl(nu+\frac{K_X.H^{n-1}}{H^n}\bigr)}{u+r_0}(p-1).\]
Then $0\le e_p\le\min\{nr_0,nu+\frac{K_X.H^{n-1}}{H^n}\}(p-1)
\le D(p-1)$, and for every nonzero torsion-free sheaf $E$, we have
\begin{align}
\mu_{H}^+(F^*E)&\le p\mu_{H}^+(E)+e_p,
&\mu_{H}^-(F^*E)&\ge p\mu_{H}^-(E)-e_p,
\label{eq:pull}\\
\mu_{H}^+(F_*E)&\le\frac{\mu_{H}^+(E)+e_p}{p},
&\mu_{H}^-(F_*E)&\ge
\frac{\mu_{H}^-(E)-e_p+(p-1)\frac{K_X.H^{n-1}}{H^n}}{p}.
\label{eq:push}
\end{align}
\end{lemma}

\begin{proof}
Taking determinants in
\eqref{eq:gg} gives
\[-nu\le \frac{K_X.H^{n-1}}{H^n}\le\min\{v,nr_0\},\]
which proves the asserted bounds on $e_p$. For
$0\le j\le n(p-1)$, the global generation of $T_X(r_0H)$ gives
inclusions
\[\T^j(\Omega_X)\hookrightarrow(\Omega_X)^{\otimes j} \hookrightarrow\cO_X(jr_0H)^{\oplus n_j}\]
with locally free cokernels. On the other hand, multiplication in the truncated symmetric algebra gives a perfect pairing
\[\T^j(\Omega_X)\otimes \T^{n(p-1)-j}(\Omega_X) \longrightarrow\omega_X^{p-1}\]
by \cite[Proof of Lemma 3.4]{Donkin-Geranios}. Since
$\T^{n(p-1)-j}(\Omega_X)((n(p-1)-j)uH)$ is globally
generated, dualizing its evaluation map gives an inclusion
\[\T^j(\Omega_X)\hookrightarrow \bigl(\omega_X^{p-1}((n(p-1)-j)uH)\bigr)^{\oplus m_j}\]
with locally free cokernel. Consequently,
\[
 \begin{aligned}
 \mu_H^+(E\otimes\T^j(\Omega_X))-\mu_H^+(E)
 &\le\min\left\{jr_0,(p-1)\left(nu+\frac{K_X.H^{n-1}}{H^n}\right)-ju\right\}\\
 &\le\frac{u}{u+r_0}jr_0
 +\frac{r_0}{u+r_0}\left((p-1)\left(nu+\frac{K_X.H^{n-1}}{H^n}\right)-ju\right)
 =e_p.
 \end{aligned}
\]

By Lemma~\ref{lem:graded}, $F^*F_*E$ has a filtration with
factors $E\otimes\T^j(\Omega_X)$. Pulling back the first HN
factor of $F_*E$ therefore gives
\[p\mu_H^+(F_*E)\le\mu_H^+(F^*F_*E) \le\mu_H^+(E)+e_p,\]
proving the upper bound in \eqref{eq:push}. Since
\[(F_*E)^\vee\simeq F_*(E^\vee\otimes\omega_X^{1-p}),\]
applying the upper bound to $E^\vee\otimes\omega_X^{1-p}$
and using $\mu_H^+(E^\vee)=-\mu_H^-(E)$ proves the lower bound.

Let $G$ be the first HN factor of $F^*E$. Adjunction gives a nonzero map $F_*(G\otimes\omega_X^{1-p})\to E$. Since $G$
is $\mu_H$-semistable, the lower bound in \eqref{eq:push} gives
\[\frac{\mu_H(G)-e_p}{p} \le\mu_H^-(F_*(G\otimes\omega_X^{1-p})) \le\mu_H^+(E).\]
This proves the upper bound in \eqref{eq:pull}. Applying it
to $E^\vee$ and using $F^*(E^\vee)\simeq(F^*E)^\vee$ proves
the lower bound.
\end{proof}

\begin{lemma}\label{lem:geometricBG}
With the fixed choices in \eqref{eq:gg}, every slope-semistable
torsion-free sheaf $W$ satisfies
\[H^{n-2}.\bigl(\ch_1(W)^2-2\rk(W)\ch_2(W)\bigr) +\left(\frac{r_0\bigl(nu+\frac{K_X.H^{n-1}}{H^n}\bigr)}{u+r_0}\right)^2 H^n\rk(W)^2\ge0.\]
\end{lemma}

\begin{proof}
Iterating \eqref{eq:pull} for a slope-semistable $W$ gives
\[
 \begin{aligned}
 \lim_{e\to\infty}\frac{\mu_H^+((F^e)^*W)}{p^e}-\mu_H(W)
 &\le\frac{r_0\bigl(nu+\frac{K_X.H^{n-1}}{H^n}\bigr)}{u+r_0},\\
 \mu_H(W)-\lim_{e\to\infty}\frac{\mu_H^-((F^e)^*W)}{p^e}
 &\le\frac{r_0\bigl(nu+\frac{K_X.H^{n-1}}{H^n}\bigr)}{u+r_0}.
 \end{aligned}
\]
Thus \cite[Theorem 5.1, equation (5.1.1)]{Langer}, with slopes
divided by $H^n$, proves the asserted inequality.
\end{proof}

Since
\[D\geq \frac{r_0\bigl(nu+\frac{K_X.H^{n-1}}{H^n}\bigr)}{u+r_0},\]
in the following, we fix $\epsilon$ so that $D\geq \sqrt{\epsilon}$.

\subsection{Frobenius and the tilted heart}

Now, we study the action of the Frobenius map on the tilted heart.

\begin{lemma}\label{lem:strip}
Let $(\alpha,\beta)\in \RR_{>0}\times \RR$, $c\ge0$, and let $E$ be a nonzero torsion-free sheaf with
\[\frac{(B_0+\beta H). H^{n-1}}{H^n}-c\le\mu_{H}^-(E)\le\mu_{H}^+(E)\le \frac{(B_0+\beta H). H^{n-1}}{H^n}.\]
If $\alpha^2>c^2+\epsilon$, then $E[1]\in\Coh^{H,B_0+\beta H}(X)$ and $\nu^{H,-}_{\alpha, B_0+\beta H}(E[1])>0$.
\end{lemma}

\begin{proof}
By definition, it is clear that $E[1]\in\Coh^{H,B_0+\beta H}(X)$. If the last HN factor $Q$ of $E[1]$ with respect to $\nu^{H}_{\alpha, B_0+\beta H}$ satisfies $\nu^{H}_{\alpha, B_0+\beta H}(Q)=\mu\leq 0$, then Lemma~\ref{lem:2.5gap} gives
\[\begin{aligned}
\mu_H^+(\cH^{-1}(Q))
\le& \frac{(B_0+\beta H).H^{n-1}}{H^n}+\mu-\sqrt{\alpha^2-\epsilon+\mu^2}\\
\le& \frac{(B_0+\beta H).H^{n-1}}{H^n}-\sqrt{\alpha^2-\epsilon}\\
<&\frac{(B_0+\beta H).H^{n-1}}{H^n}-c\\
\le&\mu_H^-(E)
\end{aligned}\]
when $\cH^{-1}(Q)\ne0$. Hence, we always have
\[\Hom(E[1],Q)=\Hom(E,\cH^{-1}(Q))=0,\]
which gives a contradiction.
\end{proof}

\begin{lemma}\label{lem:defect}
Let $\beta\in \RR$. Then for any $E\in\Coh^{H,B_0+\beta H}(X)$, we have
\[\cH^i_{\Coh^{H,\widehat B_{\beta}}(X)}(\Phi_pE)=0\]
for $i\geq 2$ and 
\[\cH^1_{\Coh^{H,\widehat B_{\beta}}(X)}(\Phi_pE)=G[1],\]
where $G$ is torsion-free and, if nonzero,
\[\frac{\widehat B_{\beta}. H^{n-1}}{H^n}-e_p\le\mu_{H}^-(G)\le\mu_{H}^+(G)\le \frac{\widehat B_{\beta}. H^{n-1}}{H^n}.\]
In particular, if $\alpha^2>e_p^2+\epsilon$, then
\[\nu^{H,-}_{\alpha,\widehat B_{\beta}}(\cH^1_{\Coh^{H,\widehat B_{\beta}}(X)}(\Phi_pE))>0.\]
\end{lemma}

\begin{proof}
From the definition of $\Phi_p$, it is clear that $\cH^i(\Phi_pE)=0$ for $i\notin \{0,-1\}$. Thus, from the construction of $\Coh^{H,\widehat B_{\beta}}(X)$, we see that $\cH^i_{\Coh^{H,\widehat B_{\beta}}(X)}(\Phi_pE)=0$ for $i\geq 2$. Moreover, $\cH^1_{\Coh^{H,\widehat B_{\beta}}(X)}(\Phi_pE)=G[1]$, where $G$ is the part of the HN filtration of $\cH^0(\Phi_pE)$ such that $\mu_H^+(G)\leq \frac{\widehat B_{\beta}. H^{n-1}}{H^n}$. Now, the remaining part follows from \eqref{eq:pull} and Lemma~\ref{lem:strip}.
\end{proof}

\subsection{Tensor products}

\begin{lemma}\label{lem:tensor}
Let $\alpha>\sqrt\epsilon$, $\beta\in \RR$ and let
$E\in\Coh^{H,B_0+\beta H}(X)$ be
$\nu^H_{\alpha,B_0+\beta H}$-semistable with $\infty>\nu^H_{\alpha,B_0+\beta H}(E)=\mu>0$.
If $t\in\ZZ_{\ge0}$ satisfies
\[\sqrt{\alpha^2-\epsilon+\mu^2}+t-\mu<\sqrt{\alpha^2-\epsilon},\]
then $E(-tH)\in \Coh^{H,B_0+\beta H}(X)$ and $\nu^{H,-}_{\alpha,B_0+\beta H}(E(-tH))>0$.
\end{lemma}

\begin{proof}
Set
\[\theta\coloneqq \sqrt{\alpha^2-\epsilon+\mu^2}-\mu.\]
Note that the hypothesis implies
$t<\sqrt{\alpha^2-\epsilon}$ and $t<\mu$. Therefore, Lemma~\ref{lem:2.5gap} gives $E\in\Coh^{H,B_0+(\beta+t)H}(X)$, hence we have $E(-tH)\in\Coh^{H,B_0+\beta H}(X)$. It remains to show $\nu^{H,-}_{\alpha,B_0+\beta H}(E(-tH))>0$.

Let $G$ be the last HN factor of $E(-tH)$ with respect to $\nu^{H}_{\alpha,B_0+\beta H}$. If $\nu^{H}_{\alpha,B_0+\beta H}(G)=\lambda\leq 0$, then $G(tH)$ is $\nu^{H}_{\alpha,B_0+(\beta+t) H}$-semistable with $\nu^{H}_{\alpha,B_0+(\beta+t) H}(G(tH))=\lambda\leq 0$. Therefore, by Lemma~\ref{lem:nestedwall}, $E$ and $G(tH)$ remain
semistable along their respective semicircles
\[\begin{aligned}
 (\alpha')^2&=\alpha^2+2\mu(\beta'-\beta)-(\beta'-\beta)^2,\\
 (\alpha')^2&=\alpha^2+2\lambda(\beta'-\beta-t)-(\beta'-\beta-t)^2,
\end{aligned}\]
where $(\alpha')^2>\epsilon$. By computing $\alpha'$ at $\beta'=\beta-\theta$ and $\beta'=\beta+t$, we see that these semicircles intersect at a point $(\alpha_1,\beta_1)$ with
$\beta-\theta<\beta_1\le \beta+t$ and  $\alpha_1>\sqrt\epsilon$. In particular, $E$ and $G(tH)$ are both $\nu^H_{\alpha_1,B_0+\beta_1H}$-semistable with slopes 
\[\nu^H_{\alpha_1,B_0+\beta_1H}(E)=\mu-\beta_1+\beta>\nu^H_{\alpha_1,B_0+\beta_1H}(G(tH))=\lambda+t-\beta_1+\beta,\]
which contradicts the existence of a nonzero map $E\to G(tH)$.
\end{proof}

For any $0\le\ell\le n(p-1)$, the tensor product induces an exact
sequence of vector bundles
\[0\longrightarrow R_\ell\longrightarrow \cO_X(-(\ell u +(p-1)v)H)^{\oplus N_\ell} \longrightarrow \T^\ell(\Omega_X)\otimes\omega_X^{1-p}\longrightarrow0,\]
where
$N_\ell=h^0(\Omega_X(uH))^\ell \cdot 
h^0(\omega_X^{-1}(vH))^{p-1}$. Next, we estimate the change of slopes after tensoring with $R_{\ell}$.

For $E\in \Coh(X)$, we denote its torsion-free part as $E_{\mathrm{tf}}$.

\begin{lemma}\label{lem:truncatedkernel}
For every $0\le\ell\le n(p-1)$ and every nonzero coherent sheaf $E$, if $R_\ell\ne0$, then we have
\[\begin{split}
\mu_H^-(E\otimes R_\ell)&\ge\mu_H^-(E)-\ell u -(p-1)v-k_*,\\
\mu_H^+(E_{\mathrm{tf}}\otimes R_\ell)&\le\mu_H^+(E_{\mathrm{tf}})-\ell u -(p-1)v.
\end{split}\]
\end{lemma}

\begin{proof}
Only in this proof, we write $P_\Omega=\cO_X(-uH)^{\oplus h^0(\Omega_X(uH))}$, $P_\omega=\cO_X(-vH)^{\oplus h^0(\omega_X^{-1}(vH))}$, and $d_\ell=\ell u+(p-1)v$, $m=\ell+p-1$. Consider the successive evaluation maps
\[W_0\twoheadrightarrow W_1\twoheadrightarrow\cdots \twoheadrightarrow W_m\twoheadrightarrow W_{m+1},\]
where
\[
\begin{aligned}
W_i&=(\Omega_X)^{\otimes i}
 \otimes P_\Omega^{\otimes(\ell-i)}
 \otimes P_\omega^{\otimes(p-1)},\\
W_{\ell+j}&=(\Omega_X)^{\otimes\ell}
 \otimes\omega_X^{-j}
 \otimes P_\omega^{\otimes(p-1-j)},\\
W_{m+1}&=\T^\ell(\Omega_X)\otimes\omega_X^{1-p},
\end{aligned}
\]
for each $0\le i\le\ell$ and $0\le j\le p-1$.
Thus $W_0=\cO_X(-d_\ell H)^{\oplus N_\ell}$ and $R_\ell=\ker(W_0\twoheadrightarrow W_{m+1})$. If we set $F_i\coloneqq \ker(W_0\twoheadrightarrow W_i)$ and $V_i\coloneqq F_i/F_{i-1}$, then we have a filtration of subbundles
\[0=F_0\subset F_1\subset\cdots\subset F_{m+1}=R_\ell\]
with $V_i\simeq\ker(W_{i-1}\twoheadrightarrow W_i)$. Explicitly, for each $1\le i\le\ell$ and $1\le j\le p-1$, we have
\[
\begin{aligned}
V_i&\simeq
(\Omega_X)^{\otimes(i-1)}
\otimes K_\Omega
\otimes P_\Omega^{\otimes(\ell-i)}
\otimes P_\omega^{\otimes(p-1)},\\
V_{\ell+j}&\simeq
(\Omega_X)^{\otimes\ell}
\otimes\omega_X^{-(j-1)}
\otimes K_{\omega^{-1}}
\otimes P_\omega^{\otimes(p-1-j)},\\
V_{m+1}&\simeq
C_\ell\otimes\omega_X^{1-p},
\end{aligned}
\]
where $C_\ell\coloneqq \ker\bigl(
(\Omega_X)^{\otimes\ell}\twoheadrightarrow
\T^\ell(\Omega_X)\bigr).$

We claim that $C_\ell(\ell uH)$ is globally generated. To check this, for any bundle $V$, we set $K_\ell(V)
\coloneqq \ker\bigl(V^{\otimes\ell}\twoheadrightarrow\T^\ell(V)\bigr)$ and $n_{\Omega}\coloneqq h^0(\Omega_X(uH))$, and let $q\colon\cO_X^{\oplus n_{\Omega}}\twoheadrightarrow\Omega_X(uH)$ be the evaluation map. Note that from the functoriality of $\T^\ell$, $K_{\ell}$ is also functorial. Therefore, if we choose an affine open cover $\{U\}$ of $X$ that trivializes $\Omega_X$, we get a splitting of $q|_U$ on each $U$:
\[s_U\colon\Omega_X(uH)|_U\longrightarrow\cO_U^{\oplus n_{\Omega}}.\]
Now, the functoriality of $K_{\ell}$ gives
\[K_\ell(q)|_U\circ K_\ell(s_U) =\mathrm{id}.\]
Consequently,
\[K_{\ell}(q)\colon \cO_X\otimes_{\kk} \ker\bigl((\kk^{n_{\Omega}})^{\otimes\ell} \twoheadrightarrow\T^\ell(\kk^{n_{\Omega}})\bigr) \twoheadrightarrow K_\ell(\Omega_X(uH)) \simeq C_\ell(\ell uH)\]
is surjective and proves the claim.

Next, we define
\[b_i\coloneqq 
\begin{cases}
d_\ell-u+k_\Omega,&1\le i\le\ell,\\
d_\ell-v+k_{\omega^{-1}},&\ell+1\le i\le m,\\
d_\ell,&i=m+1.
\end{cases}\]
The assumed global generation and the preceding claim show that each $V_i(b_iH)$ is globally generated, hence we obtain a surjection
\[H^0(X,V_i(b_iH))\otimes_{\kk}\cO_X \twoheadrightarrow V_i(b_iH)\]
and $b_i\le d_\ell+k_*$.
Therefore, tensoring with $E$ implies
\[\mu_H^-(E\otimes V_i) \ge \mu_H^-(E)-b_i \ge \mu_H^-(E)-d_\ell-k_*\]
for each $V_i\neq 0$. Since each $V_i$ is locally free,
\[0\longrightarrow E\otimes F_{i-1} \longrightarrow E\otimes F_i \longrightarrow E\otimes V_i \longrightarrow0\]
is exact. Thus
\[\mu_H^-(E\otimes R_\ell) \ge\min_{V_i\ne0}\mu_H^-(E\otimes V_i) \ge\mu_H^-(E)-d_\ell-k_*.\]

Finally, local freeness of $W_{m+1}$ gives an exact sequence
\[0\longrightarrow E_{\mathrm{tf}}\otimes R_\ell \longrightarrow E_{\mathrm{tf}}(-d_\ell H)^{\oplus N_\ell} \longrightarrow E_{\mathrm{tf}}\otimes W_{m+1} \longrightarrow0.\]
Therefore,
\[\mu_H^+(E_{\mathrm{tf}}\otimes R_\ell) \le\mu_H^+\bigl(E_{\mathrm{tf}}(-d_\ell H)^{\oplus N_\ell}\bigr) =\mu_H^+(E_{\mathrm{tf}})-d_\ell.\]
This finishes the proof.
\end{proof}

\subsection{The upper bound}\label{sec:upper-proof}

Now, we are ready to prove the upper bound in Theorem~\ref{thm:upper}.

\begin{proof}[Proof of the upper bound in Theorem~\ref{thm:upper}]
In this proof, we set
\[\alpha_0=\sqrt{p^2A^2-\epsilon},\qquad c_p=D(p-1), \qquad \mathcal A=\Coh^{H,\widehat B_\beta}(X).\]
Let $\mathcal T\subset\mathcal A$ be the full subcategory consist of zero and the objects $E$ with $\nu^{H,-}_{pA,\widehat B_\beta}(E)>0$.
Thus, $\cT$ is closed under quotients and extensions in $\mathcal A$.

The explicit bound \eqref{eq:upper-explicitprime} gives
\[\alpha_0-c_p\ge p(A-D)+D-\sqrt\epsilon>k_*,\]
and
\[\alpha_0^2-4D^2(p-1)^2 =p^2(A^2-4D^2)+4D^2(2p-1)-\epsilon>0.\]
Moreover, Lemma~\ref{lem:coarse} and \eqref{eq:gg} give
\[2e_p-(p-1)\frac{K_X.H^{n-1}}{H^n} \le(p-1)(2nu+\frac{K_X.H^{n-1}}{H^n}) \le(p-1)(2nu+v)<pa.\]
Consequently,
\begin{equation}\label{eq:prime}
 \alpha_0-c_p>k_*,\qquad
 \alpha_0>2c_p\ge2e_p,\qquad
 pa>2e_p-(p-1)\frac{K_X.H^{n-1}}{H^n}.
\end{equation}

\medskip

\textbf{Step 1}. \emph{$\Phi_pE\in \cA$.}

The assumed inequalities \eqref{eq:sourcegaps},
together with \eqref{eq:pull} and $c_1(M_p)=(1-p)B_0$, give
\begin{equation}\label{eq:Mgaps}
\begin{split}
\mu_H^+(\cH^{-1}(\Phi_pE))
&\le\frac{\widehat B_\beta. H^{n-1}}{H^n}
-pa+e_p
<\frac{\widehat B_\beta. H^{n-1}}{H^n},\\
\mu_H^-(\cH^0(\Phi_pE))
&\ge\frac{\widehat B_\beta. H^{n-1}}{H^n}
+pa-e_p
>\frac{\widehat B_\beta. H^{n-1}}{H^n}.
\end{split}
\end{equation}
Thus, $\Phi_pE\in\mathcal A$.

\medskip

\textbf{Step 2}. \emph{$\nu^{H,+}_{pA,\widehat B_\beta}(\Phi_pE)<+\infty$ and $\nu^H_{pA,\widehat B_\beta}(\Phi_pE)=0$.}

For every sheaf $T$ with $\dim T\le n-2$, adjunction gives
\[\Hom(T,\Phi_pE)\simeq\Hom(F_*(T\otimes N_p),E)=0.\]
The first inequality in \eqref{eq:Mgaps} also excludes maps from $\mu_H$-semistable sheaves of slope $\widehat B_\beta. H^{n-1}/H^n$ shifted by one.
Hence, $\nu^{H,+}_{pA,\widehat B_\beta}(\Phi_pE)<+\infty$, and \eqref{eq:scaling}
gives
\[\nu^H_{pA,\widehat B_\beta}(\Phi_pE)=0.\]

\medskip

Let $G$ be the first HN factor of $\Phi_pE$ with respect to $\nu^H_{pA,\widehat B_\beta}$. If $\nu^H_{pA,\widehat B_\beta}(G)\leq 0$, then the desired upper bound follows. So we may assume that $\nu^H_{pA,\widehat B_\beta}(G)=\mu>0$. In the following, suppose for a contradiction that
\[\mu>pB_p(A).\]
Then we have
\begin{equation}\label{eq:bad}
\sqrt{\alpha_0^2+\mu^2}-\mu+c_p<\alpha_0.
\end{equation}

\medskip

\textbf{Step 3}. \emph{We have $\mathsf L_pG\in \Coh^{H, B_0+\beta H}(X)$ and
\begin{equation}\label{eq:positiveP}
\Phi_p\mathsf L_pG\in \D_{\mathcal A}^{\le0},\qquad
\cH^0_{\mathcal A}(\Phi_p\mathsf L_pG)\in\mathcal T.
\end{equation}}

Since $\cH^{-1}(G)\hookrightarrow\cH^{-1}(\Phi_pE)$ and $\cH^i(\Phi_p E)=F^*\cH^i(E)\otimes M_p$, using \eqref{eq:Mgaps} and applying Lemma~\ref{lem:2.5gap} to $G$ gives
\[\begin{aligned}
\mu_H^+(\cH^{-1}(G))
&\le \mu_H^+(\cH^{-1}(\Phi_pE)) \le\frac{\widehat B_\beta. H^{n-1}}{H^n}
-pa+e_p,\\
\mu_H^-(\cH^0(G))
&\ge\frac{\widehat B_\beta. H^{n-1}}{H^n}
+\mu+\sqrt{\alpha_0^2+\mu^2}
>\frac{\widehat B_\beta. H^{n-1}}{H^n}+\alpha_0.
\end{aligned}\]
Combining with Lemma~\ref{lem:truncatedkernel}, we have
\[
\begin{aligned}
\mu_H^-\bigl(\cH^0(G\otimes R_\ell)\bigr)
&\ge \mu_H^-(\cH^0(G))
     -\ell u-(p-1)v-k_*\\
&>\frac{\widehat B_\beta. H^{n-1}}{H^n}
   +\alpha_0-\ell u-(p-1)v-k_*\\
&>\frac{\widehat B_\beta. H^{n-1}}{H^n},
\end{aligned}
\]
where the last inequality follows from
$\ell u+(p-1)v\le c_p$ and \eqref{eq:prime}.
Since $R_\ell$ is locally free, we have $\cH^j(G\otimes R_\ell)=0$ for $j\notin\{-1,0\}$. Consequently,
\[G\otimes R_\ell\in \D_{\cA}^{\le0}.\]

By Lemma~\ref{lem:tensor} and \eqref{eq:bad}, we have $G(-(\ell u+(p-1)v) H)\in\mathcal T$. Tensoring the defining sequence of $R_\ell$ with $G$, we get an exact triangle
\[G\otimes R_\ell\longrightarrow G(-(\ell u+(p-1)v) H)^{\oplus N_\ell}\longrightarrow G\otimes \T^\ell(\Omega_X)\otimes\omega_X^{1-p}\]
for each $0\le\ell\le n(p-1)$. As $G\otimes R_\ell\in \D^{\leq 0}_{\cA}$, we get $$G\otimes \T^\ell(\Omega_X)\otimes\omega_X^{1-p}\in \D^{\leq 0}_{\cA}.$$ 
Hence, $\V_i(G)\in \D^{\leq 0}_{\cA}$. Moreover, the cohomology long exact sequence with respect to $\cA$ gives a surjection
\[G(-(\ell u+(p-1)v) H)^{\oplus N_\ell} \twoheadrightarrow\cH^0_{\mathcal A}(G\otimes \T^\ell(\Omega_X)\otimes\omega_X^{1-p})\]
in $\Coh^{H,\widehat B_\beta}(X)$, which gives $\cH^0_{\mathcal A}(G\otimes \T^\ell(\Omega_X)\otimes\omega_X^{1-p})\in \cT$. Thus, each triangle
\[\V_{i+1}(G)\to \V_i(G)\to G\otimes \T^i(\Omega_X)\otimes\omega_X^{1-p}\]
associated with the filtration in Lemma~\ref{lem:graded} gives an exact sequence
\[\cH^0_{\cA}(\V_{i+1}(G))\to \cH^0_{\cA}(\V_{i}(G))\to \cH^0_{\mathcal A}(G\otimes \T^i(\Omega_X)\otimes\omega_X^{1-p})\to 0\]
in $\cA$. By induction on $i$, we see that $\cH^0_{\cA}(\V_{i}(G))\in \cT$ for each $i$, and \eqref{eq:positiveP} follows.

It remains to prove $\mathsf L_pG\in \Coh^{H, B_0+\beta H}(X)$. Adjunction gives a nonzero map
$\mathsf L_pG\to E$. Since $c_1(N_p)=(1-p)(K_X-B_0)$, the bounds
\eqref{eq:push} and \eqref{eq:prime} imply
\begin{align}
\mu_H^+(\cH^{-1}(\mathsf L_pG))
-\frac{(B_0+\beta H).H^{n-1}}{H^n}
&\le-a
+\frac{2e_p-(p-1)\frac{K_X.H^{n-1}}{H^n}}p<0,
\notag\\
\mu_H^-(\cH^0(\mathsf L_pG))
-\frac{(B_0+\beta H).H^{n-1}}{H^n}
&>\frac{\alpha_0-e_p}p>0.
\label{eq:Ppositive}
\end{align}
Thus $\mathsf L_pG\in\Coh^{H,B_0+\beta H}(X)$.

\medskip

\textbf{Step 4}.
Let $E'\ne0$ be the image of $\mathsf L_pG\to E$ in $\Coh^{H,B_0+\beta H}(X)$. Therefore, we have morphisms between cohomology objects $\cH^{-1}(E')\hookrightarrow\cH^{-1}(E)$ and $\cH^0(\mathsf L_pG)\twoheadrightarrow\cH^0(E')$.
Thus, combining with \eqref{eq:pull}, \eqref{eq:sourcegaps}, and
\eqref{eq:Ppositive}, we obtain
\[\begin{aligned}
\mu_H^+(\cH^{-1}(\Phi_pE'))
-\frac{\widehat B_\beta. H^{n-1}}{H^n}
&\le-pa+e_p<0,\\
\mu_H^-(\cH^0(\Phi_pE'))
-\frac{\widehat B_\beta. H^{n-1}}{H^n}
&>\alpha_0-2e_p>0.
\end{aligned}\]
Hence $\Phi_pE'\in\mathcal A$.

Let $K$ be the kernel of $\mathsf L_pG\twoheadrightarrow E'$ in
$\Coh^{H,B_0+\beta H}(X)$. Applying $\Phi_p$ and taking cohomology objects with respect to $\mathcal A$, by \eqref{eq:positiveP}, we have an exact sequence
\[\cH^0_{\mathcal A}(\Phi_p\mathsf L_pG) \longrightarrow\Phi_pE' \longrightarrow\cH^1_{\mathcal A}(\Phi_pK) \longrightarrow0\]
in $\cA$. By Lemma~\ref{lem:defect}, the last term is $S[1]$, where $S=0$ or $S$ is a nonzero torsion-free sheaf with 
$$\mu^{\pm}_H(S)\in \left[\frac{\widehat B_\beta. H^{n-1}}{H^n}-e_p,\frac{\widehat B_\beta. H^{n-1}}{H^n}\right].$$
If $S\ne0$, the inequality $p^2A^2>e_p^2+\epsilon$ and Lemma~\ref{lem:strip} applied at $(pA,\widehat B_\beta)$ give $S[1]\in\mathcal T$. This also holds when $S=0$.
Together with \eqref{eq:positiveP}, this yields
$\Phi_pE'\in\mathcal T$. On the other hand, the $\nu^H_{A,B_0+\beta H}$-semistability of $E$ gives $\nu^H_{A,B_0+\beta H}(E')\le0$. Thus, we have $\nu^H_{pA,\widehat B_\beta}(\Phi_pE')
=p\nu^H_{A,B_0+\beta H}(E')\le 0$ by \eqref{eq:scaling} and contradicts $\Phi_pE'\in\mathcal T$.
\end{proof}

This completes the proof for the upper bound of Theorem~\ref{thm:upper}.

\subsection{The lower bound}\label{sec:duality}

We use the following form of tilt duality, extending
\cite[Proposition 5.1.3(b)]{BMT} to the present setting.

\begin{lemma}\label{lem:coherent}
Let $(\alpha,\beta)\in \RR_{>\sqrt\epsilon}\times \RR$ and $E\in\Coh^{H,B_0+\beta H}(X)$ be an object with $\nu^{H,+}_{\alpha,B_0+\beta H}(E)<+\infty$. Set $E^\sharp\coloneqq\tau^{\le0}\DD (E)$ and $T_E\coloneqq\tau^{\ge1}\DD (E)$, where $\tau$ is the truncated functor with respect to the standard heart $\Coh(X)$. Then there is an exact triangle
\begin{equation}\label{eq:rectified}
 E^\sharp\longrightarrow\DD (E)\longrightarrow T_E
\end{equation}
such that

\begin{itemize}
    \item $E^\sharp\in\Coh^{H,B_0^\sharp-\beta H}(X)$, $\nu^{H,+}_{\alpha,B_0^\sharp-\beta H}(E^\sharp)<+\infty$, and $\dim\cH^j(T_E)\le n-j-2$ for $j\ge1$;

    \item $E^\sharp$ is $\nu^H_{\alpha,B_0^\sharp-\beta H}$-semistable if and
only if $E$ is $\nu^H_{\alpha,B_0+\beta H}$-semistable, and
\[\nu^H_{\alpha,B_0^\sharp-\beta H}(E^\sharp) =-\nu^H_{\alpha,B_0+\beta H}(E);\]

\item if $0\to E'\to E\to E''\to0$ is an exact sequence in $\Coh^{H,B_0+\beta H}(X)$ and $\nu^{H,+}_{\alpha,B_0+\beta H}(E')<+\infty$, $\nu^{H,+}_{\alpha,B_0+\beta H}(E'')<+\infty$, then the induced morphism $(E'')^\sharp\to E^\sharp$ is injective in $\Coh^{H,B_0^\sharp-\beta H}(X)$.
\end{itemize}
\end{lemma}

\begin{proof}
The first two properties follows from \cite[Proposition 5.1.3(b)]{BMT}. For the last statement, use \cite[Lemma 9.5(b)]{LM}.
\end{proof}

\begin{proof}[Proof of the lower bound in
Theorem~\ref{thm:twoendpoints}]
By Lemma~\ref{lem:coherent}, the object
\[E^\sharp\in\Coh^{H,B_0^\sharp-\beta H}(X)\]
is $\nu^H_{A,B_0^\sharp-\beta H}$-semistable with $\nu^H_{A,B_0^\sharp-\beta H}(E^\sharp)=0$.
Dualizing $\cH^{-1}(E)[1]\to E\to\cH^0(E)$ and using
\eqref{eq:rectified}, we obtain
\[\cH^{-1}(E^\sharp)\simeq\mathcal Hom(\cH^0(E),\omega_X)\]
and an exact sequence
\[
 \begin{aligned}
 0\longrightarrow\mathcal Ext^1(\cH^0(E),\omega_X)
 &\longrightarrow\cH^0(E^\sharp)\\
 &\longrightarrow\mathcal Hom(\cH^{-1}(E),\omega_X)
 \longrightarrow\mathcal Ext^2(\cH^0(E),\omega_X).
 \end{aligned}
\]
The first Ext sheaf is torsion, and the last is supported in codimension at least two by \cite[Lemma 2.16]{LM}. Thus $\cH^0(E^\sharp)_{\tf}$ and
$\mathcal Hom(\cH^{-1}(E),\omega_X)$ agree in codimension one, so the slopes of their HN factors with respect to $\mu_H$ coincide. Duality for torsion-free sheaves and
\eqref{eq:sourcegaps} therefore give
\[\begin{aligned}
 \mu_H^+(\cH^{-1}(E^\sharp))
 &=\frac{K_X.H^{n-1}}{H^n}-\mu_H^-(\cH^0(E))\\
 &\le\frac{(B_0^\sharp-\beta H).H^{n-1}}{H^n}-a,\\
 \mu_H^-(\cH^0(E^\sharp))
 &=\frac{K_X.H^{n-1}}{H^n}-\mu_H^+(\cH^{-1}(E))\\
 &\ge\frac{(B_0^\sharp-\beta H).H^{n-1}}{H^n}+a.
 \end{aligned}\]

Set
\[L^{\dagger}\coloneqq\omega_X^q\otimes L^{-1}.\]
Then $c_1(L^\dagger)=qB_0^\sharp$ and $(L^\dagger)^{-(p-1)/q}
=\omega_X^{1-p}\otimes M_p^{-1}=N_p$. Thus, the normalized Frobenius pullback associated with $B_0^\sharp$ is $\Psi_p$. Applying the upper bound in Theorem~\ref{thm:twoendpoints} to $E^\sharp$ gives
\[\Psi_p(E^\sharp)\in\Coh^{H,B_0^\sharp-p\beta H}(X), \qquad \nu^{H,+}_{pA,B_0^\sharp-p\beta H} \bigl(\Psi_p(E^\sharp)\bigr)\le pB_p(A).\]

Applying $\Psi_p$ to the triangle \eqref{eq:rectified}
and using \eqref{eq:dual}, we obtain
\[\Psi_p(E^\sharp)\longrightarrow\DD(\Phi_pE) \longrightarrow\Psi_p(T_E).\]
Since $F$ is flat, $\Psi_p$ commutes with the standard truncation functors. In particular, we get
\begin{equation}\label{eq:truncation}
(\Phi_pE)^\sharp\simeq\Psi_p(E^\sharp).
\end{equation}

Let $Q$ be the last HN factor of $\Phi_pE$ with respect to $\nu^H_{pA,\widehat B_\beta}$. In the exact sequence
\[0\longrightarrow K\longrightarrow\Phi_pE \longrightarrow Q\longrightarrow0\]
in $\Coh^{H,\widehat B_\beta}(X)$, all HN factors of three objects have
finite $\nu^H_{pA,\widehat B_\beta}$-slopes. Therefore, the monomorphism assertion of
Lemma~\ref{lem:coherent} and \eqref{eq:truncation} give
\[Q^\sharp\hookrightarrow(\Phi_pE)^\sharp \simeq\Psi_p(E^\sharp).\]
We then conclude that
\[\begin{aligned}
-\nu^{H,-}_{pA,\widehat B_\beta}(\Phi_pE)
&=\nu^H_{pA,B_0^\sharp-p\beta H}(Q^\sharp)\\
&\le\nu^{H,+}_{pA,B_0^\sharp-p\beta H}
       \bigl(\Psi_p(E^\sharp)\bigr)\\
&\le pB_p(A).
\end{aligned}\]
This proves the lower bound and, together with the upper bound in Section~\ref{sec:upper-proof},
establishes Theorem~\ref{thm:twoendpoints}.
\end{proof}

We end this section by recording the numerical consequences of Theorem~\ref{thm:twoendpoints} that will be used in the proof of the main theorem.

We keep the notations and assumptions in Theorem~\ref{thm:twoendpoints}. Define $\alpha_p\coloneqq \sqrt{A^2-\epsilon/p^2}$. Let $G_i$ be the HN factors of $\Phi_pE$ with respect to $\nu^H_{pA,\widehat B_\beta}$, ordered by decreasing slope. We define
\begin{equation}\label{eq:weights}
 \begin{gathered}
 w_i(E)\coloneqq \frac{v_1^{H,\widehat B_\beta}(G_i)}{pv_1^{H,B_0+\beta H}(E)},\qquad
 \xi_i(E)\coloneqq \frac{\nu^H_{pA,\widehat B_\beta}(G_i)}p,\qquad
 x_i(E)\coloneqq \frac{pv_0^H(G_i)}{v_1^{H,\widehat B_\beta}(G_i)},\\
 k(E)\coloneqq \frac{v_0^H(E)}{v_1^{H,B_0+\beta H}(E)},\qquad V_p(E)\coloneqq \sum_iw_i(E)(\xi_i(E))^2.
 \end{gathered}
\end{equation}
By \eqref{eq:scaling}, Theorem~\ref{thm:twoendpoints}, and
Proposition~\ref{prop:tiltBG}(a) applied to each $G_i$, we have
\begin{equation}\label{eq:constraints}
 \begin{gathered}
 w_i(E)>0,\quad\sum_iw_i(E)=1,\quad\sum_iw_i(E)\xi_i(E)=0,\\
 \quad\sum_iw_i(E)x_i(E)=k(E),\quad |\xi_i(E)|\le B_p(A),\\
 \alpha_p^2(x_i(E))^2+2\xi_i(E)x_i(E)\le1.
 \end{gathered}
\end{equation}
In particular, we have $0\le V_p(E)\le B_p(A)^2$, and \eqref{eq:upper} gives
\[\lim_{p\to \infty}B_p(A)= B(A)\coloneqq D+\frac{D^2}{2(A-D)}.\]

\section{Proof of the main theorem}\label{sec:main-proof}

Let $(X,H)$ be a polarized smooth projective variety of dimension $n\geq3$ over a field $\kk$ of characteristic $0$. In this section, we prove Theorem~\ref{thm:main1.1}.

\subsection{Setup}

By \cite[Lemma 9.21]{LM}, after base change to $\overline{\kk}$, we may assume that $\kk$ is algebraically closed. As in Section~\ref{sec:geometry}, we choose line
bundles $D_1,\ldots,D_\rho$ that generate
$\NS(X)_{\mathrm{tf}}$ and $D_1=K_X$, and choose an integer $N>0$ such that $NH\pm D_i$ are all very ample. The corresponding complete linear systems give an embedding $i\colon X\hookrightarrow \PP$ into a product $\PP$ of $2\rho$ projective spaces. Set $H_j\coloneqq i^*h_j$, where $h_1,\dots,h_{2\rho}$ are the hyperplane classes on factors of $\PP$. Then
\[\operatorname{span}_{\RR}\{H_j\}=\NS(X)_\RR.\]
As in Section~\ref{sec:geometry}, we set $J\coloneqq \sum_jH_j=2\rho NH=:MH$.
We fix the choices in \eqref{eq:gg} and their evaluation kernels
with respect to $H$, and keep $D=nu+v$ throughout.

Fix a $\nu^H_{A,B_0+\beta H}$-semistable object
$E\in\Coh^{H,B_0+\beta H}(X)$ with $\nu^H_{A,B_0+\beta H}(E)=0$, where
$A^2\in\QQ$, $B_0\in\NS(X)_\QQ$, $\beta\in\QQ$, and
$A\ge2D$. Choose $q\in\ZZ_{>0}$ and a line bundle $L$ with $c_1(L)=qB_0$. By the standard spread out argument, we can find a normal finitely generated $\mathbb{Z}$-domain $R\subset \kk$ with a projective flat morphism $\cX\to T=\mathrm{Spec}(R)$ such that the base change $\cX_{\kk}$ of the generic fiber of $\cX\to T$ to $\kk$ is isomorphic to $X$. Moreover, up to enlarge $R$ and shrink $T$, we can assume that

\begin{itemize}
    \item every fiber of $\cX\to T$ is smooth,

    \item there is a $T$-ample line bundle $\cH$ on $\cX$ such that $\cH_{\kk}=H$,

    \item there are line bundles $\cD_i$ on $\cX$ such that $(\cD_i)_{\kk}=D_i$,

    \item there is a line bundle $\cL$ on $\cX$ such that $c_1(\cL_{\kk})=qB_0\in \NS(X)$,

    \item there is a $T$-perfect object $\cE\in \Db(\cX)$ such that $\cE_{\kk}=E$, and

    \item there exists $\epsilon\in \QQ_{\geq 0}$ such that for any point $s\to T$, we have
    \[\Del^{\cH_s}_{B_s,\epsilon}(W)\geq 0\]
for any $B_s\in \NS(X_s)_{\RR}$ and any $\mu_{\cH_s}$-semistable sheaf $W$ on $X_s$.
\end{itemize}

Note that the last property can be achieved by Theorem~\ref{thm:uniformBG}. If $s\to T$ is a geometric point, then we set
$X_s\coloneqq\cX_s$, $H_s\coloneqq\cH_s$, $J_s\coloneqq MH_s$,
$L_s\coloneqq\cL_s$, $B_{0,s}\coloneqq \frac{1}{q}c_1(L_s)$,
$D_{i,s}\coloneqq(\cD_i)_s$.

After shrinking the base, the fixed choices $u,r_0,v$ in
\eqref{eq:gg}, their evaluation maps, and the globally generated twists
$K_{\Omega,s}(k_\Omega H_s)$ and
$K_{\omega^{-1},s}(k_{\omega^{-1}}H_s)$ of their kernels extend
uniformly to the fibers. We retain $D=nu+v$ and
$k_*=\max\{u,v,k_\Omega,k_{\omega^{-1}}\}$ from
Section~\ref{sec:frobenius}. By Lemma~\ref{lem:geometricBG},
we may choose $\epsilon\le D^2$.

\begin{lemma}\label{lem:ss-model}
For all but finitely many primes $p$, there are geometric points $s\to T$ of characteristic $p$ and objects $E_s$ such that

\begin{itemize}
    \item $E_s\in \Coh^{H_s,B_{0,s}+\beta H_s}(X_s)$ is $\nu^{H_s}_{A,B_{0,s}+\beta H_s}$-semistable with $\nu^{H_s}_{A,B_{0,s}+\beta H_s}(E_s)=0$,

    \item $v_i^{H_s,B_{0,s}+\beta H_s}(E_s)=v_i^{H,B_0+\beta H}(E)$ for $0\leq i\leq2$, and
    \[v_3^{H_s,B_{0,s}+\beta H_s}(E_s)\geq v_3^{H,B_0+\beta H}(E),\]

    \item $\ch_1(E_s)\in \CH^1(X_s)$ is the specializations of $\ch_0(E)$ and $\ch_1(E)\in \CH^1(X)$.
\end{itemize}
\end{lemma}

\begin{proof}
For all but finitely many primes $p$, we know that $(p)\subset R$ is a nontrivial ideal. Thus, we can choose a prime $\mathfrak p\subset R$ of height one which is minimal over $(p)\subset R$. Then $R_{\mathfrak p}$ is a DVR. Now, the result follows from \cite[Theorem 10.10]{LM}, whose argument works for any $n\geq 3$.
\end{proof}

In the following, we always fix the above data that satisfies the properties listed above. 

\subsection{A uniform inequality}

We first apply Proposition~\ref{prop:weak} to each special fiber of $\cX\to T$.

After shrinking $T$ if necessary, we may assume that all $N\cH\pm \cD_i$ are very ample over $T$. Therefore, they induce an embedding of $\cX$ into a product of $2\rho$ projective spaces over $T$, extending $i\colon X\hookrightarrow \PP$. Therefore, the class $\Theta$ in \eqref{eq:theta} extends via the relative normal bundle and we denote its restriction to $X_s$ by $\Theta_s$.

\begin{lemma}\label{lem:seed}
There exists $m\gg 1$ such that for every geometric point $s\to T$ outside a finite set of characteristics, every rational linear combination $B_s$ of $D_{i,s}$ and $H_s$, every $\alpha\geq MA_0$ and every $\nu^{H_s}_{\alpha, B_s}$-semistable object $F\in \Coh^{H_s, B_s}(X_s)$ with $\nu^{H_s}_{\alpha, B_s}(F)=0$, we have
\[H_s^{n-3}.\ch_3^{B_{s}}(F) \le H_s^{n-3}.\left(\frac{\alpha^2}{6}H_s^2+\Gamma_{0,s}\right). \ch_1^{B_{s}}(F).\]
Here
\[A_0\coloneqq \sqrt{\frac{m^4}{12}-\frac{2J^{n-2}.\Theta}{J^n}}>\sqrt{\epsilon}\]
and
\[\Gamma_{0,s}\coloneqq \left(\frac{m^4}{8}-\frac{A_0^2}{6}\right)J_s^2-\Theta_s.\]
\end{lemma}

\begin{proof}
Note that we can choose $m$ in Lemma~\ref{lem:embedding} uniformly on each fiber (see also~\cite{LLSZ}), as it only depends on the Hilbert polynomial. Thus, the result follows from Proposition~\ref{prop:weak} and Lemma~\ref{lem:smallAimplieslargeA}.
\end{proof}

\subsection{Large \texorpdfstring{$p$}{p} limits}

By Lemma~\ref{lem:2.5gap}, we have
\[\begin{aligned}
 \mu_H^+(\cH^{-1}(E))&\le\frac{(B_0+\beta H).H^{n-1}}{H^n}-A,\\
 \mu_H^-(\cH^0(E))&\ge\frac{(B_0+\beta H).H^{n-1}}{H^n}+A.
 \end{aligned}\]
Set $a=\sqrt{A^2-\epsilon}>0$. For the objects $E_s$ in Lemma~\ref{lem:ss-model}, the proof of Lemma~\ref{lem:2.5gap} gives
\[aH^n\bigl(\rk\cH^{-1}(E_s)+\rk\cH^0(E_s)\bigr) \leq v_1^{H,B_0+\beta H}(E).\]
Therefore, we can fix an integer $R_E\geq1$ such that $$\rk\cH^{-1}(E_s)+\rk\cH^0(E_s)\leq R_E$$ for each $E_s$ in Lemma~\ref{lem:ss-model}. Since $T_{X_s}(r_0H_s)$ is globally
generated, by \cite[Corollary 2.5 and Theorem 5.1, equation (5.1.1)]{Langer}, we have
\[\Del^{H_s}_{B_{0,s}+\beta H_s,\delta_p}(W)\geq 0\]
for any HN factor $W$ of $\cH^{-1}(E_s)$ or $\cH^0(E_s)_{\tf}$ with respect to $\mu_{H_s}$, where
\[\delta_p\coloneqq \frac{(R_E-1)^2r_0^2}{(p-1)^2}.\]
Note that for $p\gg1$, we have $\delta_p<A^2$. Since $A$, $R_E$, and
$r_0$ are all independent of $p$, we have $a_p\coloneqq \sqrt{A^2-\delta_p}=A-O(p^{-2})$ and
\[A-a_p=\frac{\delta_p}{A+a_p}=O(p^{-2}).\]
Now, applying the same calculation in Lemma~\ref{lem:2.5gap} gives
\begin{equation}\label{eq:complexgaps}
 \begin{aligned}
 \mu_{H_s}^+(\cH^{-1}(E_s))&\le\frac{(B_{0,s}+\beta H_s).H_s^{n-1}}{H_s^n}-a_p,\\
 \mu_{H_s}^-(\cH^0(E_s)_{\tf})&\ge\frac{(B_{0,s}+\beta H_s).H_s^{n-1}}{H_s^n}+a_p
 \end{aligned}
\end{equation}
for $p\gg1$.

For the geometric points $s\to T$ in Lemma~\ref{lem:ss-model}, we denote by $F_s\colon X_s\to X_s$ the absolute Frobenius and use the
normalization of Section~\ref{sec:normalization}. For instance, we set $M_{p,s}\coloneqq L_s^{-(p-1)/q}$, $\widehat B_{\beta,s}\coloneqq B_{0,s}+p\beta H_s$, and 
\[\Phi_{p,s}\coloneqq (-\otimes M_{p,s})\circ F_s^*.\]

In this subsection, we write
\[v_i=v_i^{H,B_0+\beta H}(E)\]
for simplicity. We use
$w_i(E_s)$, $\xi_i(E_s)$, $x_i(E_s)$, and $V_p(E_s)$ from
\eqref{eq:weights}, with $(X,H,B_0,E)$ replaced by
$(X_s,H_s,B_{0,s},E_s)$. In particular, we have $k(E_s)=k(E)=\frac{v_0}{v_1}$ and $\alpha_p=\sqrt{A^2-\epsilon/p^2}$.
Set
\[\begin{gathered}
 \eta\coloneqq \frac{\Del^H_{B_0+\beta H}(E)}{v_1^2}
      =1-A^2k(E)^2\in[0,1],\\
 \eta_p\coloneqq1-\alpha_p^2k(E)^2
       =\eta+\frac{\epsilon k(E)^2}{p^2}.
 \end{gathered}\]

\begin{lemma}\label{lem:jointmoments}
We have
\begin{equation}\label{eq:joint-bound}
 \frac{v_3^{H,B_0+\beta H}(E)}{v_1^{H,B_0+\beta H}(E)}-\frac{A^2}{6}
 \le\frac{B(A)^2+B(A)\sqrt{A^2+B(A)^2}}3\sqrt\eta.
\end{equation}
\end{lemma}

\begin{proof}
We take a geometric point $s\to T$ as in Lemma~\ref{lem:ss-model}, with $p$ sufficiently large that $pa_p>2D(p-1)$, $p\equiv1\pmod{q}$, and
\begin{equation}\label{eq:main-explicitprime}
p>\max\Bigl\{4,\frac{k_*+\sqrt\epsilon}{A-D},
 \frac12+\frac{\epsilon}{8D^2},\frac{MA_0}{A},\frac{(R_E-1)r_0}{a_p}\Bigr\}.
\end{equation}

Let $G_i$ be the
HN factors of $\Phi_{p,s}E_s$ with respect to
$\nu^{H_s}_{pA,\widehat B_{\beta,s}}$ of slopes $p\xi_i(E_s)$.
By Lemma~\ref{lem:nestedwall}, $G_i$ is tilt-semistable of slope zero at
\[(p\sqrt{A^2+\xi_i(E_s)^2},\widehat B_{\beta,s}+p\xi_i(E_s)H_s).\]
By \eqref{eq:main-explicitprime}, we have $pA>MA_0$, so
Lemma~\ref{lem:seed} applies and gives
\[\begin{split}
 v_3^{H_s,\widehat B_{\beta,s}}(G_i)
 \le{}&H_s^{n-3}.\Gamma_{0,s}.\ch_1^{\widehat B_{\beta,s}}(G_i)\\
 &+\frac{p^2(A^2+4\xi_i(E_s)^2)}6
          v_1^{H_s,\widehat B_{\beta,s}}(G_i)\\
 &+\left(\frac{p^2A^2}{3}-\frac{H_s^{n-2}.\Gamma_{0,s}}{H_s^n}\right)
             p\xi_i(E_s)v_0^{H_s}(G_i).
 \end{split}\]
Set $y_p=\sum_iw_i(E_s)\xi_i(E_s)x_i(E_s)$.
Since $v_3^{H_s,B_{0,s}+\beta H_s}(E_s)\geq v_3$, summing over all $G_i$, applying \eqref{eq:scaling}, and dividing by $p^3v_1$ gives
\[
 \begin{aligned}
 \frac{v_3}{v_1}-\frac{A^2}{6}
 &\le\frac{H_s^{n-3}.\Gamma_{0,s}.\ch_1^{B_{0,s}+\beta H_s}(E_s)}{p^2v_1}
       +\frac23V_p(E_s)\\
 &\quad+\left(\frac{A^2}{3}-\frac{H_s^{n-2}.\Gamma_{0,s}}{p^2H_s^n}\right)y_p.
 \end{aligned}
\]
Note that by \eqref{eq:complexgaps} and Theorem~\ref{thm:twoendpoints}, we have $|\xi_i(E_s)|\le B_p(A)$. Since $\lim_{p\to \infty}\alpha_p=A$ and $\lim_{p\to \infty}B_p(A)=B(A)$, from the quadratic bound of $x_i(E_s)$ in \eqref{eq:constraints}, we see that $x_i(E_s)$ is also bounded when $p\to \infty$. In particular, $y_p$ is bounded. By Lemma~\ref{lem:ss-model}, the full first Chern class specializes, so the $\Gamma_{0,s}$-term is uniformly $O(p^{-2})$. Since $\alpha_p^2=A^2-\epsilon/p^2$, it follows that
\begin{equation}\label{eq:jointmoments}
 \frac{v_3}{v_1}-\frac{A^2}{6}
 \le\frac23V_p(E_s)+\frac{\alpha_p^2}{3}y_p+o(1)
\end{equation}
when $p\to \infty$. In the following, we estimate $V_p(E_s)$ and $y_p$.

By the global generation of $T_{X_s}(r_0H_s)$, for any nonzero torsion-free sheaf $W$ on $X_s$,  \cite[proof of Corollary 2.5, p.~258]{Langer}, applied to the HN factors of $W$ with respect to $\mu_{H_s}$, gives
\[
 \begin{aligned}
 \mu_{H_s}^+(F_s^*W)&\le p\mu_{H_s}^+(W)+(\rk W-1)r_0,\\
 \mu_{H_s}^-(F_s^*W)&\ge p\mu_{H_s}^-(W)-(\rk W-1)r_0.
 \end{aligned}
\]
Apply these bounds to
$\cH^{-1}(E_s)$ and $\cH^0(E_s)_{\tf}$ when nonzero, and tensor
by $M_{p,s}$, we obtain
\begin{equation}\label{eq:margin}
 \begin{aligned}
 \mu_{H_s}^+(\cH^{-1}(\Phi_{p,s}E_s))
 &\le\frac{\widehat B_{\beta,s}.H_s^{n-1}}{H_s^n}-pg_p,\\
 \mu_{H_s}^-(\cH^0(\Phi_{p,s}E_s)_{\tf})
 &\ge\frac{\widehat B_{\beta,s}.H_s^{n-1}}{H_s^n}+pg_p,
 \end{aligned}
\end{equation}
where
\[g_p=\min\left\{\alpha_p, a_p-\frac{(R_E-1)r_0}{p}\right\}.\]
Note that under \eqref{eq:main-explicitprime}, we have $g_p>0$ and $\lim_{p\to \infty} g_p=A$.

For any subobject $P$ and quotient $Q$ of
$\Phi_{p,s}E_s$ in $\Coh^{H_s,\widehat B_{\beta,s}}(X_s)$,
there are an inclusion $\cH^{-1}(P)\hookrightarrow\cH^{-1}(\Phi_{p,s}E_s)$ and a surjection $\cH^0(\Phi_{p,s}E_s)\twoheadrightarrow\cH^0(Q)$.
Together with \eqref{eq:margin} and the definition of $\Coh^{H_s,\widehat B_{\beta,s}}(X_s)$, these give
\begin{equation}\label{eq:p1-ineq}
 \begin{aligned}
 v_1^{H_s,\widehat B_{\beta,s}}(Q)-pg_pv_0^{H_s}(Q)&\ge0,\\
 v_1^{H_s,\widehat B_{\beta,s}}(P)+pg_pv_0^{H_s}(P)&\ge0.
 \end{aligned}
\end{equation}

For simplicity, we write $w_i,\xi_i,x_i$ for
$w_i(E_s),\xi_i(E_s),x_i(E_s)$.
When $k(E)\ge0$, take $Q$ to be the part of the HN filtration of $\Phi_{p,s}E_s$ whose factors have
nonpositive slopes. When $k(E)<0$, take $P$ to be the part of the HN filtration of $\Phi_{p,s}E_s$ whose factors have nonnegative slopes, and then replace
$(\xi_i,x_i)$ by $(-\xi_i,-x_i)$ and reverse the order of the factors.
This leaves $V_p(E_s)$ and $y_p$ unchanged. By \eqref{eq:p1-ineq} and \eqref{eq:constraints},
in both cases we have
\[\sum_iw_i\xi_i=0,\qquad \sum_iw_ix_i=|k(E)|,\qquad \sum_{\xi_i\le0}w_i(1-g_px_i)\ge0.\]
From $\lim_{p\to \infty}\alpha_p=\lim_{p\to \infty}g_p=A$, the boundedness of $x_i$, and
$\sum_iw_i=1$, we obtain
\begin{equation}\label{eq:suffix-limit}
 \sum_{\xi_i\le0}w_i(1-\alpha_px_i)\ge-o(1)
\end{equation}
when $p\to \infty$.

Set
\[r\coloneqq \frac{\sqrt{\alpha_p^2+B_p(A)^2}-B_p(A)}{\alpha_p}.\]
By \eqref{eq:prime} and \eqref{eq:upper}, we have
$B_p(A)/\alpha_p<3/4$, and hence $1/2<r<1$.
For $|\xi|\le B_p(A)$ and $\alpha_p^2x^2+2\xi x\le1$, we have
\begin{equation}\label{eq:511}
2\xi^2+\alpha_p^2\xi x
 \le\frac{\alpha_pB_p(A)}r
      -\alpha_p^2B_p(A)x+\alpha_pr\xi,
\end{equation}
and
\begin{equation}\label{eq:512}
 2\xi^2+\alpha_p^2\xi x\le
 \begin{cases}
 \alpha_pB_p(A)(1-\alpha_px)
       +(\alpha_p-2B_p(A))\xi,&\xi\le0,\\[2pt]
 \dfrac{\alpha_pB_p(A)(r+2)}r(1-\alpha_px)
       +(\alpha_p-2B_p(A))\xi,&\xi>0.
 \end{cases}
\end{equation}
Indeed, $\alpha_p^2x^2+2\xi x\le1$ gives
\[x\le\frac{\sqrt{\alpha_p^2+\xi^2}-\xi}{\alpha_p^2},\]
and the function
\[2\xi^2+(\xi+K)(\sqrt{\alpha_p^2+\xi^2}-\xi)\]
in $\xi$ is convex when $\xi\ge -K$.
Taking its chords on $[-B_p(A),B_p(A)]$ and $[-B_p(A),0]$
with $K=B_p(A)$, and on $[0,B_p(A)]$ with
$K=B_p(A)(r+2)/r$, gives \eqref{eq:511} and \eqref{eq:512}.

Now, we apply \eqref{eq:511} and \eqref{eq:512} to $(\xi_i,x_i)$ and sum with weights
$w_i$. Since $\sum_iw_i\xi_i=0$ and
\[2V_p(E_s)+\alpha_p^2y_p=\sum_i w_i(2\xi_i^2+\alpha_p^2\xi_ix_i),\]
\eqref{eq:511} gives
\begin{equation}\label{eq:first}
 2V_p(E_s)+\alpha_p^2y_p
 \le\frac{\alpha_pB_p(A)}r
       \bigl(1-r\alpha_p|k(E)|\bigr).
\end{equation}
Similarly, using \eqref{eq:suffix-limit} and \eqref{eq:512}, we obtain
\begin{equation}\label{eq:second}
 \begin{aligned}
 2V_p(E_s)+\alpha_p^2y_p
 &\le\frac{\alpha_pB_p(A)(r+2)}r
          \bigl(1-\alpha_p|k(E)|\bigr)\\
 &\quad-\frac{2\alpha_pB_p(A)}r
          \sum_{\xi_i\le0}w_i(1-\alpha_px_i)\\
 &\le\frac{\alpha_pB_p(A)(r+2)}r
          \bigl(1-\alpha_p|k(E)|\bigr)+o(1)
 \end{aligned}
\end{equation}
when $p\to \infty$.

Note that for $0\le t\le1$ and $1/2\le r<1$, we have
\[\min\{1-rt,(r+2)(1-t)\}\le\sqrt{1-t^2}.\]
Applying this with $t=\alpha_p|k(E)|\in[0,1]$ and \eqref{eq:first} and \eqref{eq:second}, and recalling that
$\eta_p=1-\alpha_p^2k(E)^2$, we obtain
\[
 \begin{aligned}
 \frac23V_p(E_s)+\frac{\alpha_p^2}{3}y_p
 &\le\frac{\alpha_pB_p(A)}{3r}\sqrt{\eta_p}+o(1)\\
 &=\frac{B_p(A)^2+B_p(A)\sqrt{\alpha_p^2+B_p(A)^2}}3
       \sqrt{\eta_p}+o(1)
 \end{aligned}
\]
when $p\to \infty$. Combining this with \eqref{eq:jointmoments} and letting
$p\to\infty$ proves \eqref{eq:joint-bound}.
\end{proof}

\subsection{Proof of the main theorem}\label{sec:proof}

In the following, we prove
\[v_3^{H,B_0+\beta H}(E)-\frac{\alpha^2}{6}v_1^{H,B_0+\beta H}(E) \le\mathsf{C}(X,H)\sqrt{\Del^H_{B_0+\beta H}(E)}\]
for any $B_0\in \NS(X)_{\RR}$, $(\alpha,\beta)\in \RR_{>0}\times \RR$, and any $\nu^{H}_{\alpha, B_0+\beta H}$-semistable object $E\in \Coh^{H,B_0+\beta H}(X)$ with $\nu^{H}_{\alpha, B_0+\beta H}(E)=0$. Then the desired quadratic
inequality in Theorem~\ref{thm:main1.1} follows from Proposition~\ref{prop:ch3boundimplyquadform}.

\begin{proof}[Proof of Theorem~\ref{thm:main1.1}]
First, we consider $\alpha\ge 2D$ with $\alpha^2\in\QQ$, $B_0\in\NS(X)_\QQ$, $\beta\in\QQ$. Let $E\in \Coh^{H, B_0+\beta H}(X)$ be a
$\nu^H_{\alpha,B_0+\beta H}$-semistable object with $\nu^H_{\alpha,B_0+\beta H}(E)=0$. We set $\eta=\Del^H_{B_0+\beta H}(E)/(v_1^{H,B_0+\beta H}(E))^2$ and
\[f_\alpha(\eta)\coloneqq \frac{B(\alpha)^2+B(\alpha)\sqrt{\alpha^2+B(\alpha)^2}}3\sqrt\eta.\]
Therefore, Lemma~\ref{lem:jointmoments} gives
\[v_3^{H,B_0+\beta H}(E)-\frac{\alpha^2}{6}v_1^{H,B_0+\beta H}(E) \le v_1^{H,B_0+\beta H}(E)f_\alpha(\eta).\]
Therefore, the same argument as in Lemma~\ref{lem:rationalparametertoreal} extends the statement to $\beta\in \RR$. If $\alpha=2D$, one has $B(2D)=\frac32D$ and $f_{2D}(\eta)=2D^2\sqrt\eta$. Thus, for any $\beta\in \RR$ and any $\nu^H_{2D,B_0+\beta H}$-semistable object $E$ with $\nu^H_{2D,B_0+\beta H}(E)=0$, we have
\[v_3^{H,B_0+\beta H}(E)-\frac{4D^2}{6}v_1^{H,B_0+\beta H}(E) \le2D^2\sqrt{\Del^H_{B_0+\beta H}(E)}.\]
Combining with Proposition~\ref{prop:propagate} and $\epsilon=0$, we see that 
\[v_3^{H,B_0+\beta H}(E)-\frac{\alpha^2}{6}v_1^{H,B_0+\beta H}(E) \le \mathsf{C}(X,H)\sqrt{\Del^H_{B_0+\beta H}(E)}\]
with 
\[\mathsf{C}(X,H)\coloneqq\frac83D^2\]
for any $(\alpha,\beta)\in \RR_{>0}\times \RR$ and any $\nu^H_{\alpha,B_0+\beta H}$-semistable object $E$ with $\nu^H_{\alpha,B_0+\beta H}(E)=0$.

Finally, Lemma~\ref{lem:rationalparametertoreal} extends the statement to every $B_0\in \NS(X)_{\RR}$. This finishes the proof.
\end{proof}

\begin{remark}\label{rmk:constant}
Here, we recall the definition of the explicit constant $\mathsf{C}(X,H)$. Choose integers $u>0$ and $v\geq 0$ such that $\Omega_{X}(uH)$ and $\omega_{X}^{-1}(vH)$ are globally generated. Then we have
\[\mathsf{C}(X,H)=\frac{8}{3}(nu+v)^2.\]
In particular, if $\omega_X=\cO_X$, we may take $v=0$ and obtain
\[\mathsf{C}(X,H)=\frac{8}{3}n^2u^2.\]
If, moreover, $mH$ is very ample, then $\Omega_X(2mH)$ is globally generated. Thus, in this case we may take $u=2m$ and obtain
\[\mathsf{C}(X,H)=\frac{32}{3}n^2m^2.\]
\end{remark}

\section{Applications}

In this final section, we discuss some applications of Theorem~\ref{thm:main1.1}. We first show that Theorem~\ref{thm:main1.1} makes results in \cite{FT1,FT2} unconditional. In particular, we prove Theorem~\ref{thm:rk-r} and Corollary~\ref{cor:sym-inv}. Then, we discuss the genus bound of curves in threefolds. Following ideas from \cite{liu-ruan,liu}, we prove the Castelnuovo Bound Conjecture up to an explicit constant caused by $\mathsf C(X,H)$ (cf.~Theorem~\ref{thm:cast-bound}).

\subsection{Higher rank DT-invariants}
Let $X$ be a  Calabi--Yau threefold over $\CC$, i.e.~a smooth projective threefold with $\omega_X=\cO_X$ and $\rH^1(X,\cO_X)=0$. We denote by $\Knum(X)$ the numerical K-group of $\Db(X)$. Thus, the homomorphism $\ch \colon \KK(X)\to \rH^*(X,\QQ)$ factors through the surjection $\KK(X)\twoheadrightarrow \Knum(X)$, and we denote the induced homomorphism $\Knum(X)\to \rH^*(X,\QQ)$ by $\ch$ as well.

For any ample divisor $H$ on $X$ and $\bv\in \Knum(X)$, we denote by $\DT_H(\bv)\in \QQ$ the generalized Donaldson--Thomas invariant counting Gieseker-semistable sheaves (with respect to $H$) of class $\bv$ constructed in \cite{JS}.

\begin{theorem}\label{thm:rk-r}
Let $(X,H)$ be a polarized Calabi--Yau threefold. Then for a fixed $\bv\in \Knum(X)$ with $\rk(\bv)\geq 0$, \[\mathsf{DT}_H(\bv)=F_1(\mathsf{DT}_H(\alpha_1),\mathsf{DT}_H(\alpha_2),\dots)\]
is a universal polynomial in invariants $\mathsf{DT}_H(\alpha_i)$ with $\rk(\alpha_i)=1$.

Similarly, for a fixed $\bv\in \Knum(X)$ with $\rk(\bv)> 0$, \[\mathsf{DT}_H(\bv)=F_0(\mathsf{DT}_H(\beta_1),\mathsf{DT}_H(\beta_2),\dots)\]
is a universal polynomial in invariants $\mathsf{DT}_H(\beta_i)$ with $\rk(\beta_i)=0$ and $\ch_1(\beta_i)\neq 0$.
\end{theorem}

\begin{proof}
The result essentially follows from \cite{FT1,FT2}. In the following, we briefly explain how to check inequalities and assumptions in \cite{FT1,FT2} using Theorem~\ref{thm:main1.1}. Choose $m$ such that $mH$ is very ample. By Matsusaka's theorem, we may choose $m$ that it depends only on the numbers $H^3$ and $c_2(X).H$. By Remark~\ref{rmk:constant}, we may take $\mathsf C=96m^2$.

Following \cite[Section 1]{FT1}, we write
\[\ch_H(F)\coloneqq \left(\ch_0(F),\frac{H^2.\ch_1(F)}{H^3}, \frac{H.\ch_2(F)}{H^3},\frac{\ch_3(F)}{H^3}\right) =(r,c,s,d).\]
The parameters in \cite{FT1,FT2} satisfy
\[w\coloneqq \frac{\alpha^2+b^2}{2},\qquad \nu_{b,w}(F)\coloneqq \nu^H_{\alpha,bH}(F)+b=\frac{s-wr}{c-br}.\]
By Theorem~\ref{thm:main1.1}, a $\nu^H_{\alpha, bH}$-semistable object
$F\in\Coh^{H,bH}(X)$ satisfies
\begin{equation}\label{eq:linear}
 \begin{aligned}
 0\le{}&(c^2-2rs)w+(3rd-cs)b+2s^2-3cd\\
       &+3\mathsf C(c-br)\sqrt{c^2-2rs}.
 \end{aligned}
\end{equation}
In applying \cite{FT1,FT2}, we define $\ell_f(F)$ by equality in
\eqref{eq:linear}, for $c^2-2rs>0$. When $r\ne0$, it passes
through $\Pi(F)=(c/r,s/r)$.

We first verify \textup{(wBG)} in \cite[Section 3.1]{FT2}.
Using the notation there, we write $\ch(\bv)=(r,D,-\beta,-m)$ and $\bv_n=\bv-[\cO_X(-nH)]$, where $r>0$, $D\in \rH^2(X,\ZZ)$ and $\beta\in \rH^4(X,\QQ)$. Then
\[\ch(\bv_n)=\left(r-1,D+nH, -\beta-\frac{n^2}{2}H^2,-m+\frac{n^3}{6}H^3\right).\]
The intersections of $\ell_f(\bv_n)$ with $w=b^2/2$ satisfy
\[
 \begin{aligned}
 b_1&=-n+O(n^{-1}),\\
 b_2&=\mu_H(\bv)-\frac1n\left(
 2\mu_H(\bv)^2+\frac{4\beta.H}{rH^3}
                    +\frac{6\mathsf C}{\sqrt r}\right)+O(n^{-2}).
 \end{aligned}
\]
Increasing $\mathsf C$ if necessary, and then taking $n\gg0$, gives
\[b_1<-n+\frac1{4r^2H^3},\qquad \mu_H(\bv)-\frac1{4r^2H^3}<b_2<\mu_H(\bv).\]
Thus, $\ell_f(\bv_n)$ lies strictly above the line $\ell_{\bv,n}$ in
\cite[Section 3.1]{FT2}. On the part of $\ell_{\bv,n}$ with
$w>b^2/2$, the right-hand side of \eqref{eq:linear} is negative,
so there are no semistable objects of class $\bv_n$ along $\ell_{\bv,n}$ by Theorem~\ref{thm:main1.1}. This proves \textup{(wBG)} and the second statement follows.

For the first statement, we check the rank $0$ reduction of \cite{FT1}. Fix a class $\bv$ with
$\ch_H(\bv)=(0,c,s_0,d_0)$, where $c>0$, and set
$\bv_{n_0}=\bv-[\cO_X(-n_0H)]$ for $n_0\gg0$. Thus
\[\ch_H(\bv_{n_0})= \left(-1,n_0+c,s_0-\frac{n_0^2}{2},d_0+\frac{n_0^3}{6}\right).\]
We now denote by $\ell_f$ the equality line of
\eqref{eq:linear} for this class. It passes through
$\Pi(\bv_{n_0})$ and its slope is the slope in
\cite[equation (23)]{FT1} minus
$3\mathsf C/\sqrt{2cn_0+c^2+2s_0}$. The two intersection point of $\ell_f$ with
$w=b^2/2$ have $b$-coordinates
\begin{equation}\label{eq:endpoints}
 b_1^f=-n_0+\frac c3+O(n_0^{-1}),\qquad
 b_2^f=\frac{n_0}{2}-\frac c{12}+\frac{3s_0}{2c}
                 +O(n_0^{-1/2}).
\end{equation}
The slope bound in \cite[equation (24)]{FT1} and the estimates in
\cite[Lemmas 2.1--2.2]{FT1} therefore still hold; note that we use \cite[Definition 2.3]{FT1} with this choice of $\ell_f$.

Let $\bw_n$ be close to $\bv_{n_0}$ and write, as in
\cite[equation (33)]{FT1},
\[
 \begin{aligned}
 \ch_H(\bw_n)&=\left(-1,n+c,-\frac{n^2}{2}+s,\frac{n^3}{6}+d\right),\\
 \ch_H(E_0)&=(0,c-c',s-s',d-d'),\\
 \ch_H(E_1)&=\left(-1,n+c',-\frac{n^2}{2}+s',
                                  \frac{n^3}{6}+d'\right),
 \end{aligned}
\]
where $E_0,E_1$ are the semistable factors of ranks $0,-1$ and
$\delta_n=n_0-n$. By \cite[Section 2.2]{FT1}, the remaining uses
of the BMT inequality are in its equations (38) and (40).
The proof of \cite[Lemma B.3]{FT2}, applied to $E_0$ with
\eqref{eq:linear}, replaces the former by
\begin{equation}\label{eq:rankzero}
 d-d'\le\frac{(s-s')^2}{2(c-c')}
                 +\frac{(c-c')^3}{24}+\mathsf C(c-c').
\end{equation}
For $E_1$, the additional term on the left of
\cite[equation (40)]{FT1} is
\[3\mathsf C(n+c')\sqrt{(c')^2+2nc'+2s'}.\]
In the case $c'\le0$ of \cite[Proposition 2.4]{FT1}, the bounds
$c'=O(1)$ and $s'=O(n_0)$ give corrections of orders $O(n_0)$
and $O(n_0^{3/2})$. Hence, its final estimate in the proof of \cite[Proposition 2.4]{FT1} is still
\[0\le\frac{c'}2+ \frac{\delta_n}{c-c'} \left(\frac{\delta_n}{2}-\frac{2c}{3}\right) +O(n_0^{-1}),\qquad c'\le0,\quad 0\le\delta_n\le\frac c3.\]
Since $c',\delta_n\in(H^3)^{-1}\ZZ$, this gives
$c'=\delta_n=0$ for $n_0\gg0$. Consequently, \cite[Proposition 2.4]{FT1} still holds.

For \cite[Proposition 2.5]{FT1}, we assume that $\bw_n=\bv_{n_0}$ and $c'=0$. The inequalities in our case give
\begin{equation}\label{eq:rankminusone}
 s'\left(n_0+\frac{s'}{2c}+\frac{s_0}{c}\right)
 \le\frac{3s_0^2}{2c}-3d_0+\frac{c^3}{8}
       +3\mathsf Cc+3\mathsf C\sqrt{2s'}.
\end{equation}
Here $s'\ge0$, and $s'>0$ implies $s'\ge1/(2H^3)$.
If $s'>0$, dividing \eqref{eq:rankminusone} by $s'$
gives a contradiction for $n_0\gg0$. Thus $s'=0$. The discriminant of $E_1$ is zero,
and the correction in \eqref{eq:linear} vanishes.
The argument of \cite[Lemma 2.6]{FT2} consequently gives
$E_1\simeq T(-n_0H)[1]$, where $c_1(T)=0$. This means \cite[Proposition 2.5]{FT1} also holds in our case.

Now, we treat \cite[Proposition 2.6]{FT1}. Case (i) in the proof of \cite[Proposition 2.6]{FT1} follows from \eqref{eq:endpoints}. For Case (ii), the bounds for $\widetilde s$ are unchanged. For $\widetilde d$,
\eqref{eq:rankzero} and \cite[equation (48)]{FT1} give
\[
 \begin{aligned}
 \widetilde d\ge{}&
 d-\frac{c-c'}2\left(\frac{n_0}{4}+|s_0|H^3\right)^2
 +\frac{n^2(c-c')}{2}\\
 &-\frac{n(c-c')^2}{2}
 +\frac{(c-c')^3}{8}-\mathsf C(c-c')\\
 \ge{}&
 d+\frac{15}{32}n_0^2(c-c')\\
 &-n_0(c-c')\left(\frac{|s_0|H^3}{4}
                         +\delta_n+\frac{c-c'}2\right)\\
 &-\frac{c-c'}2(s_0H^3)^2-\mathsf C(c-c').
 \end{aligned}
\]
Since $0\le\delta_n\le c/3$ and $0<c-c'\le c/3$, for
$n_0>2\mathsf C/c$, we have
\[n_0\left(\frac34|s_0|H^3+c-\delta_n-\frac{c-c'}2\right) +\frac12(s_0H^3)^2-\mathsf C \ge\frac c2n_0-\mathsf C>0.\]
It follows, using the lower bound for $d$ in
\cite[Definition 2.3]{FT1}, that
\[
 \begin{aligned}
 \widetilde d>{}&
 d+\frac{15}{32}n_0^2(c-c')\\
 &-n_0(c-c')(|s_0|H^3+c)-(c-c')(s_0H^3)^2\\
 \ge{}&
 \frac{15}{32}n_0^2\delta_{n'}
 -n_0\delta_{n'}(|s_0|H^3+c)
 -\delta_{n'}(s_0H^3)^2+d_0,
 \end{aligned}
\]
where $n'=n-c+c'$ and $\delta_{n'}=\delta_n+c-c'$.
This is the required bound for $\widetilde d$. Hence, \cite[Proposition 2.6]{FT1} holds.

Finally, \cite[Lemma 2.8]{FT1} uses only
the classical Bogomolov--Gieseker inequality, hence its proof works verbatim in our case. Thus, the wall description in \cite[Section 2]{FT1} holds. Therefore, the wall crossing arguments in \cite[Section 3--4]{FT1} express the rank $0$ invariants in terms of rank $1$ invariants. Combining this with the second statement proves the first.
\end{proof}

As mentioned in \cite{FT1,FT2}, the coefficients of $F_0$ and $F_1$ depend only on
$\rH^*(X,\QQ)$ as a graded ring with pairing, $\ch(\bv)$, the Chern classes of $X$, and the class $H\in \rH^2(X,\ZZ)$. Combining with
\cite{pardon}, we obtain the following consequence.

\begin{corollary}\label{cor:sym-inv}
Let $(X_1,H_1)$ and $(X_2,H_2)$ be polarized Calabi--Yau threefolds.
Choose K\"ahler forms $\omega_i$ with $[\omega_i]=H_i\in \rH^2(X_i,\RR)$ and a
symplectomorphism
\[\phi\colon(X_1,\omega_1)\xra{\simeq}(X_2,\omega_2).\]
Let $\bv_i\in\Knum(X_i)$ for $i=1,2$ that satisfy
$\ch(\bv_1)=\phi^*\ch(\bv_2)$. Then
\[\DT_{H_1}(\bv_1)=\DT_{H_2}(\bv_2).\]
\end{corollary}

\subsection{Castelnuovo Bound Conjecture}

Finally, we discuss consequences of Theorem~\ref{thm:main1.1} concerning Castelnuovo Bound Conjecture. One of its general form can be stated as follows.

\begin{conjecture}[{Castelnuovo bound for threefolds}]\label{conj:general-cast}
Let $(X,H)$ be a polarized smooth projective complex threefold. Assume that $H$ generates $\NS(X)_\mathrm{tf}$ and $s\in \ZZ_{\geq 1}$ is the least integer such that there exists a nonzero effective divisor numerically equivalent to $sH$. Then for any one-dimensional closed subscheme $C\subset X$ of degree $d$ and genus $g$, we have
\[g\leq \frac{1}{2sH^3}d^2+\frac{s+\frac{K_X.H^2}{H^3}}{2}d+1.\]
\end{conjecture}

Using Theorem~\ref{thm:main1.1}, we can prove Conjecture~\ref{conj:general-cast} up to an computable constant.

\begin{theorem}\label{thm:cast-bound}
Let $(X,H)$ be a polarized smooth projective threefold over an algebraically closed field of characteristic $0$. Assume that $H$ generates $\NS(X)_\mathrm{tf}$ and $s\in \ZZ_{\geq 1}$ is the least integer such that there exists a nonzero effective divisor numerically equivalent to $sH$. Let $\mathsf C\geq0$ be any constant for which \eqref{eq:quad} holds for $(X,H)$ and $(\alpha,B)\in \RR_{>0}\times \RR H$. Then for any one-dimensional closed subscheme $C\subset X$ of degree $d$ and genus $g$, we have
\[g\leq \frac{1}{2sH^3}d^2+\frac{s+\frac{K_X.H^2}{H^3}}{2}d+1+\max\left\{2\sqrt{2}s\mathsf C,\frac{3}{2}s^{\frac{1}{3}}\mathsf C^{\frac{4}{3}} \right\}H^3.\]
\end{theorem}

\begin{proof}
Set $A=sH$ and $c=\mathsf C/s^2$. Thus \eqref{eq:quad} holds with $H$ replaced by $A$ and
$\mathsf C$ replaced by $c$. By our assumption, every nonzero effective
divisor is numerically equivalent to $aA$ for some $a\geq 1$.
For a closed subscheme $Z\subset X$ of dimension at most one, we write
\[t_Z=\frac{A.Z}{A^3},\qquad h_Z=\frac{\ch_3(\cI_Z)}{A^3},\qquad f(t)=\frac{t^2+t}{2}.\]
Here $t_Z=0$ if $\dim Z\leq 0$, and we allow $Z=\varnothing$. Therefore, the inequality we want to prove is equivalent to
\begin{equation}\label{eq:normalized-bound}
 h_C-f(t_C)\leq
 \max\left\{2\sqrt{2}c,\frac32c^{4/3}\right\}
\end{equation}
for any one-dimensional closed subscheme $C\subset X$.

First, let $D\equiv aA$ be an integral divisor and $Z\subset D$ a closed subscheme of dimension at most one. The argument of \cite[Lemma~4.5]{liu-ruan} shows that $\cI_{Z/D}$ is $\nu^A_{\alpha,-\frac{t_Z}{a}A}$-semistable for all $\alpha>0$. Applying \eqref{eq:quad} and letting $\alpha\to0$ as in \cite[Proposition~4.6]{liu-ruan}, we get
\begin{equation}\label{eq:surface-bound}
 h_Z\leq\frac{t_Z^2}{2a}+\frac a2t_Z+ca.
\end{equation}

We next prove the auxiliary upper bound
\begin{equation}\label{eq:degree-bound}
 h_C-f(t_C)\leq c\sqrt{2t_C}
\end{equation}
for any one-dimensional closed subscheme $C\subset X$ by induction on $H.C$. By \cite[Lemma 4.3]{liu-ruan}, we may assume that $\cO_C$ is pure. Set $\beta_C\coloneqq \sqrt{\alpha^2_C+2t_C}$ and
\[\alpha_C\coloneqq\max\left\{\frac{\sqrt t_C}{2},\frac{t_C}{2}-1\right\}.\]
If $\cI_C$ is $\nu^A_{\alpha_C,-\beta_CA}$-semistable, applying Theorem~\ref{thm:main1.1} gives
\begin{equation}\label{eq:semistable-bound}
 h_C\leq\frac23 \beta_Ct_C+c\sqrt{2t_C}.
\end{equation}
Since
\[
 \beta_C=\begin{cases}
 \frac32\sqrt t_C,&t_C\leq4,\\
 \frac{t_C}{2}+1,&t_C\geq4,
 \end{cases}
 \qquad
 \frac23\beta_Ct_C\leq f(t_C),
\]
this proves \eqref{eq:degree-bound} in this case.

Now, suppose that $\cI_C$ is not $\nu^A_{\alpha_C,-\beta_CA}$-semistable. Since $\cI_C$ is $\nu^A_{\alpha,\beta A}$-semistable for $\alpha\gg0$ and any fixed $\beta<0$, let $W$ be the outermost actual wall enclosing $(\alpha_C,-\beta_C)$.

The radius of $W$ is greater than $\alpha_C\geq\sqrt{t_C}/2$. Since the right endpoint of a wall of $\cI_C$ increases with its radius, this endpoint is greater than
\[-\beta_C+\alpha_C=\max\{-\sqrt{t_C},-2\}.\]
By the estimate of \cite[Lemma~3.2]{liu-ruan} and \cite[Proposition~3.2]{liu}, there are an effective divisor $D\equiv aA$ and closed subschemes $C_1\subset C$ and $C_2\subset C\cap D$ such that the destabilizing subobject of $\cI_C$ along $W$ is $\cI_{C_1}(-D)$ and we have an exact sequence
\[0\longrightarrow\cO_{C_1}(-D)\longrightarrow\cO_C \longrightarrow\cO_{C_2}\longrightarrow0.\]
In particular, we have $t_C=t_{C_1}+t_{C_2}$ and $h_C=h_{C_1}+h_{C_2}+at_{C_1}$.

Note that for any point $(\alpha,\beta)\in W$ with $\beta>-2$, we have $0<A^2.\ch_1^{\beta A}(\cI_{C_1}(-D))<A^2.\ch_1^{\beta A}(\cI_C)$. Hence, $0<a<2$. Moreover, the minimality of $s$ gives $a\geq1$. Therefore, $D$ is integral, since each nonzero effective component of $D$ has numerical class at least $sH$ and $D\equiv asH$ with $a<2$.

The equation of $W$ is
\[{\alpha'}^2+ \left(\beta'+\frac{t_{C_2}}a+\frac a2\right)^2 =\left(\frac{t_{C_2}}a+\frac a2\right)^2-2t_C.\]
At the top of $W$, $A^2.\ch_1^{\beta A}(\cI_{C_1}(-D))>0$ gives $t_{C_2}>a^2/2$. Since $(\alpha_C,-\beta_C)$ is inside $W$ and $\beta_C^2=\alpha_C^2+2t_C$, we obtain
\begin{equation}\label{eq:wall-numerics}
 \frac{t_{C_2}}a+\frac a2>\beta_C,\qquad
 t_{C_2}>\frac{a^2}{2}
             +a\sqrt{\alpha_C^2+2t_{C_1}}.
\end{equation}
Indeed, the first inequality follows from the equation of $W$; the second follows by using $t_{C_2}>a^2/2$. As $\alpha_C^2\geq t_C/4$, the second inequality implies
\[(t_{C_2}-a^2)(4t_{C_2}-a^2)>9a^2t_{C_1}, \qquad t_{C_2}>a^2\geq a, \qquad \sqrt{2t_{C_1}}+a<\sqrt{2t_C}.\] 
In particular, we have $t_{C_1}<t_C$. Applying \eqref{eq:surface-bound} to $C_2\subset D$ now yields
\begin{equation}\label{eq:residual-bound}
 \begin{aligned}
 h_C-f(t_C)
 &\leq h_{C_1}-f(t_{C_1})+ca\\
 &\quad -(t_{C_2}-a)
     \left(t_{C_1}+\frac{a-1}{2a}t_{C_2}\right).
 \end{aligned}
\end{equation}
Thus by $t_{C_1}<t_C$ and induction, we obtain
\[h_C-f(t_C)\leq c\sqrt{2t_{C_1}}+ca \leq c\sqrt{2t_C}.\]
This proves \eqref{eq:degree-bound}.

We next prove \eqref{eq:normalized-bound}. When $t_C\leq4$, it follows from \eqref{eq:degree-bound}. Suppose $t_C\geq4$. If $\cI_C$ is $\nu^A_{\alpha_C,-\beta_CA}$-semistable, then by $\beta_C=t_C/2+1$ and \eqref{eq:semistable-bound}, we get
\[h_C-f(t_C) \leq-\frac{t_C^2}{6}+\frac{t_C}{6}+c\sqrt{2t_C} \leq-\frac{t_C^2}{8}+c\sqrt{2t_C} \leq\frac32c^{4/3}.\]
If $\cI_C$ is not $\nu^A_{\alpha_C,-\beta_CA}$-semistable, the first inequality in \eqref{eq:wall-numerics} gives
\[t_{C_2}>a+\frac{a}{2-a}t_{C_1}.\]
Consequently,
\[(t_{C_2}-a) \left(t_{C_1}+\frac{a-1}{2a}t_{C_2}\right) \geq\frac{a(3-a)}{2(2-a)^2}t_{C_1}^2 \geq\frac{t_{C_1}^2}{(2-a)^2}.\]
Using \eqref{eq:degree-bound} for $C_1$ in
\eqref{eq:residual-bound}, we find
\[
 \begin{aligned}
 h_C-f(t_C)
 &\leq ac+c\sqrt{2t_{C_1}}
                   -\frac{t_{C_1}^2}{(2-a)^2}\\
 &\leq ac+\frac34(2-a)^{2/3}c^{4/3}.
 \end{aligned}
\]
Maximizing the last expression over $1\leq a<2$ gives
\[
 h_C-f(t_C)\leq
 \begin{cases}
  2c+c^2/16,&0\leq c\leq8,\\
  c+\frac34c^{4/3},&c\geq8
 \end{cases}
 \leq\max\left\{2\sqrt2c,\frac32c^{4/3}\right\}.
\]
Thus \eqref{eq:normalized-bound} holds in all cases and the result follows from Hirzebruch--Riemann--Roch.
\end{proof}

In particular, using Remark~\ref{rmk:constant} and Theorem~\ref{thm:cast-bound}, we have:

\begin{corollary}\label{cor:cast-cy}
Let $X$ be a Calabi--Yau threefold of Picard number one and $H$ be the ample generator of $\NS(X)$. Then for any one-dimensional closed subscheme $C\subset X$ of degree $d$ and genus $g$, we have \[g\leq \frac{1}{2H^3}d^2+\frac{1}{2}d+1+288\sqrt[3]{12}m^{\frac{8}{3}}H^3,\] where $m\geq 1$ is an integer such that $mH$ is very ample.
\end{corollary}

\begin{proof}
By Miyaoka’s theorem, we have $s=1$. Then the result follows from Theorem~\ref{thm:cast-bound} by taking $\mathsf C=96m^2$ as in Remark~\ref{rmk:constant}.
\end{proof}

Note that by \cite{liu}, the quadratic term of Conjecture~\ref{conj:general-cast} is optimal. Combining with Theorem~\ref{thm:cast-bound} and the construction in \cite[Theorem 4.13]{liu}, we see that the linear term is optimal as well, which provides a refinement for \cite[Theorem 1.7]{liu}.

\begin{corollary}
Let $(X,H)$ be a polarized smooth projective threefold over an algebraically closed field of characteristic $0$. Assume that $H$ generates $\NS(X)_\mathrm{tf}$ and $s\in \ZZ_{\geq 1}$ is the least integer such that there exists a nonzero effective divisor numerically equivalent to $sH$. Then we have
\[\lim_{d\to \infty} \frac{g^{X}_{\max}(d)-\frac{1}{2sH^3}d^2}{d}=\frac{s+\frac{K_X.H^2}{H^3}}{2},\]
where
\[g^{X}_{\max}(d)\coloneqq \max\{g(C)\mid C\subset X \text{ is a } 1\text{-dimensional closed subscheme of degree }d\}.\]
\end{corollary}


\begin{thebibliography}{BMSZ17}



\bibitem[AFK$^+$24]{fey:modularity}
S.~Alexandrov, S.~Feyzbakhsh, A.~Klemm, B.~Pioline and T.~Schimannek,
\emph{Quantum geometry, stability and modularity},
Commun. Number Theory Phys. \textbf{18} (2024), no.~1, 49--151.

\bibitem[AB13]{AB}
D.~Arcara and A.~Bertram,
\emph{Bridgeland-stable moduli spaces for $K$-trivial surfaces},
with an appendix by M.~Lieblich,
J. Eur. Math. Soc. (JEMS) \textbf{15} (2013), no.~1, 1--38.


\bibitem[Bay19]{Bay:deformation}
A.~Bayer,
\emph{A short proof of the deformation property of Bridgeland
stability conditions},
Math. Ann. \textbf{375} (2019), no.~3--4, 1597--1613.


\bibitem[BBMT14]{bbmt14}
A.~Bayer, A.~Bertram, E.~Macr\`i and Y.~Toda,
\emph{Bridgeland stability conditions on threefolds II: An application to Fujita's conjecture},
J. Algebraic Geom. \textbf{23} (2014), no.~4, 693--710.

\bibitem[BLM$^+$21]{BLM} A.~Bayer, M.~Lahoz, E.~Macr\`i, H.~Nuer,
A.~Perry and P.~Stellari, \emph{Stability conditions in families},
Publ. Math. Inst. Hautes \'Etudes Sci. \textbf{133} (2021), 157--325.

\bibitem[BM14]{BM:projectivity}
A.~Bayer and E.~Macr\`i,
\emph{Projectivity and birational geometry of Bridgeland moduli spaces},
J. Amer. Math. Soc. \textbf{27} (2014), no.~3, 707--752.


\bibitem[BM23]{bm:icm}
A.~Bayer and E.~Macr\`i,
\emph{The unreasonable effectiveness of wall-crossing in algebraic geometry},
in \emph{ICM---International Congress of Mathematicians, Vol.~3, Sections 1--4},
EMS Press, Berlin, 2023, pp.~2172--2195.

\bibitem[BMS16]{BMS} A.~Bayer, E.~Macr\`i and P.~Stellari, \emph{The space of stability conditions on abelian threefolds, and on some Calabi--Yau threefolds}, Invent. Math. \textbf{206} (2016), no.~3, 869--933.

\bibitem[BMT14]{BMT} A.~Bayer, E.~Macr\`i and Y.~Toda,
\emph{Bridgeland stability conditions on threefolds I:
Bogomolov--Gieseker type inequalities}, J. Algebraic Geom.
\textbf{23} (2014), no.~1, 117--163.

\bibitem[BMSZ17]{BMSZ} M.~Bernardara, E.~Macr\`i, B.~Schmidt and X.~Zhao,
\emph{Bridgeland stability conditions on Fano threefolds},
\'Epijournal G\'eom. Alg\'ebrique \textbf{1} (2017), Art.~2, 24~pp.


\bibitem[Bri07]{bridgeland:stability} T.~Bridgeland, \emph{Stability conditions on triangulated categories}, Ann. of Math. (2) \textbf{166} (2007), no.~2, 317--345.

\bibitem[Bri08]{bridgeland:K3}
T.~Bridgeland, \emph{Stability conditions on $K3$ surfaces},
Duke Math. J. \textbf{141} (2008), no.~2, 241--291.

\bibitem[BM02]{Bridgeland–Maciocia}
T.~Bridgeland and A.~Maciocia,
\emph{Fourier--Mukai transforms for K3 and elliptic fibrations},
J. Algebraic Geom. \textbf{11} (2002), no.~4, 629--657.

\bibitem[CF26]{CF} Y.~Cheng and S.~Feyzbakhsh, \emph{Stability conditions on threefolds}, preprint (2026), \href{https://arxiv.org/abs/2607.04788v1}{arXiv:\allowbreak 2607.04788v1}.

\bibitem[GRS26]{GRS}
V.~Di Gennaro, A.~Rapagnetta and P.~Sabatino,
\emph{On the genus of a curve in a projective $3$-fold},
preprint (2026); \href{https://arxiv.org/abs/2601.17852}{arXiv:2601.17852}.

\bibitem[DG15]{Donkin-Geranios} S.~Donkin and H.~Geranios,
\emph{Composition factors of tensor products of truncated symmetric powers},
Adv. Math. \textbf{285} (2015), 394--433.


\bibitem[Dou01]{Douglas}
M.~R.~Douglas,
\emph{D-branes, categories and $N=1$ supersymmetry},
J. Math. Phys. \textbf{42} (2001), no.~7, 2818--2843.


\bibitem[DFR05]{DFR}
M.~R.~Douglas, B.~Fiol and C.~R\"omelsberger,
\emph{Stability and BPS branes},
J. High Energy Phys. (2005), no.~9, Paper No.~006.


\bibitem[FKLR25]{FKLR}
S.~Feyzbakhsh, N.~Koseki, Z.~Liu and N.~Rekuski,
\emph{Stability conditions on Calabi--Yau threefolds via Brill--Noether theory of curves},
preprint (2025), \href{https://arxiv.org/abs/2509.24990}{arXiv:2509.24990}.

\bibitem[FT23a]{feyz:curve-counting}
S.~Feyzbakhsh and R.~P.~Thomas,
\emph{Curve counting and S-duality},
\'Epijournal G\'eom. Alg\'ebrique \textbf{7} (2023), Art.~15, 25~pp.

\bibitem[FT23b]{FT1}
S.~Feyzbakhsh and R.~P.~Thomas,
\emph{Rank $r$ DT theory from rank $1$},
J. Amer. Math. Soc. \textbf{36} (2023), no.~3, 795--826.

\bibitem[FT24]{FT2}
S.~Feyzbakhsh and R.~P.~Thomas,
\emph{Rank $r$ DT theory from rank $0$},
Duke Math. J. \textbf{173} (2024), no.~11, 2063--2116.

\bibitem[FLZ22]{FLZ}
L.~Fu, C.~Li and X.~Zhao,
\emph{Stability manifolds of varieties with finite Albanese morphisms},
Trans. Amer. Math. Soc. \textbf{375} (2022), no.~8, 5669--5690.


\bibitem[Gie79]{Gieseker}
D.~Gieseker,
\emph{On a theorem of Bogomolov on Chern classes of stable bundles},
Amer. J. Math. \textbf{101} (1979), no.~1, 77--85.


\bibitem[HRS96]{RS:tilting}
D.~Happel, I.~Reiten and S.~O.~Smal{\o},
\emph{Tilting in abelian categories and quasitilted algebras},
Mem. Amer. Math. Soc. \textbf{120} (1996), no.~575, viii+88~pp.

\bibitem[HL10]{HL}
D.~Huybrechts and M.~Lehn,
\emph{The geometry of moduli spaces of sheaves}, second edition,
Cambridge Mathematical Library, Cambridge University Press, Cambridge, 2010, xviii+325~pp.

\bibitem[JS]{JS}
D.~Joyce and Y.~Song,
\emph{A theory of generalized Donaldson--Thomas invariants},
Mem. Amer. Math. Soc. \textbf{217} (2012), no.~1020, iv+199~pp.

\bibitem[Kos18]{koseki:product-threefolds}
N.~Koseki,
\emph{Stability conditions on product threefolds of projective
spaces and Abelian varieties},
Bull. Lond. Math. Soc. \textbf{50} (2018), no.~2, 229--244.


\bibitem[Kos20]{Naoki:Stabnefbundle}
N.~Koseki,
\emph{Stability conditions on threefolds with nef tangent bundles},
Adv. Math. \textbf{372} (2020), Paper No.~107316, 29~pp.


\bibitem[Kos22]{koseki:double-triple-solids}
N.~Koseki,
\emph{Stability conditions on Calabi--Yau double/triple solids},
Forum Math. Sigma \textbf{10} (2022), Paper No.~e63, 33~pp.


\bibitem[Kos23]{Kos} N.~Koseki, \emph{On the Bogomolov--Gieseker inequality
in positive characteristic}, Int. Math. Res. Not. IMRN
(2023), no.~24, 20784--20811.


\bibitem[Lan04]{Langer} A.~Langer, \emph{Semistable sheaves in positive characteristic},
Ann. of Math. (2) \textbf{159} (2004), no.~1, 251--276.


\bibitem[Lan04b]{Langer:mixed}
A.~Langer,
\emph{Moduli spaces of sheaves in mixed characteristic},
Duke Math. J. \textbf{124} (2004), no.~3, 571--586.


\bibitem[Li19a]{li:quintic}
C.~Li, \emph{On stability conditions for the quintic threefold},
Invent. Math. \textbf{218} (2019), no.~1, 301--340.

\bibitem[Li19b]{li:fano-pic-rk-1}
C.~Li, \emph{Stability conditions on Fano threefolds of Picard number $1$},
J. Eur. Math. Soc. (JEMS) \textbf{21} (2019), no.~3, 709--726.

\bibitem[Li25]{Li:real}
C.~Li,
\emph{A real reduction of the manifold of Bridgeland stability conditions},
preprint (2025);
\href{https://arxiv.org/abs/2506.21995}{arXiv:2506.21995}.

\bibitem[Li26]{Li} C.~Li,
\emph{A remark on stability conditions on smooth projective varieties},
preprint (2026), \href{https://arxiv.org/abs/2601.22994}{arXiv:2601.22994}.

\bibitem[LLL$^+$26]{LLL} C.~Li, Z.~Liu, Z.~Liu, E.~Macr\`i, A.~Perry, P.~Stellari and X.~Zhao, \emph{Stability conditions and moduli spaces on projective families}, preprint (2026), \href{https://arxiv.org/abs/2607.28411v1}{arXiv:2607.28411v1}.

\bibitem[LLSZ]{LLSZ}
C.~Li, Z.~Liu, B.~Sung and X.~Zhao,
\emph{Stability conditions and inequalities for higher Chern characters},
preprint.

\bibitem[LMPSZ25]{LMPSZ} C.~Li, E.~Macr\`i, A.~Perry,
P.~Stellari and X.~Zhao,
\emph{Stability conditions on products of curves and Hilbert schemes of surfaces},
Geom. Topol., to appear;
\href{https://arxiv.org/abs/2512.14207}{arXiv:2512.14207}.

\bibitem[Liu22]{liu:quadratic}
S.~Liu, \emph{Stability condition on Calabi--Yau threefold of complete intersection of quadratic and quartic hypersurfaces},
Forum Math. Sigma \textbf{10} (2022), Paper No.~e106, 29~pp.

\bibitem[Liu21]{Liu} Y.~Liu, \emph{Stability conditions on product varieties},
J. Reine Angew. Math. \textbf{770} (2021), 135--157.


\bibitem[Liu24]{liu}
Z.~Liu, \emph{Castelnuovo bound for curves in projective $3$-folds},
preprint (2024);
\href{https://arxiv.org/abs/2407.20161}{arXiv:2407.20161}.

\bibitem[LR22]{liu-ruan}
Z.~Liu and Y.~Ruan,
\emph{Castelnuovo bound and higher genus Gromov--Witten
invariants of quintic $3$-folds}, preprint (2022);
\href{https://arxiv.org/abs/2210.13411v2}{arXiv:2210.13411v2}.

\bibitem[LM26]{LM} Z.~Liu and T.~Mao,
\emph{Tilt-stability on singular schemes and Bogomolov--Gieseker-type
inequalities}, preprint (2026), \href{https://arxiv.org/abs/2605.13808v1}{arXiv:2605.13808v1}.

\bibitem[Mac14]{macri:projective-space}
E.~Macr\`i,
\emph{A generalized Bogomolov--Gieseker inequality for the
three-dimensional projective space},
Algebra Number Theory \textbf{8} (2014), no.~1, 173--190.


\bibitem[MP15]{maciocia-piyaratne:abelian-I}
A.~Maciocia and D.~Piyaratne,
\emph{Fourier--Mukai transforms and Bridgeland stability conditions
on abelian threefolds},
Algebr. Geom. \textbf{2} (2015), no.~3, 270--297.


\bibitem[MP16]{maciocia-piyaratne:abelian-II}
A.~Maciocia and D.~Piyaratne,
\emph{Fourier--Mukai transforms and Bridgeland stability conditions
on abelian threefolds II},
Internat. J. Math. \textbf{27} (2016), no.~1, Paper No.~1650007.


\bibitem[MS19]{martinez-schmidt:blowups}
C.~Martinez and B.~Schmidt,
\emph{Bridgeland stability on blow ups and counterexamples},
with an appendix by O.~Das,
Math. Z. \textbf{292} (2019), no.~3--4, 1495--1510.


\bibitem[OAI26]{OAI-BMT}
OpenAI,
\emph{Prescribed large-volume charges on threefolds with trivial canonical bundle},
preprint (2026);
\href{https://github.com/openai/math/blob/main/preprints/Prescribed-large-volume-charges-on-threefolds-with-trivial-canonical-bundle-September-24-2026/paper.pdf}{preprint}.

\bibitem[Par23]{pardon}
J.~Pardon, \emph{Universally counting curves in Calabi--Yau threefolds},
preprint (2023);
\href{https://arxiv.org/abs/2308.02948v3}{arXiv:2308.02948v3}.

\bibitem[PT19]{PT} D.~Piyaratne and Y.~Toda,
\emph{Moduli of Bridgeland semistable objects on $3$-folds and
Donaldson--Thomas invariants}, J. Reine Angew. Math.
\textbf{747} (2019), 175--219.

\bibitem[Pol07]{Pol}
A.~Polishchuk,
\emph{Constant families of $t$-structures on derived categories
of coherent sheaves},
Mosc. Math. J. \textbf{7} (2007), no.~1, 109--134.


\bibitem[Sch14]{schmidt:quadric}
B.~Schmidt,
\emph{A generalized Bogomolov--Gieseker inequality for the smooth
quadric threefold},
Bull. Lond. Math. Soc. \textbf{46} (2014), no.~5, 915--923.


\bibitem[Sch17]{Sch17}
B.~Schmidt, \emph{Counterexample to the generalized Bogomolov--Gieseker inequality for threefolds},
Int. Math. Res. Not. IMRN (2017), no.~8, 2562--2566.

\bibitem[SB98]{ShepherdBarron:Frobenius}
N.~I.~Shepherd-Barron,
\emph{Semi-stability and reduction mod $p$},
Topology \textbf{37} (1998), no.~3, 659--664.

\bibitem[Stacks]{Stacks} The Stacks Project Authors,
\emph{Stacks Project}, \url{https://stacks.math.columbia.edu}, 2026.


\bibitem[Sun19]{sun:tilt-vanishing}
H.~M.~Sun,
\emph{Tilt-stability, vanishing theorems and Bogomolov--Gieseker
type inequalities},
Adv. Math. \textbf{347} (2019), 677--707.

\bibitem[Sun21]{sun2021stabilityconditionsthreefoldsvanishing}
H.~M.~Sun,
\emph{Stability conditions on threefolds with vanishing Chern classes},
preprint (2020);
\href{https://arxiv.org/abs/2006.00756}{arXiv:2006.00756}.


\bibitem[Sun08]{Sun} X.~Sun, \emph{Direct images of bundles under
Frobenius morphism}, Invent. Math. \textbf{173} (2008), no.~2, 427--447.

\bibitem[Sun10]{Sun:Frobenius-survey}
X.~Sun,
\emph{Frobenius morphism and semi-stable bundles},
in \emph{Algebraic geometry in East Asia---Seoul 2008},
Adv. Stud. Pure Math., vol.~60,
Mathematical Society of Japan, Tokyo, 2010, pp.~161--182.


\bibitem[Th]{Th}
R.~P.~Thomas,
\emph{A holomorphic Casson invariant for Calabi--Yau $3$-folds,
and bundles on $K3$ fibrations},
J. Differential Geom. \textbf{54} (2000), no.~2, 367--438.

\bibitem[Tod09]{Toda:limit}
Y.~Toda,
\emph{Limit stable objects on Calabi--Yau $3$-folds},
Duke Math. J. \textbf{149} (2009), no.~1, 157--208.

\bibitem[Tod13]{toda:bogomolov-counting}
Y.~Toda, \emph{Bogomolov--Gieseker-type inequality and counting invariants},
J. Topol. \textbf{6} (2013), no.~1, 217--250.

\bibitem[Tod14]{Toda:CYBG}
Y.~Toda,
\emph{A note on Bogomolov--Gieseker type inequality for
Calabi--Yau $3$-folds},
Proc. Amer. Math. Soc. \textbf{142} (2014), no.~10, 3387--3394.

\bibitem[Tod17]{Toda:GepnerBG}
Y.~Toda,
\emph{Gepner point and strong Bogomolov--Gieseker inequality
for quintic $3$-folds},
in \emph{Higher Dimensional Algebraic Geometry---In Honour
of Professor Yujiro Kawamata's Sixtieth Birthday},
Adv. Stud. Pure Math., vol.~74,
Mathematical Society of Japan, Tokyo, 2017, pp.~381--405.

\bibitem[Xu25]{xu:quintic}
C.~Xu,
\emph{Stronger Bogomolov--Gieseker type inequality on quintic threefold},
preprint (2025);
\href{https://arxiv.org/abs/2511.21288}{arXiv:2511.21288}.
\end{thebibliography}
\end{document}